\documentclass[11pt]{amsart}

\usepackage[T1]{fontenc}
\usepackage[utf8]{inputenc}
\usepackage{lmodern}

\usepackage[american]{babel}
\usepackage[babel, final]{microtype}
\usepackage{amssymb}
\makeatletter
\DeclareFontFamily{U}{mathx}{\hyphenchar\font45}
\DeclareFontShape{U}{mathx}{m}{n}{
      <5> <6> <7> <8> <9> <10>
      <10.95> <12> <14.4> <17.28> <20.74> <24.88>
      mathx10
      }{}
\DeclareSymbolFont{mathx}{U}{mathx}{m}{n}
\DeclareFontSubstitution{U}{mathx}{m}{n}

\let\widecheck\@undefined
\let\widebar\@undefined
\DeclareMathAccent{\widecheck}{\mathord}{mathx}{"71}
\DeclareMathAccent{\widebar}{\mathord}{mathx}{"73}
\makeatother

\input{mnsymbol}
\makeatletter
\newcommand{\supsize}{%
  \expandafter\ifx\csname S@\f@size\endcsname\relax
  \calculate@math@sizes
  \fi
  \csname S@\f@size\endcsname
  \fontsize\sf@size\z@\selectfont
}

\DeclareRobustCommand{\tsup}[1]{%
  \leavevmode\raise.9ex\hbox{\supsize #1}%
}
\DeclareTextSymbolDefault{\textprimechar}{OMS}
\DeclareTextSymbol{\textprimechar}{OMS}{48}

\makeatother

\makeatletter
\DeclareSymbolFont{stmry}{U}{stmry}{m}{n}
\SetSymbolFont{stmry}{bold}{U}{stmry}{b}{n}

\let\llbracket\@undefined
\let\rrbracket\@undefined
\DeclareMathDelimiter{\llbracket}{\mathopen}%
                     {stmry}{"4A}{stmry}{"71}
\DeclareMathDelimiter{\rrbracket}{\mathclose}%
                     {stmry}{"4B}{stmry}{"79}
\makeatother

\usepackage[all]{xy}

\usepackage{ifpdf}
\usepackage{comment}
\usepackage{multirow}
\usepackage[shortlabels]{enumitem}

\ifpdf
  \usepackage[pdftex, dvipsnames]{xcolor}
  \usepackage[pdftex, final]{graphicx}
  \usepackage{caption}
  \usepackage{subcaption}
  \usepackage{pinlabel}
  \usepackage[pdftex,%
    letterpaper,%
    includehead,%
    includefoot,%
    nomarginpar,%
    lmargin=1in,%
    rmargin=1in,%
    tmargin=1in,%
    bmargin=1in,%
  ]{geometry}
  \usepackage[pdftex,%
    final,%
    colorlinks=true,%
    linkcolor=NavyBlue,%
    citecolor=NavyBlue,%
    filecolor=NavyBlue,%
    menucolor=NavyBlue,%
    urlcolor=NavyBlue,%
    bookmarks=true,%
    bookmarksdepth=3,%
    bookmarksnumbered=true,%
    bookmarksopen=true,%
    bookmarksopenlevel=2,%
  ]{hyperref}
  \hypersetup{
    pdftitle={Ribbon homology cobordisms},
    pdfauthor={Aliakbar Daemi, Tye Lidman, David Shea Vela-Vick, C.-M. Michael 
      Wong},
    pdfsubject={Ribbon homology cobordisms},
    pdfkeywords={Homology cobordism, ribbon concordance, Heegaard Floer 
      homology, instanton Floer homology, character variety, Thurston geometry, 
      Dehn surgery}
  }
\else
  \usepackage[dvips, dvipsnames]{xcolor}
  \usepackage[dvips, final]{graphicx}
  \usepackage{caption}
  \usepackage{subcaption}
  \usepackage{pinlabel}
  \usepackage[dvips,%
    letterpaper,%
    includehead,%
    includefoot,%
    nomarginpar,%
    lmargin=1in,%
    rmargin=1in,%
    tmargin=1in,%
    bmargin=1in,%
  ]{geometry}
  \usepackage[ps2pdf,%
    final,%
    colorlinks=true,%
    linkcolor=NavyBlue,%
    citecolor=NavyBlue,%
    filecolor=NavyBlue,%
    menucolor=NavyBlue,%
    urlcolor=NavyBlue,%
    bookmarks=true,%
    bookmarksdepth=3,%
    bookmarksnumbered=true,%
    bookmarksopen=true,%
    bookmarksopenlevel=2,%
  ]{hyperref}
  \hypersetup{
    pdftitle={Ribbon homology cobordisms},
    pdfauthor={Aliakbar Daemi, Tye Lidman, David Shea Vela-Vick, C.-M. Michael 
      Wong},
    pdfsubject={Ribbon homology cobordisms},
    pdfkeywords={Homology cobordism, ribbon concordance, Heegaard Floer 
      homology, instanton Floer homology, character variety, Thurston geometry, 
      Dehn surgery}
  }
\fi

\usepackage{tikz}
\usetikzlibrary{tikzmark}

\input{macros}

\newcommand{\set}[1]{\mathchoice%
  {\left\lbrace #1 \right\rbrace}%
  {\lbrace #1 \rbrace}%
  {\lbrace #1 \rbrace}%
  {\lbrace #1 \rbrace}%
}

\newcommand{\paren}[1]{\mathchoice%
  {\left( #1 \right)}%
  {( #1 )}%
  {( #1 )}%
  {( #1 )}%
}
\newcommand{\brac}[1]{\mathchoice%
  {\left[ #1 \right]}%
  {[ #1 ]}%
  {[ #1 ]}%
  {[ #1 ]}%
}
\newcommand{\abs}[1]{\mathchoice%
  {\left\lvert #1 \right\rvert}%
  {\lvert #1 \rvert}%
  {\lvert #1 \rvert}%
  {\lvert #1 \rvert}%
}

\newcommand{\dirsum}{\oplus}                 
\newcommand{\bigdirsum}{\bigoplus}           
\newcommand{\tensor}{\otimes}                
\newcommand{\bigtensor}{\bigotimes}          
\newcommand{\union}{\cup}                    
\newcommand{\intersect}{\cap}                

\newcommand{\comp}{\circ}                    
\newcommand{\cross}{\times}                  

\newcommand{\isom}{\cong}                    

\newcommand{\numset}[1]{\mathbb{#1}}

\newcommand{\Z}{\numset{Z}}

\newcommand{\Q}{\numset{Q}}
\newcommand{\R}{\numset{R}}
\newcommand{\C}{\numset{C}}
\newcommand{\F}[1]{\numset{F}_{#1}}

\DeclareMathOperator{\id}{id}
\DeclareMathOperator{\Id}{Id}

\DeclareMathOperator{\im}{Im} 

\DeclareMathOperator{\Hom}{Hom} 

\DeclareMathOperator{\rk}{rk} 

\newcommand{\matrixgrp}[1]{\mathrm{#1}}
\newcommand{\GL}{\matrixgrp{GL}}
\newcommand{\SL}{\matrixgrp{SL}}
\newcommand{\SO}{\matrixgrp{SO}}
\newcommand{\SU}{\matrixgrp{SU}}

\newcommand{\fg}{\mathfrak g}

\newcommand{\bdy}{\partial} 
\newcommand{\connsum}{\mathbin{\#}} 
\newcommand{\bigconnsum}{\mathop{\mathchoice%
    {\vcenter{\hbox{\LARGE $\#$}}}
    {\vcenter{\hbox{\large $\#$}}}
    {\vcenter{\hbox{\footnotesize $\#$}}}
    {\vcenter{\hbox{\scriptsize $\#$}}}
  }}

\DeclareMathOperator{\Sym}{Sym}

\newcommand{\nbhd}[1]{\nu (#1)}

\DeclareMathOperator{\CF}{CF}
\DeclareMathOperator{\HF}{HF}

\DeclareMathOperator{\HFK}{HFK}

\newcommand{\opgen}[1]{#1^{\bullet}} 
\newcommand{\opminus}[1]{#1^-}
\newcommand{\opplus}[1]{#1^+}
\newcommand{\opinfty}[1]{#1^\infty}
\newcommand{\ophat}[1]{\widehat{#1}}
\newcommand{\optilde}[1]{\widetilde{#1}}

\newcommand{\opbar}[1]{\overline{#1}}

\newcommand{\CFh}{\ophat{\CF}}

\newcommand{\HFg}{\opgen{\HF}}
\newcommand{\HFm}{\opminus{\HF}}
\newcommand{\HFp}{\opplus{\HF}}
\newcommand{\HFi}{\opinfty{\HF}}
\newcommand{\HFh}{\ophat{\HF}}

\newcommand{\HFKh}{\ophat{\HFK}}

\newcommand{\curves}[1]{\boldsymbol{#1}}
\newcommand{\alphas}[1][]{%
  \ifthenelse{\equal{#1}{}}{\curves{\alpha}}{\curves{\alpha^{#1}}}
}
\newcommand{\betas}[1][]{%
  \ifthenelse{\equal{#1}{}}{\curves{\beta}}{\curves{\beta_{#1}}}
}

\newcommand{\genset}[1]{\mathfrak{#1}}

\DeclareMathOperator{\spin}{Spin}
\newcommand{\SpinC}{\spin^c}

\newcommand{\spinc}{\mathfrak{s}}

\newcommand{\spinct}{\mathfrak{t}}

\DeclareMathOperator{\ind}{ind}

\DeclareMathOperator{\SFH}{SFH}

\newcommand*{\disjunion}{\sqcup}
\newcommand*{\bigdisjunion}{\coprod}

\newcommand*{\fprod}{\mathbin{*}}

\renewcommand*{\Id}{\mathbb{I}}

\renewcommand*{\F}{\mathbb{F}}
\newcommand*{\zeetwo}{\Z / 2}

\DeclareMathOperator{\End}{End}

\newcommand*{\extprod}{\Lambda}
\newcommand*{\topwedge}{\Lambda^{\mathrm{top}}}

\DeclareMathOperator{\Tors}{Tors}

\newcommand*{\homeo}{\cong}

\DeclareMathOperator{\Int}{Int}

\renewcommand*{\opgen}[1]{#1^{\mathord{\circ}}}

\newcommand*{\Fh}{\ophat{F}}
\newcommand*{\Fm}{\opminus{F}}

\newcommand*{\Fg}{\opgen{F}}

\newcommand*{\double}[1]{D (#1)}

\newcommand*{\Yt}{\widetilde{Y}}
\newcommand*{\Mh}{\widehat{M}}
\newcommand*{\Rh}{\widehat{R}}
\newcommand*{\Nh}{\widehat{N}}

\newcommand*{\orrev}[1]{(-#1)}

\newcommand*{\nbhdgamma}{\nbhd{\gamma}}
\newcommand*{\nbhdgammai}{\nbhd{\gamma_i}}

\newcommand*{\pts}{\mathbf{p}}

\renewcommand*{\graph}{\Gamma}
\newcommand*{\graphXp}{\graph_{X'}}
\newcommand*{\graphP}{\graph_P}
\newcommand*{\graphQi}{\graph_{Q_i}}
\newcommand*{\graphRi}{\graph_{R_i}}
\newcommand*{\graphZp}{\graph_{Z'}}

\newcommand*{\Fc}{\mathcal{F}}

\newcommand*{\bigparen}[2]{%
  \left( \raisebox{9pt}{$\displaystyle#1$} \; \, #2 \right)
}
\newcommand*{\tallsum}[1]{%
  \raisebox{9pt}{$\displaystyle#1$} \; \, %
}

\newcommand*{\spincm}{\spinc_-}
\newcommand*{\spincp}{\spinc_+}

\newcommand*{\Ym}{Y_-}
\newcommand*{\Yp}{Y_+}
\newcommand*{\Ypm}{Y_{\pm}}

\newcommand*{\Mm}{M_-}
\newcommand*{\Mp}{M_+}

\newcommand*{\sutm}{\sut_-}
\newcommand*{\sutp}{\sut_+}

\newcommand*{\Msutm}{(\Mm, \sutm)}
\newcommand*{\Msutp}{(\Mp, \sutp)}

\newcommand*{\Km}{K_-}
\newcommand*{\Kp}{K_+}
\newcommand*{\Kpm}{K_{\pm}}

\newcommand*{\wm}{w_-}
\newcommand*{\wpl}{w_+}

\newcommand*{\rhom}{\rho_-}
\newcommand*{\rhop}{\rho_+}
\newcommand*{\rhoW}{\rho_W}

\DeclareMathOperator{\Tor}{Tor}

\DeclareMathOperator{\Lef}{Lef}

\newcommand*{\casson}{\lambda}
\newcommand*{\lFO}{\casson_{\mathrm{FO}}}
\newcommand*{\lSW}{\casson_{\mathrm{SW}}}

\newcommand*{\ZZ}{\Z[\Z]}

\renewcommand*{\connsum}{\mathbin{\sharp}} 
\renewcommand*{\bigconnsum}{\mathop{\sharp}}

\newcommand*{\eval}[2]{\mathchoice%
  {\left\langle #1, #2 \right\rangle}%
  {\langle #1, #2 \rangle}%
  {\langle #1, #2 \rangle}%
  {\langle #1, #2 \rangle}%
}

\newcommand*{\moduli}{\mathcal{M}}

\newcommand*{\opred}[1]{#1^{\mathrm{red}}}
\newcommand*{\opsharp}[1]{#1^{\sharp}}
\newcommand*{\opnatural}[1]{#1^{\natural}}
\newcommand*{\opfrom}[1]{\widehat{#1}}
\newcommand*{\opto}[1]{\widecheck{#1}}

\DeclareMathOperator{\HFI}{HFI}
\DeclareMathOperator{\HM}{HM}
\DeclareMathOperator{\IFH}{I}
\DeclareMathOperator{\IFC}{C}
\DeclareMathOperator{\SHI}{SHI}
\DeclareMathOperator{\KHI}{KHI}

\DeclareMathOperator{\IFCeq}{C}
\newcommand*{\deq}{d}

\newcommand*{\HFIh}{\widehat{\HFI}}
\newcommand*{\HFr}{\opred{\HF}}

\newcommand*{\HMg}{\opgen{\HM}}
\newcommand*{\HMf}{\opfrom{\HM}}
\newcommand*{\HMt}{\opto{\HM}}
\newcommand*{\HMb}{\opbar{\HM}}
\newcommand*{\HMr}{\opred{\HM}}

\newcommand*{\Io}{\IFH}
\newcommand*{\Is}{\opsharp{\IFH}}
\newcommand*{\In}{\opnatural{\IFH}}
\newcommand*{\If}{\opfrom{\IFH}}
\newcommand*{\It}{\opto{\IFH}}
\newcommand*{\Ib}{\opbar{\IFH}}
\newcommand*{\Ig}{\opgen{\IFH}}

\newcommand*{\Co}{\IFC}

\newcommand*{\Cf}{\opfrom{\IFCeq}}
\newcommand*{\Ct}{\opto{\IFCeq}}
\newcommand*{\Cb}{\opbar{\IFCeq}}

\newcommand*{\df}{\opfrom{\deq}}
\newcommand*{\dt}{\opto{\deq}}
\newcommand*{\db}{\opbar{\deq}}

\newcommand{\fC}{\genset{S}}

\newcommand*{\Cto}{\optilde{\Co}}
\newcommand*{\dto}{\optilde{d}}

\newcommand*{\Ys}{\opsharp{Y}}
\newcommand*{\ws}{\opsharp{w}}
\newcommand*{\Ws}{\opsharp{W}}

\newcommand*{\cSO}{\mathcal{SO}}

\newcommand*{\lambdat}{\optilde{\lambda}}

\newcommand*{\Kt}{\optilde{K}}

\newcommand{\A}{\mathbb{A}}
\newcommand{\CS}{\mathrm{CS}}
\newcommand*{\CSt}{\widetilde{\CS}}
\newcommand{\charvar}{\mathcal{X}}
\newcommand{\charvarG}{\charvar_G}
\newcommand*{\charvarSUtwo}{\charvar_{\SUtwo}}
\newcommand{\repvar}{\mathcal{R}}
\newcommand*{\repvarG}{\repvar_G}

\newcommand*{\SUtwo}{\SU (2)}

\newcommand*{\SOthree}{\SO (3)}

\newcommand*{\connsp}[1]{\mathcal{#1}}
\newcommand*{\connspA}{\connsp{A}}
\newcommand*{\connspB}{\connsp{B}}
\newcommand*{\energy}{\mathcal{E}}
\DeclareMathOperator{\tr}{tr}
\newcommand*{\asdop}{\mathcal{D}}
\newcommand*{\asdopa}{\asdop_A}
\newcommand*{\asdopabar}{\asdop_{\Abar}}

\newcommand*{\closure}[1]{\overline{#1}}

\newcommand*{\Abar}{\closure{A}}
\newcommand*{\DWbar}{\closure{\double{W}}}

\DeclareMathOperator{\Ad}{Ad}
\DeclareMathOperator{\ad}{ad}

\newcommand*{\prop}{\mathcal{P}}

\newcommand*{\PHS}{\Sigma (2, 3, 5)}
\newcommand*{\RPthree}{\R \mathrm{P}^3}

\newcommand*{\Nil}{\mathrm{Nil}}
\newcommand*{\Sol}{\mathrm{Sol}}
\newcommand*{\SLtwot}{\widetilde{\SL (2, \R)}}
\newcommand*{\HtwoR}{\mathbb{H}^2 \cross \R}

\newcommand*{\tprod}{\mathbin{\widetilde{\times}}}

\newcommand*{\polyh}[1]{\mathbf{#1}}
\newcommand*{\icosah}{\polyh{I}}

\newcommand*{\dih}{\polyh{D}}

\newcommand*{\sut}{\eta}

\newcommand*{\trivconn}{\Theta}

\newcommand*{\Dform}{\Omega}

\newcommand*{\coxeter}{\check{h}}

\newcommand*{\curvature}[1]{F (#1)}
\newcommand*{\curvA}{\curvature{A}}

\newcommand*{\gauge}{\mathcal{G}}

\newcommand*{\Gt}{\widetilde{G}}

\newcommand*{\dimR}{\dim_{\R}}

\newcommand*{\grpgen}[1]{\mathchoice%
  {\left\langle #1 \right\rangle}%
  {\langle #1 \rangle}%
  {\langle #1 \rangle}%
  {\langle #1 \rangle}%
}
\newcommand*{\grppre}[2]{\mathchoice%
  {\left\langle #1 \, \middle\vert \, #2 \right\rangle}%
  {\langle #1 \, \vert \, #2 \rangle}%
  {\langle #1 \, \vert \, #2 \rangle}%
  {\langle #1 \, \vert \, #2 \rangle}%
}
\newcommand*{\normal}[1]{\mathchoice%
  {\left\llangle #1 \right\rrangle}%
  {\llangle #1 \rrangle}%
  {\llangle #1 \rrangle}%
  {\llangle #1 \rrangle}%
}
\newcommand*{\ideal}[1]{\mathchoice%
  {\left\langle #1 \right\rangle}%
  {\langle #1 \rangle}%
  {\langle #1 \rangle}%
  {\langle #1 \rangle}%
}

\newcommand*{\blank}{\scalebox{0.75}[1.0]{$\mathord{-}$}} 

\newcommand*{\cover}[1]{\widetilde{#1}}

\DeclareMathOperator{\imlc}{im}
\renewcommand*{\im}{\imlc}

\DeclareMathOperator{\Arg}{Arg}

\newcommand*{\Un}{\matrixgrp{U} (n)}

\newcommand*{\Qlocpoly}{\Q \llbracket x^{-1}, x]}

\newcommand*{\homor}{\mathfrak{o}}

\newcommand*{\Uone}{\matrixgrp{U} (1)}

\DeclareMathOperator{\diag}{diag}

\newcommand*{\Ymp}{Y_{\mp}}

\newcommand*{\card}[1]{\abs{#1}}
\newcommand*{\subgrp}{\leq}

\newcommand*{\ccover}[1]{\widehat{#1}}

\newcommand*{\knotclos}{T^3}

\title[Ribbon homology cobordisms]{Ribbon homology cobordisms}

\author[Aliakbar Daemi]{Aliakbar Daemi}
\address{Department of Mathematics \\ Washington University in St.\ Louis \\ 
  St.\ Louis, MO 63130}
\email{\href{mailto:adaemi@wustl.edu}{adaemi@wustl.edu}}
\urladdr{\url{https://www.math.wustl.edu/~adaemi/}}

\author[Tye Lidman]{Tye Lidman}
\address{Department of Mathematics \\ North Carolina State University \\ 
  Raleigh, NC 27695}
\email{\href{mailto:tlid@math.ncsu.edu}{tlid@math.ncsu.edu}}
\urladdr{\url{https://sites.google.com/ncsu.edu/tlid/}}

\author[David Shea Vela-Vick]{David Shea Vela-Vick}
\address{Department of Mathematics \\ Louisiana State University \\ Baton 
  Rouge, LA 70803}
\email{\href{mailto:shea@math.lsu.edu}{shea@math.lsu.edu}}
\urladdr{\url{https://www.math.lsu.edu/~shea/}}

\author[C.-M. Michael Wong]{C.-M. Michael Wong}
\address{Department of Mathematics \\ Dartmouth College \\ Hanover, NH 03755}
\email{\href{mailto:wong@math.dartmouth.edu}{wong@math.dartmouth.edu}}
\urladdr{\url{https://math.dartmouth.edu/~wong/}}

\begin{document}

\begin{abstract}
  We study $4$-dimensional homology cobordisms without $3$-handles, showing 
  that they interact nicely with Thurston geometries, character varieties, and 
  instanton and Heegaard Floer homologies.
  Using these, we derive obstructions to such cobordisms. As one example of 
  these obstructions, we generalize other recent results on the behavior of 
  knot Floer homology under ribbon concordances.
  Finally, we provide topological applications, including to Dehn surgery 
  problems.
\end{abstract}

\maketitle

\section{Introduction}
\label{sec:intro}

The advent of topological quantum field theories (TQFTs) in the past few decades 
has renewed interest in smooth cobordisms and the associated category. In 
dimension $4$, Seiberg--Witten Floer homology, a gauge-theoretic TQFT, was 
recently used by Manolescu \cite{Man:Triangulation} to study the homology 
cobordism group $\Theta_{\Z}^3$, disproving the Triangulation Conjecture in 
dimensions $n \geq 5$. Many questions remain: For example, it is unknown 
whether $\Theta_{\Z}^3$ has any torsion.

In this article, we study $4$-dimensional cobordisms from a new perspective, in 
terms of their directionality. More precisely, we study \emph{ribbon 
  cobordisms}, which are $2n$-dimensional manifolds that can be built from 
$k$-handles with $k \leq n$.  They arise in at least two natural ways: as Stein 
cobordisms between closed, contact manifolds, and as the exterior of (strongly 
homotopy-) ribbon surfaces, which are cobordisms between link exteriors 
\cite{Gor:Rib}.  Note that every (homology) cobordism can be split into two 
ribbon (homology) cobordisms. While homology cobordism is an equivalence 
relation, ribbon homology cobordisms are not symmetric. In fact, as we will 
see, ribbon homology cobordisms give rise to a preorder on $3$-manifolds that 
seems to agree with orderings by various invariants and Thurston geometries. We conjecture:

\begin{conjecture}[cf.~{\cite[Conjecture~1.1]{Gor:Rib}}]
  \label{conj:order}
  The preorder on the set of homeomorphism classes of closed, connected, 
  oriented $3$-manifolds given by ribbon $\Q$-homology cobordisms is
  a partial order.
\end{conjecture}

Here, an \emph{$R$-homology cobordism} between two compact, oriented 
$3$-manifolds $Y_1$ and $Y_2$ is an oriented, smooth cobordism $W \colon Y_1 
\to Y_2$ such that
$H_* (W, Y_i; R) = 0$ for $i \in \set{1, 2}$.
For example, the exterior of a knot concordance
is a $\Z$-homology cobordism.  Our results can be summarized as:

\begin{metatheorem}
  \label{metathm:main}
  Let $\Ym$ and $\Yp$ be compact, connected, oriented $3$-manifolds possibly 
  with boundary, and suppose that there exists a ribbon homology cobordism from 
  $\Ym$ to $\Yp$. Then the complexity of $\Ym$ is no greater than that of 
  $\Yp$, as measured by each of the following invariants:
  \begin{enumerate}[label=(\Alph*), ref=\Alph*]
    \item \label{it:meta-pi1} The fundamental group;
    \item \label{it:meta-char} The $G$-character variety, for a compact, 
      connected Lie group $G$, and its Zariski tangent space at a conjugacy 
      class;
    \item \label{it:meta-floer} Various flavors of instanton and Heegaard Floer 
      homologies.
  \end{enumerate}
  These comparisons are sometimes realized by explicit morphisms in the 
  appropriate category.
\end{metatheorem}

Note that \eqref{it:meta-pi1} was proved by Gordon \cite{Gor:Rib} in the case 
that $\Ym$ and $\Yp$ have toroidal boundary, and his proof immediately 
generalizes to the closed case following Geometrization;
see \fullref{ssec:intro_context} for context.
We will provide more precise statements for \eqref{it:meta-char} and 
\eqref{it:meta-floer} in \fullref{ssec:intro_character} and 
\fullref{ssec:intro_floer}.

Our metatheorem has many topological applications, which we will discuss below.  
But first, \eqref{it:meta-pi1} and Geometrization together give us the 
following:

\begin{theorem}
  \label{thm:geom-intro}
  There is a hierarchy among the Thurston geometries with respect to ribbon 
  $\Q$-homology cobordisms, given by the diagram
  \[
    S^3 \to (S^2 \cross \R) \to \R^3 \to \Nil \to \Sol \to (\HtwoR) \union 
    \SLtwot \union \mathbb{H}^3.
  \]
  In other words, suppose that $\Ym$ and $\Yp$ are compact $3$-manifolds with 
  empty or toroidal boundary that admit distinct geometries, and that there 
  exists a ribbon $\Q$-homology cobordism from $\Ym$ to $\Yp$. Then there is a 
  sequence of arrows from the geometry of $\Ym$ to that of $\Yp$ in the diagram 
  above.
\end{theorem}

A more refined version of \fullref{thm:geom-intro} is stated in 
\fullref{thm:geom-hierarchy}.

\begin{remark}
It is natural to ask how the ribbon $\Q$-homology cobordism preorder interacts with
the JSJ decomposition.  Unfortunately,
there exist examples where $\Ypm$ is hyperbolic and $\Ymp$ has non-trivial JSJ 
decomposition; see \fullref{rmk:seifert} for more details.
\end{remark}

Evidence for \fullref{conj:order} is provided by the metatheorem and 
\fullref{thm:geom-intro}, as well as \fullref{cor:chern-simons}, and 
\fullref{cor:norm} below. \fullref{conj:order} is analogous to 
\cite[Conjecture~1.1]{Gor:Rib}, which states that the preorder on the set of 
knots in $S^3$ by ribbon concordance is a partial order.

\begin{remark}
  \label{rmk:slice-ribbon}
  Another major open problem regarding ribbon concordance is the Slice--Ribbon 
  Conjecture. In a similar spirit, a natural question to ask is whether a 
  $\Z$-homology sphere bounding a $\Z$-homology $4$-ball always bounds a 
  $\Z$-homology $4$-ball without $3$-handles.
\end{remark}

We now turn to some new applications.  We begin with an application to 
Seifert fibered homology spheres that illustrates the use of several different 
tools described above.

\begin{theorem}
  \label{thm:seifert}
  Suppose that $\Ym$ and $\Yp$ are the Seifert fibered homology spheres $\Sigma 
  (a_1, \dotsc, a_n)$ and $\Sigma (a_1', \dotsc, a_m')$ respectively, and that 
  there exists a ribbon $\Q$-homology cobordism from $\Ym$ to $\Yp$.  Then
  \begin{enumerate}
    \item\label{it:casson} The Casson invariants of $\Ym$ and $\Yp$ satisfy 
      $\abs{\lambda (\Ym)} \leq \abs{\lambda (\Yp)}$;
    \item\label{it:plumbing} Either $\Ym$ and $\Yp$ both bound 
      negative-definite plumbings, or both bound positive-definite plumbings;  
      and
    \item\label{it:fibers} The numbers of exceptional fibers satisfy $n \leq 
      m$.
  \end{enumerate}
\end{theorem}

The first two items above can be proved with either instanton or Heegaard Floer 
homology using Metatheorem~\eqref{it:meta-floer}.  However, the authors do not 
know of a Floer-homology proof of \eqref{it:fibers}.

Next, we have applications to ribbon concordance. Recall that a strongly 
homotopy-ribbon concordance is a knot concordance in $S^3 \cross I$ whose 
exterior is ribbon.\footnote{All ribbon concordances are strongly 
  homotopy-ribbon.}

\begin{corollary}
  \label{cor:montesinos}
  Suppose that $\Km$ and $\Kp$ are Montesinos knots in $S^3$ of determinant 
  $1$, and that the number of rational tangles in $\Km$ with denominator at 
  most $2$ is greater than that of $\Kp$. Then there does not exist a strongly 
  homotopy-ribbon concordance from $\Km$ to $\Kp$.
\end{corollary}

Recall that a knot in $S^3$ is \emph{small} if there are no closed, 
non--boundary-parallel, incompressible surfaces in its exterior, and that torus 
knots are small.

\begin{corollary}
  \label{cor:composite-to-small}
  Suppose that $\Km$ is a composite knot in $S^3$, and that $\Kp$ is a small 
  knot in $S^3$.  Then there does not exist a strongly homotopy-ribbon 
  concordance from $\Km$ to $\Kp$.
\end{corollary}

We also obtain applications to reducible Dehn surgery problems. The following 
is a sample theorem; see \fullref{sec:surgery} for its proof, as well as 
another similar result. The same techniques can be used to obtain other results 
along the same lines, which we do not pursue in this article.

\begin{theorem}
  \label{thm:surgery-general}
  Suppose that $Y$ is an irreducible $\Q$-homology sphere, $L$ is a 
  null-homotopic link in $Y$ of $\ell$ components, and $Y_0 (L) \homeo N 
  \connsum \ell (S^1 \cross S^2)$, where $Y_0 (L)$ denotes the result of 
  $0$-surgery along each component of $L$.  Then $N$ is orientation-preserving 
  homeomorphic to $Y$.
\end{theorem}

\begin{remark}
  \label{rmk:surgery-general-triv}
  Since the first appearance of this article, Hom and the second author 
  \cite[Corollary~1.2]{HomLidman} have used \fullref{thm:surgery-general} to 
  show that when $L$ is a knot, $L$ must in fact be trivial.
\end{remark}

\begin{remark}
  \label{rmk:surgery-L-space}
  The technique used to prove \fullref{thm:surgery-general} can also be applied 
  to the case where $Y$ is an $L$-space $\Z$-homology sphere and $L$ is a 
  non-trivial knot.  In this case, one could show that $Y_0 (L)$ does not 
  contain an $S^1 \cross S^2$ summand, which follows from Ni 
  \cite[p.~1144]{Ni-taut}.
\end{remark}

Finally, we also obtain an application to the computation of the Furuta--Ohta 
invariant $\lFO$ for $\ZZ$-homology $S^1 \cross S^3$'s \cite{FurutaOhta}. 

\begin{corollary}
  \label{cor:casson}
  Suppose that $\Ym$ and $\Yp$ are $\Z$-homology spheres, and that $W \colon 
  \Ym \to \Yp$ is a ribbon $\Q$-homology cobordism, and that $\DWbar$ is the 
  $\ZZ$-homology $S^1 \cross S^3$ obtained by gluing the ends of $\double{W}$ 
  by the identity.  Then $\lFO (\DWbar) = \casson (\Ym)$.  In particular, $\lFO 
  (\DWbar)$ agrees with the Rokhlin invariant of $\Ym$ mod $2$.
\end{corollary}

\begin{remark}
  \label{rmk:sw-intro}
  We believe that the proof of \fullref{cor:casson} can be adapted to show the 
  analogous statement $\lSW (\DWbar) = - \casson (\Ym)$ for the 
  Mrowka--Ruberman--Saveliev invariant $\lSW$ \cite{MRS11}.  This would verify 
  the conjecture that $\lSW = - \lFO$ \cite[Conjecture~B]{MRS11}, for 
  $\ZZ$-homology $S^1 \cross S^3$'s that are of the form $\DWbar$. See 
  \fullref{rmk:sw} for more details.
\end{remark}

\begin{remark}
  \label{rmk:zhs}
  Suppose that $\Yp$ is a $\Z$-homology sphere. Then for any $\Q$-homology 
  sphere $\Ym$, a ribbon $\Q$-homology cobordism from $\Ym$ to $\Yp$ is in fact 
  a ribbon $\Z$-homology cobordism, and the existence of such a cobordism 
  implies that $\Ym$ is also a $\Z$-homology sphere. See \fullref{lem:same-h1} 
  for the proof. This is relevant, for example, to \fullref{thm:seifert} and 
  \fullref{cor:casson}, as well as \fullref{thm:main-i-o} and 
  \fullref{thm:main-i-eq} later.
\end{remark}

To ease our discussion, we set up some conventions for the article.

\begin{convention}
  All $3$- and $4$-manifolds are assumed to be oriented and smooth, and, except 
  in \fullref{ssec:hf-surgery}, they are also assumed to be 
  connected.\footnote{Connectedness is often not essential in our statements, 
    but we impose it for ease of exposition.} Accordingly, we also assume that 
  handle decompositions of cobordisms between non-empty $3$-manifolds have no 
  $0$- or $4$-handles. We say that a handle decomposition is \emph{ribbon} if 
  it has no $3$-handles. We always denote the ends of a ribbon homology 
  cobordism by $\Ypm$; for results that hold for more general cobordisms, we 
  typically denote the cobordism by, for example, $W \colon Y_1 \to Y_2$. All 
  sutured manifolds are assumed to be balanced.  We denote by $I$ the interval 
  $[0, 1]$.
  Unless otherwise specified, all singular homologies have coefficients in 
  $\Z$, instanton Floer homologies have coefficients in $\Q$, and Heegaard 
  Floer homologies have coefficients in $\zeetwo$.
\end{convention}

\subsection{Context}
\label{ssec:intro_context}

In the seminal work of Gordon \cite{Gor:Rib} on ribbon concordance, the key 
theorem, which is very special to the absence of $3$-handles, is the following.

\begin{theorem}[Gordon {\cite[Lemma~3.1]{Gor:Rib}}]
  \label{thm:gordon}
  Let $\Ym$ and $\Yp$ be compact $3$-manifolds possibly with boundary, and 
  suppose that $W \colon \Ym \to \Yp$ is a ribbon $\Q$-homology cobordism.  
  Then
  \begin{enumerate}
    \item\label{it:gordon-inj}  The map $\iota_* \colon \pi_1 (\Ym) \to \pi_1 
      (W)$ induced by inclusion is injective; and
    \item\label{it:gordon-surj}  The map $\iota_* \colon \pi_1 (\Yp) \to \pi_1 
      (W)$ induced by inclusion is surjective.
  \end{enumerate}
\end{theorem}

(While Gordon's original statement is only for exteriors of ribbon 
concordances, the more general result holds, as explained in this subsection.) 
Gordon uses the theorem above, combined with various properties of knot groups, 
to study questions related to ribbon concordance.

Our employment of several different approaches above is motivated by two 
observations.  First, since Gordon's work, there have been many breakthroughs 
in low-dimensional topology, including the Geometrization Theorem for 
$3$-manifolds, the applications of representation theory and gauge theory, and, 
relatedly, the advent of Floer theory. Each of these constitutes a new, 
powerful tool that can be applied in the context of ribbon homology cobordisms, 
and a major goal of the present article is to systematically carry out these 
applications. In particular, we will develop obstructions from these theories, 
which we will then use for topological gain.

Second, while the approaches reflect very different perspectives, there are 
interesting theoretical connections between them. To illustrate this point, we 
discuss \fullref{thm:gordon} further. This theorem follows from the deep 
property of residual finiteness of $3$-manifold groups together with the 
elegant results of Gerstenhaber and Rothaus \cite{GerstenhaberRothaus} on the 
representations of finitely presented groups to a compact, connected Lie group 
$G$.  (The residual finiteness of closed $3$-manifold groups has only been 
known after the proof of the Geometrization Theorem; this new development is 
the ingredient that extends Gordon's original statement to closed $3$-manifolds 
in \fullref{thm:gordon}.) The statement of Gerstenhaber and Rothaus can be 
reinterpreted as saying that the $G$-representations of $\pi_1 (\Ym)$ extend to 
those of $\pi_1 (W)$, and \fullref{thm:gordon}~\eqref{it:gordon-surj} implies 
that any non-trivial representation of $\pi_1 (W)$ determines a non-trivial 
representation of $\pi_1 (\Yp)$ by pullback. Thus, \fullref{thm:gordon} 
naturally leads to the study of the character varieties of $\Ypm$. Moving 
further, focusing on $G = \SUtwo$, we observe that the $\SUtwo$-representations 
of $\pi_1 (\Ypm)$ are related to the instanton Floer homology of $\Ypm$.  Like 
instanton Floer homology, Heegaard Floer homology is defined by considering 
certain moduli spaces of solutions; however, while they share many formal 
properties, the exact relationship between these two theories remains somewhat 
unclear.  Finally, we note that the Geometrization Theorem implies that if 
$\Ypm$ is geometric, its geometry can be determined from $\pi_1 (\Ypm)$ in many 
situations.

In fact, apart from theoretical connections, there is considerable interplay 
among these perspectives even in their \emph{applications}. We direct the 
interested reader to \fullref{thm:seifert}, \fullref{rmk:stein-irred}, and 
\fullref{rmk:sol-lspace} for a few examples.

\subsection{Character varieties and ribbon homology cobordisms}
\label{ssec:intro_character}

As briefly mentioned above, the proof of \fullref{thm:gordon} requires 
understanding the relationship between $G$-representations of $\pi_1 (\Ypm)$ 
and $\pi_1 (W)$.  Consequently, given a ribbon $\Q$-homology cobordism $W 
\colon \Ym \to \Yp$, we will also obtain relations between the character 
varieties of $\Ym$ and $\Yp$.  Recall that for a group $\pi$ and compact, 
connected Lie group $G$ (e.g.\ $\SUtwo$), we can define the 
\emph{representation variety} $\repvarG (\pi)$, which is the set of 
$G$-representations of $\pi$; we can also quotient by the conjugation action to 
obtain the \emph{character variety} $\charvarG (\pi)$.
For a path-connected space $X$, we will write $\repvarG (X)$ for $\repvarG 
(\pi_1(X))$, and $\charvarG (X)$ for $\charvarG (\pi_1 (X))$.   As discussed 
above, we have the following proposition.

\begin{proposition}
  \label{prop:character-embedding}
  Let $\Ym$ and $\Yp$ be compact $3$-manifolds possibly with boundary, and 
  suppose that $W \colon \Ym \to \Yp$ is a ribbon $\Q$-homology cobordism. Then 
  any $\rhom \in \repvarG (\Ym)$ can be extended to an element $\rhoW \in 
  \repvarG (W)$ that pulls back to an element $\rhop \in \repvarG (\Yp)$, and 
  distinct elements in $\repvarG (\Ym)$ corresponds to distinct elements in 
  $\repvarG (\Yp)$. The analogous statement for $\charvarG$ also holds.
\end{proposition} 

See \fullref{prop:character-embedding-2} for a restatement and proof.  Recall 
that the Chern--Simons functional gives an $\R / \Z$-valued function on 
$\repvarG (Y)$; the image of this function is a finite subset of $\R / \Z$, 
which we call the \emph{$G$--Chern--Simons invariants} of $Y$.  
\fullref{prop:character-embedding} implies a relation between the 
$G$--Chern--Simons invariants of $\Ym$ and $\Yp$.

\begin{corollary}
  \label{cor:chern-simons}
  Let $\Ym$ and $\Yp$ be closed $3$-manifolds, and suppose that there exists a 
  ribbon $\Q$-homology cobordism from $\Ym$ to $\Yp$.  Then the set of 
  $G$--Chern--Simons invariants of $\Ym$ is a subset of that of $\Yp$.
\end{corollary}

\begin{remark}
  \label{rmk:stein-irred}
  Stein manifolds provide a large family of $4$-manifolds without $3$-handles.  
  It is interesting to compare the discussion above with the work of Baldwin 
  and Sivek \cite{BS:Stein}, who use instanton Floer homology to prove that if 
  $Y$ is a $\Z$-homology sphere that admits a Stein filling with non-trivial 
  homology, then $\pi_1 (Y)$ admits an irreducible $\SUtwo$-representation.  In 
  comparison, if $W \colon \Ym \to \Yp$ is a Stein $\Q$-homology cobordism, and 
  $\pi_1 (\Ym)$ admits a non-trivial $\SUtwo$-representation, then it extends 
  to an $\SUtwo$-representation of $\pi_1 (W)$ that pulls back to a non-trivial 
  $\SUtwo$-representation of $\pi_1 (\Yp)$ by 
  \fullref{prop:character-embedding}, which requires no gauge theory.
\end{remark}

In fact, with a bit more work, we can compare the local structures of the 
character varieties. For a path-connected space $X$ and a representation $\rho 
\colon \pi_1 (X) \to G$, recall that the \emph{Zariski tangent space} to 
$\charvarG (X)$ at the conjugacy class $[\rho]$ is the first group cohomology 
of $\pi_1 (X)$ with coefficients in the adjoint representation associated to 
$\rho$, denoted by $H^1 (X; \Ad_{\rho})$; see \fullref{ssec:character-grpcoh} 
for more details.  Below, we also consider the zeroth group cohomology $H^0 (X; 
\Ad_{\rho})$.

\begin{proposition}
  \label{prop:zariski}
  Let $\Ym$ and $\Yp$ be compact $3$-manifolds possibly with boundary, and 
  suppose that $W \colon \Ym \to \Yp$ is a ribbon $\Q$-homology cobordism.  Fix 
  $\rhom \in \repvarG (\Ym)$, choose an extension $\rhoW \in \repvarG (W)$, and 
  denote by $\rhop \in \repvarG (\Yp)$ the pullback of $\rhoW$.  Suppose that 
  $\dimR H^0 (\Ym; \Ad_{\rhom}) = \dimR H^0 (\Yp; \Ad_{\rhop})$.  Then
  \[
    \dimR H^1 (\Ym; \Ad_{\rhom}) \leq \dimR H^1 (W; \Ad_{\rhoW}) \leq \dimR H^1 
    (\Yp; \Ad_{\rhop}).
  \]
\end{proposition}

This seemingly technical result, applied to ribbon $\Q$-homology cobordisms 
between Seifert fibered homology spheres, will be our avenue to prove 
\fullref{thm:seifert}~\eqref{it:fibers}.

\subsection{Floer homologies and ribbon homology cobordisms}
\label{ssec:intro_floer}

Another way that representations appear in $3$- and $4$-manifold topology is 
through instanton Floer homology, where we specialize to $G = \SUtwo$ or 
$\SOthree$.  Recall that a Floer homology associates a vector space or module 
to a $3$-manifold, and a linear transformation or homomorphism to a cobordism.  
In the case of instanton Floer homology, the associated group comes roughly 
from counting $\SUtwo$ or $\SOthree$ representations of the fundamental group.
Below, we state a theorem for the behavior of a \emph{general} Floer homology 
theory under ribbon homology cobordisms. In \fullref{sec:results_floer}, we 
give results for most versions of Floer homology with precise conditions on the 
$3$-manifolds and the ribbon homology cobordism.

\begin{theorem}
  \label{thm:main-floer}
  Let $F$ be one of the $3$-manifold Floer homology theories discussed in 
  \fullref{sec:results_floer}.  Let $\Ym$ and $\Ym$ be compact $3$-manifolds, 
  and suppose that $W \colon \Ym \to \Yp$ is a ribbon homology cobordism.  Then 
  $F (W)$ includes $F (\Ym)$ into $F (\Yp)$ as a summand.\footnote{For some 
    flavors of Floer homology, we prove the weaker statement that $F (\Ym)$ is 
    isomorphic to a summand of $F (\Yp)$.}
\end{theorem}

Very recently, Zemke and his collaborators \cite{Zem19, MilZem19, LevZem19} 
have shown that ribbon concordances induce injections on knot Heegaard Floer 
homology and Khovanov homology, and this has led to several other interesting 
results \cite{JuhMilZem19, Sar19}, including an exciting relationship between 
knot Heegaard Floer homology and the bridge index 
\cite[Corollary~1.9]{JuhMilZem19}. In the special case that $F$ is sutured 
Heegaard Floer homology, and $W$ is the exterior of a strongly-homotopy ribbon 
concordance, \fullref{thm:main-floer} recovers the results of Zemke 
\cite[Theorem~1.1]{Zem19} and Miller and Zemke \cite[Theorem~1.2]{MilZem19} on 
knot Heegaard Floer homology. (For a more precise statement, see 
\fullref{cor:main-hf-knot}.)

While much of the work involving Floer homologies is inspired by the work of 
Zemke et al., our proofs use a different argument that holds in a more general 
context.

\subsection*{Organization}
In \fullref{sec:character}, we study the relationship between ribbon homology 
cobordisms and character varieties, proving \fullref{prop:character-embedding} 
and \fullref{prop:zariski}.
In \fullref{sec:geometry}, we prove \fullref{thm:geom-intro}, pertaining to 
Thurston geometries.

Next, in \fullref{sec:results_floer}, we give the precise statements associated 
with \fullref{thm:main-floer}, on the behavior of various versions of Floer 
homology under ribbon homology cobordisms.
The following three sections are then devoted to proving these Floer-theoretic 
results.
First, in \fullref{sec:topology}, we give the necessary topological background 
to analyze the double of a ribbon homology cobordism, and give a short 
application to metrics with positive scalar curvature.
In \fullref{sec:i}, after giving an overview of the Chern--Simons functional 
(proving \fullref{cor:chern-simons}) and instanton Floer homology, we prove 
\fullref{thm:main-i-o} to \fullref{thm:main-i-eq} which are instantiations of 
\fullref{thm:main-floer} for instanton Floer homology, as well as 
\fullref{cor:casson}; we also outline a proof of one of these theorems via 
character varieties.
In \fullref{sec:hf}, we set up the necessary tools for Heegaard Floer homology 
and prove \fullref{thm:main-hf} to \fullref{thm:main-hf-inv} which are versions 
of \fullref{thm:main-floer} for Heegaard Floer homology.

Combining the results above, in \fullref{sec:specific}, we prove some specific 
obstructions that arise from results discussed so far, including 
\fullref{thm:seifert}, \fullref{cor:montesinos}, 
\fullref{cor:composite-to-small}, and other statements.
Finally, in \fullref{sec:surgery}, we provide further applications of ribbon 
homology cobordisms to Dehn surgery problems,
proving
\fullref{thm:surgery-general}.

We provide a few routes for the reader. The reader solely interested in 
character varieties, Thurston geometries, or Dehn surgeries can read only 
\fullref{sec:character}, \fullref{sec:geometry}, or \fullref{sec:surgery}, 
respectively. If the sole interest is in instanton Floer homology, then refer 
to \fullref{sec:results_floer}, \fullref{sec:topology} and \fullref{sec:i}.  
For Heegaard Floer homology, see \fullref{sec:results_floer}, 
\fullref{sec:topology} and \fullref{sec:hf}.  

\subsection*{Acknowledgements}
Many of these ideas were developed while TL was visiting AD at Columbia 
University and DSV and CMMW at Louisiana State University, and we thank both 
departments for their support and hospitality.  Part of the research was 
conducted while AD was at the Simons Center for Geometry and Physics. The 
authors are grateful to Steven Sivek for helping them strengthen 
\fullref{cor:composite-to-small} to hold for all composite knots $\Km$, to 
Cameron Gordon for pointing out a mistake in \fullref{lem:same-h1} in a 
previous version, and to Ian Zemke for pointing out the extension of 
\fullref{thm:main-hf} from $\Q$-homology spheres to closed $3$-manifolds, as 
well as the equivalence between his two graph TQFTs in 
\fullref{ssec:hf-surgery}.  The authors also thank Riley Casper, John Etnyre, 
Sherry Gong, Jen Hom, Misha Kapovich, Zhenkun Li, Lenny Ng, and Yi Ni for 
helpful discussions.  Last but not least, the authors thank the anonymous 
referee for many helpful comments that improved the exposition of the article.

AD was partially supported by NSF Grant DMS-1812033.
TL was partially supported by NSF Grant DMS-1709702 and a Sloan Fellowship.
DSV was partially supported by NSF Grant DMS-1907654 and Simons Foundation 
Grant 524876.
CMMW was partially supported by NSF Grant DMS-2039688 and an AMS--Simons Travel 
Grant.

\section{The fundamental group and character varieties}
\label{sec:character}

In this section, we study the fundamental groups and character varieties of 
$3$-manifolds related by ribbon cobordisms.

\subsection{Background}
\label{ssec:character-bg}

Throughout, we let $G$ denote a compact, connected Lie group.  For a group 
$\pi$, let $\repvarG (\pi)$ denote the space of $G$-representations.  If $X$ is 
a connected manifold, we write $\repvarG (X)$ for $\repvarG (\pi_1 (X))$.  We 
write $\charvarG (\pi)$ for the set of conjugacy classes of 
$G$-representations.  We will omit $G$ from the notation when $G = \SUtwo$.

We first prove the following proposition, which is a restatement of 
\fullref{thm:gordon} and \fullref{prop:character-embedding}. The argument, 
using work of Gerstenhaber and Rothaus \cite{GerstenhaberRothaus}, repeats that 
of Gordon \cite{Gor:Rib} and also that of Cornwell, Ng, and Sivek \cite{CNS16}.

\begin{proposition}
  \label{prop:character-embedding-2}
  Let $\Ym$ and $\Yp$ be compact $3$-manifolds possibly with boundary, and 
  suppose that $W \colon \Ym \to \Yp$ is a ribbon $\Q$-homology cobordism. Then 
  the inclusion $\iota_+ \colon \Yp \to W$ induces a surjection $(\iota_+)_* 
  \colon \pi_1 (\Yp) \to \pi_1 (W)$ and an injection $\iota_+^* \colon \repvarG 
  (W) \to \repvarG (\Yp)$, and the inclusion $\iota_- \colon \Ym \to W$ induces 
  an injection $(\iota_-)_* \colon \pi_1 (\Ym) \to \pi_1 (W)$ and a surjection 
  $\iota_-^* \colon \repvarG (W) \to \repvarG (\Ym)$.
\end{proposition}

\begin{proof}
  Since $W$ consists entirely of $1$- and $2$-handles, we may flip $W$ upside 
  down and view it as a cobordism from $-\Yp$ to $-\Ym$.  From this 
  perspective, $W$ is obtained by attaching $2$- and $3$-handles to $-\Yp$.  It 
  follows that the inclusion from $-\Yp$ into $W$ induces a surjection from 
  $\pi_1 (-\Yp)$ to $\pi_1 (W)$.
  
  For $\iota_- \colon \Ym \to W$, we will prove the second claim first. Choose 
  a representation $\rho \colon \pi_1 (\Ym) \to G$.  Since $W$ is a 
  $\Q$-homology cobordism, it admits a handle decomposition with an equal 
  number $m$ of $1$- and $2$-handles.  This allows us to write $\pi_1 (W) \isom 
  (\pi_1 (\Ym) \fprod \grpgen{b_1, \dotsc, b_m}) / \normal{v_1, \dotsc, v_m}$, 
  where the generators $b_i$ are induced by the $1$-handles and the relators 
  $v_i$ are induced by the $2$-handles.  The words $v_i$ induce a map $K \colon 
  G^m \to G^m$, and the existence of an extension of $\rho$ to $\pi_1 (W)$ is equivalent to 
  solving the equation $K = \vec{e}$.  (To handle the elements in $\pi_1 (\Ym)$ 
  that appear in $v_i$, we apply $\rho$ to the element to view it in $G$.)  By 
  quotienting out by $\pi_1 (\Ym)$, each element $v_i$ induces a word $v_i'$ in 
  the free group $\grpgen{b_1, \dotsc, b_n}$.  Consider the matrix $B$ whose 
  $(ij)^\text{th}$ coordinate is the signed number of times that $b_j$ appears 
  in $v_i'$.  Since $H_1 (W, \Ym; \Q) = 0$, we see that $\det (B) \neq 0$.  It 
  now follows from \cite[Theorem~1]{GerstenhaberRothaus} that there exists a 
  solution to the equation $K = \vec{e}$. 

  Now we show that the inclusion map $(\iota_-)_*$ from $\pi_1 (\Ym)$ to 
  $\pi_1(W)$ is injective.  The residual finiteness property of $3$-manifold 
  groups implies that for any non-trivial $x \in \pi_1 (\Ym)$, there exists a 
  finite quotient $H$ of $\pi_1 (\Ym)$ by a normal subgroup $N$ such that $x 
  \notin N$.
  We claim that the induced map $\overline{(\iota_-)_*} \colon H \to \pi_1 (W) 
  / \normal{(\iota_-)_* (N)}$ is injective; this will imply that $(\iota_-)_* 
  (x)$ is a non-trivial element of $\pi_1(W)$. To prove our claim, note that 
  $\pi_1 (W) / \normal{(\iota_-)_* (N)} \isom (H \fprod \grpgen{b_1, \dotsc, 
    b_m}) / \normal{v_1'', \dotsc, v_m''}$, where $v_i''$ is obtained from 
  $v_i$ by reducing the elements in $\pi_1 (\Ym)$ to $H$.  Now 
  \cite[Theorem~2]{GerstenhaberRothaus} says that there is a finite extension 
  $\widetilde{H}$ containing elements $\beta_1, \dotsc, \beta_m$ such that 
  $v_i'' (\beta_1, \dotsc, \beta_m) = e$ for each $i \in \set{1, \dotsc, m}$.  
  In other words, there is a homomorphism $\Phi \colon \pi_1 (W) / 
  \normal{(\iota_-)_* (N)} \to \widetilde{H}$ such that $\Phi \comp 
  \overline{(\iota_-)_*}$ is the inclusion of $H$ into $\widetilde{H}$. This 
  implies that $\overline{(\iota_-)_*}$ is injective, and our proof is 
  complete.
\end{proof}

\begin{corollary}
  \label{cor:charvarG-finite}
  Let $\Ym$ and $\Yp$ be compact $3$-manifolds possibly with boundary, and 
  suppose that $W \colon \Ym \to \Yp$ is a ribbon $\Q$-homology cobordism.  If 
  $\charvarG (\Yp)$ is finite, then $\charvarG (\Ym)$ is finite.
\end{corollary}

\begin{proof}
  This is a direct consequence of \fullref{prop:character-embedding-2}.
\end{proof}

In the next subsection, we give a more structured comparison of the character 
varieties with a bit more work.  

\subsection{Group cohomology computations and Zariski tangent spaces}
\label{ssec:character-grpcoh}

We briefly review some definitions and constructions in group cohomology; see 
\cite{BrownCohomology} for more details.  Let $\pi$ be a group and let $M$ be a 
$\Z[\pi]$-module.  The group cohomology $H^* (\pi; M)$ with coefficients in $M$ 
is defined by taking a projective $\Z[\pi]$-resolution $\dotsb \to P_1 \to P_0 
\to \Z$ of $\Z$, where $\Z$ has the $\Z[\pi]$-module structure where $\pi$ acts 
by the identity.  Then $C^* (\pi; M)$ is defined by applying $\Hom_{\Z[\pi]} 
(\blank, M)$, and omitting $\Z$, as in
\[
  0 \to \Hom_{\Z[\pi]} (P_0, M) \to \Hom_{\Z[\pi]} (P_1, M) \to \dotsb,
\]
and $H^* (\pi; M)$ is the cohomology of this cochain complex.  A natural way of 
constructing a free resolution of $\Z$ is as follows.  Consider an aspherical CW complex $X$ with $\pi_1 (X) = \pi$, and lift this to a CW 
structure on the universal cover $\cover{X}$.  Then, the (augmented/reduced) CW 
chain complex for $\cover{X}$ naturally inherits a $\Z[\pi]$-module structure, 
where $\pi$ acts by the deck transformation group action, and this is a free 
$\Z[\pi]$-resolution of $\Z$.  (The lift of an individual cell in $X$ yields a 
$\pi$'s worth of cells upstairs, and these constitute a single copy of 
$\Z[\pi]$ in the cellular chain complex for the universal cover.)  Recall that 
a presentation $\pi = \grppre{a_\alpha}{w_\beta}$ determines a CW structure on 
$X$ with one $0$-cell, one $1$-cell $e^1_{\alpha}$ for each generator 
$a_{\alpha}$, and one $2$-cell $e^2_{\beta}$ for each relator $w_\beta$; then 
$H^* (\pi; M)$ can be computed from a cochain complex with Abelian groups
\[
  C^0 (\pi; M) = \Hom_{\Z[\pi]} (\Z[\pi], M) \isom M, \quad C^1 (\pi; M) \isom 
  \prod_\alpha M, \quad C^2(\pi; M) \isom \prod_\beta M,
\]
and possibly non-trivial higher cochain groups $C^i (\pi; M)$ for $i > 2$ that 
we will not be concerned with. The ($\alpha \beta)^{\text{th}}$ component of 
the differential from $C^1(\pi; M)$ to $C^2(\pi; M)$ is non-zero only if 
$a_\alpha$ appears in $w_\beta$.  Indeed, in $X$, if $\overline{e^1_\alpha} \cap \overline{e^2_\beta} 
= \emptyset$, then the same is true in the universal cover.  

Now, given a representation $\rho \in \repvarG (\pi)$, we can consider the 
$\Z[\pi]$-module $\Ad_\rho$, which is the Lie algebra $\fg$ of $G$ with the 
$\Z[\pi]$-action where $\pi$ acts by the composition of $\rho$ and the adjoint 
representation.  Note that $\Ad_\rho$ is in fact an $\R[\pi]$-module, and so 
$H^* (\pi; \Ad_\rho)$ is an $\R$--vector space.  Recall also that $H^1 (\pi; 
\Ad_\rho)$ is the \emph{Zariski tangent space} of $\charvarG (\pi)$ at 
$[\rho]$.    We are now ready to show that ribbon homology cobordisms 
induce relations between the Zariski tangent spaces.

\begin{proof}[Proof of \fullref{prop:zariski}]
  We begin by comparing $\dimR H^1 (W; \Ad_{\rhoW})$ and $\dimR H^1 (\Yp; 
  \Ad_{\rhop})$.  First, we recall the inflation--restriction exact sequence in 
  group cohomology (see, for example, \cite[6.8.3]{Weibel}), which says that, 
  given a normal subgroup $N$ of $\pi$ and a $\Z[\pi]$-module $M$, there exists 
  an injection of $H^1 (\pi / N; M^N)$ into $H^1 (\pi; M)$, where $M^N$ is the 
  subgroup of elements of $M$ fixed by the action of $\pi$ restricted to $N$.  
  It is clear that $M^N$ naturally inherits a $\Z[\pi/N]$-module structure.  
  (Further, if $M$ actually has an $\R[\pi]$-module structure, then everything 
  respects the $\R$--vector space structures.)  

  In our case, we take $\pi = \pi_1(\Yp)$, take $N$ to be the kernel of the 
  quotient map from $\pi_1 (\Yp)$ to $\pi_1 (W)$, and take $M = \Ad_{\rhop}$; 
  then $(\Ad_{\rhop})^N$ is a $\Z[\pi_1(W)]$-module.  By construction, $N 
  \subset \ker (\rhop) \subset \ker (\Ad \comp \rhop)$; thus, $N$ acts by the 
  identity on $\Ad_{\rhop}$, and so $(\Ad_{\rhop})^N$ is in fact 
  $\Ad_{\rho_W}$.  Therefore, we conclude that $\dimR H^1 (W; \Ad_{\rhoW}) \leq 
  \dimR H^1 (\Yp; \Ad_{\rhop})$.  

  Next, we consider the restriction of $\rhoW$ to $\pi_1 (\Ym)$.  Suppose that 
  $\pi_1 (\Ym)$ has a presentation of the form $\pi_1 (\Ym) = \grppre{a_1, 
    \dotsc, a_g}{w_1, \dotsc, w_r}$.  (We do not require $\Ypm$ to be closed, 
  and so there may not exist a balanced presentation.)  Then $\pi_1 (W)$ admits 
  a presentation of the form
  \[
    \pi_1(W) = \grppre{a_1, \dotsc, a_g, b_1, \dotsc, b_m}{w_1, \dotsc, w_r, 
      v_1, \dotsc, v_m}.
  \]
  As discussed above, $H^* (\Ym; \Ad_{\rhom})$ is the cohomology of a cochain 
  complex of the form
  \begin{equation}
    \label{eq:zariski-cochain}
    0 \to \fg \xrightarrow{\psi} \bigdirsum_{i=1}^g \fg \xrightarrow{\phi} 
    \bigdirsum_{j=1}^r \fg \to \dotsb.
  \end{equation}
  Thus, $H^0 (\Ym; \Ad_{\rhom}) = \ker (\psi)$, and $\dimR H^1(\Ym; 
  \Ad_{\rhom}) = \dimR \ker (\phi) - \dimR \im (\psi)$.  We consider a similar 
  setup for $\pi_1 (W)$, where $C^1 (W; \Ad_{\rhoW})$ (resp.\ $C^2 (W;  
  \Ad_{\rhoW})$) has $g + m$ (resp.\ $r + m$) copies of $\fg$, and write 
  $\psi'$ and $\phi'$ for the associated differentials. It is obvious now that  
  the condition $\dimR H^0 (\Ym; \Ad_{\rhom}) = \dimR H^0 (W; \Ad_{\rhoW})$ 
  implies that $\dimR \im (\psi) = \dimR \im (\psi')$.
  
  We now aim to compare $H^1 (\Ym; \Ad_{\rhom})$ and $H^1 (W; \Ad_{\rhoW})$.  
  Note that we have an $\R$--vector space decomposition $C^i (W; \Ad_{\rho_W}) 
  = C^i(\Ym; \Ad_{\rhom}) \dirsum \fg^m$ for $i \in \set{1, 2}$.  Since the 
  relators $w_1, \dotsc, w_r$ do not interact with the $m$ additional 
  generators in $\pi_1(W)$, we have a block decomposition
  \[
    \phi' = \begin{pmatrix} \phi & 0 \\ \eta & \gamma \end{pmatrix}.
  \]
  Writing $\dimR \fg = d$, we note that $\eta$ is a ($dm \times dg$)-matrix and 
  $\gamma$ is a ($dm \times dm$)-matrix.  We deduce
  \begin{align*}
    \dimR H^1 (\Ym; \Ad_{\rhom}) &= \dimR \ker(\phi) - \dimR \im(\psi) \\
    &= dg - \dimR \im (\phi) - \dimR \im (\psi) \\
    &= dg + dm - (\dimR \im (\phi) + dm) - \dimR \im(\psi') \\
    &\leq d(g+m) - \dimR \im (\phi') - \dimR \im (\psi') \\
    &= \dimR \ker (\phi') - \dimR \im (\psi') \\
    &= \dimR H^1 (W; \Ad_{\rhoW}),
  \end{align*}
  which completes the proof.
\end{proof}

\section{Thurston geometries}
\label{sec:geometry}

In this section, we study the relationship between ribbon $\Q$-homology 
cobordisms between compact $3$-manifolds and the Thurston geometries that these 
manifolds admit.

We first prove a homology version of \fullref{thm:gordon}.

\begin{lemma}
  \label{lem:gordon-h1}
  Let $\Ym$ and $\Yp$ be compact $3$-manifolds possibly with boundary, and 
  suppose that $W \colon \Ym \to \Yp$ is a ribbon $\Q$-homology cobordism.  
  Then
  \begin{enumerate}
    \item\label{it:gordon-h1-inj}  The inclusion of $\Ym$ into $W$ induces an 
      injection on $H_1$; and
    \item\label{it:gordon-h1-surj} The inclusion of $\Yp$ into $W$ induces a 
      surjection on $H_1$.
  \end{enumerate}
\end{lemma}

\begin{proof}
  For \eqref{it:gordon-h1-inj}, view $W$ as constructed by attaching $1$- and 
  $2$-handles to $\Ym$; the fact that $W$ is a $\Q$-homology cobordism implies 
  that the attaching circles of the $2$-handles are linearly independent in 
  $H_1 (\Ym \connsum m (S^1 \cross S^2)) / H_1 (\Ym)$, implying that $H_2 (W, 
  \Ym) = 0$.  The statement now follows from the long exact sequence associated 
  to the pair $(W, \Ym)$.

  The statement \eqref{it:gordon-h1-surj} follows from Abelianizing the 
  statement of \fullref{thm:gordon}~\eqref{it:gordon-surj}.
\end{proof}

This has the following consequence, explaining \fullref{rmk:zhs}:

\begin{lemma}
  \label{lem:same-h1}
  Suppose that $\Ym$ and $\Yp$ are $\Q$-homology spheres such that $H_1 (\Ym)$ 
  and $H_1 (\Yp)$ are isomorphic. Then any ribbon $\Q$-homology cobordism from 
  $\Ym$ to $\Yp$ is in fact a ribbon $\Z$-homology cobordism. In particular, in 
  view of \fullref{lem:gordon-h1}, the same conclusion holds in the case that 
  $\Yp$ is a $\Z$-homology sphere.
\end{lemma}

\begin{proof}
  First note that for any ribbon $\Q$-homology cobordism, we have $H_2 (W, \Ym) 
  = H_3 (W, \Ym) = 0$, by considering the attachment of $1$- and $2$-handles to 
  $\Ym$ to form $W$. Analogously, we also have $H_1 (W, \Yp) = H_3 (W, \Yp) = 
  0$. Thus, in general, the only possibly nonzero relative homology groups are 
  $H_1 (W, \Ym) \isom H_2 (W, \Yp)$, which are torsion. (These are isomorphic 
  via $H^2 (W, \Ym)$.)

  When $H_1 (\Ym) \isom H_1 (\Yp)$, \fullref{lem:gordon-h1} implies that $H_1 
  (\Ym)$, $H_1 (W)$, and $H_1 (\Yp)$ all have the same finite cardinality, and 
  so the injection of $H_1 (\Ym)$ into $H_1 (W)$ must be an isomorphism; thus, 
  $H_1 (W, \Ym) = 0$.
\end{proof}

The authors thank Cameron Gordon for pointing out that \fullref{lem:same-h1} is 
false when $b_1 (\Ym) > 0$; this case was mistakenly included in a previous 
version.

We now turn to the key lemma that relates $\pi_1 (\Ym)$ and $\pi_1 (\Yp)$ under 
a ribbon $\Q$-homology cobordism.

\begin{lemma}
  \label{lem:grpprops}
  Let $\prop$ be one of the following properties of groups:
  \begin{enumerate}
    \item \label{item:propfirst} Finite;
    \item Cyclic;
    \item Abelian;
    \item Nilpotent;
    \item \label{item:proplast} Solvable; or
    \item \label{item:virtual} Virtually $\prop'$, where $\prop'$ is one of the 
      properties above.
  \end{enumerate}
  Let $\Ym$ and $\Yp$ be compact $3$-manifolds. Suppose that $\pi_1 (\Yp)$ has 
  property $\prop$, while $\pi_1 (\Ym)$ does not. Then there does not exist a 
  ribbon $\Q$-homology cobordism from $\Ym$ to $\Yp$.
\end{lemma}

\begin{proof}
  Suppose that there exists a ribbon $\Q$-homology cobordism $W \colon \Ym \to 
  \Yp$. By \fullref{thm:gordon}, $\pi_1 (\Ym)$ is a subgroup of $\pi_1 (W)$, 
  which is a quotient of $\pi_1 (\Yp)$. For \eqref{item:propfirst} to 
  \eqref{item:proplast}, the lemma is now evident. For \eqref{item:virtual}, a 
  simple algebraic argument shows that if $\prop'$ is a property inherited by 
  subgroups (resp.\ quotients), then so is the property ``virtually $\prop'$''.
\end{proof}

Let $Y$ be a compact $3$-manifold with empty or toroidal boundary.  These are 
the only cases that we will be interested in.  Then according to 
\cite[Theorem~1.11.1]{AFW}, $Y$ belongs to one of the classes in 
\fullref{fig:hierarchy1} (if $Y$ is closed) or \fullref{fig:hierarchy2} (if $Y$ 
has toroidal boundary).  Indeed, if $Y$ is spherical or has a finite solvable 
cover $\Yt$ that is a torus bundle, then $Y$ is obviously closed. In the latter 
case, by \cite[Theorem~1.10.1]{AFW}, $\Yt$ admits either a Euclidean, $\Nil$-, 
or $\Sol$-geometry; by \cite[Theorem~1.9.3]{AFW}, $Y$ is itself geometric, and, 
according to \cite[Table~1.1]{AFW}, also admits one of these geometries.  That 
the last rows of \fullref{fig:hierarchy1} and \fullref{fig:hierarchy2} 
encompass all remaining cases is a consequence of the Geometric Decomposition 
Theorem; see \cite[Theorem~1.9.1]{AFW}. Note that five out of seven ($S^2 
\cross \R$)-manifolds \cite[p.~457]{Sco83} either have $S^2$ as a boundary 
component or are not orientable; the other two are $S^1 \cross S^2$ and 
$\RPthree \connsum \RPthree$. Also, if $Y$ is geometric and has toroidal 
boundary, and is not homeomorphic to $K \tprod I$, $S^1 \cross D^2$, or $T^2 
\cross I$, then it must have ($\HtwoR$)-, $\SLtwot$-, or hyperbolic geometry.

\begin{figure}[!htbp]
  \setlength{\tabcolsep}{11pt}
  \renewcommand{\arraystretch}{2.5}
  \begin{tabular}{c | c c c}
    & $\pi_1 (Y)$ is finite and ... & \multicolumn{2}{c}{$\pi_1 (Y)$ is 
      infinite and virtually ...}\\
    \hline
    trivial & \tikzmarknode{S3}{$Y \homeo S^3$}\\
    \txt{cyclic but\\ not trivial} & \tikzmarknode{lens}{$Y$ is a lens space} 
    & \tikzmarknode{RP3}{$Y \homeo \RPthree \connsum \RPthree$} & 
    \tikzmarknode{S1S2}{$Y \homeo S^1 \cross S^2$}\\
    \txt{Abelian but\\ not cyclic} & & 
    \multicolumn{2}{c}{\tikzmarknode{Euc}{$Y$ is Euclidean}}\\
    \txt{nilpotent but\\ not Abelian\\ \hspace*{1cm}} & 
    \multirow{2}{*}{\tikzmarknode{fsol}{\txt{$Y$ is spherical and\\ not 
          cyclically covered\\ by $S^3$ or $\PHS$\footnotemark}}} & 
    \multicolumn{2}{c}{\tikzmarknode{Nil}{$Y$ admits a $\Nil$-geometry}}\\
    \txt{solvable but\\ not nilpotent\\ \hspace*{1cm}\\ \hspace*{1cm}} & & 
    \multicolumn{2}{c}{\tikzmarknode{Sol}{$Y$ admits a $\Sol$-geometry}}\\
    \txt{not solvable} & \tikzmarknode{fnsol}{\txt{$Y$ is cyclically\\ 
        covered by $\PHS$}} & 
    \multicolumn{2}{c}{\tikzmarknode{others}{\txt{$Y$ admits an 
          ($\HtwoR$)-,\\ $\SLtwot$-, or hyperbolic\\ geometry, or $Y$ admits 
          a\\ non-trivial geometric\\ decomposition, or $Y$ is\\ not prime 
          (and not $\RPthree \connsum \RPthree$)}}}
  \end{tabular}
  \begin{tikzpicture}[overlay, remember picture, shorten >=5pt, shorten 
    <=5pt]
    \draw [->] (S3.south) to (lens.north -| S3);
    \draw [->] (S3.east) [bend left] to (RP3.north);
    \draw [->] (lens.south -| S3) to (fsol.north -| S3);
    \draw [->] (fsol.east) to (Euc.west);
    \draw [->] (fsol.south -| S3) to (fnsol.north -| S3);
    \draw [->] (fnsol.east) to (others.west);
    \draw [->] (RP3.south) to ([xshift={-5pt}]Euc.north);
    \draw [->] (S1S2.south) to ([xshift={5pt}]Euc.north);
    \draw [->] (Euc.south -| Euc) to (Nil.north -| Euc);
    \draw [->] (Nil.south -| Euc) to (Sol.north -| Euc);
    \draw [->] (Sol.south -| Euc) to (others.north -| Euc);
  \end{tikzpicture}
  \caption{Hierarchy of ribbon $\Q$-homology cobordisms of closed 
    $3$-manifolds. For $3$-manifolds with infinite $\pi_1$, the adverb 
    ``virtually'' applies to all adjectives in the leftmost column.  (For 
    example, the fundamental group of a Euclidean manifold is virtually Abelian 
    but not virtually cyclic.)}
  \label{fig:hierarchy1}
\end{figure}

\begin{figure}[!htbp]
  \setlength{\tabcolsep}{11pt}
  \renewcommand{\arraystretch}{2.5}
  \begin{tabular}{c | c c c c c}
    & \multicolumn{5}{c}{$\pi_1 (Y)$ is infinite and virtually ...}\\
    \hline
    \text{solvable} & \hspace*{2.6pc} & \tikzmarknode{S1D2}{$Y \homeo S^1 
      \cross D^2$} & \tikzmarknode{K2I}{$Y \homeo K^2 \tprod I$} & 
    \tikzmarknode{T2I}{$Y \homeo T^2 \cross I$} & \\
    \txt{not solvable} & \multicolumn{5}{c}{\tikzmarknode{others2}{\txt{$Y$ 
          admits an ($\HtwoR$)-, $\SLtwot$-, or hyperbolic geometry,\\ or $Y$ 
          admits a non-trivial geometric decomposition, or $Y$ is not prime}}}
  \end{tabular}
  \begin{tikzpicture}[overlay, remember picture, shorten >=5pt, shorten 
    <=5pt]
    \draw [->] (S1D2.east) to (K2I.west);
    \draw [->] (K2I.south) to (others2.north);
    \draw [->] (T2I.south) to ([xshift={10pt}]others2.north);
  \end{tikzpicture}
  \caption{Hierarchy of ribbon $\Q$-homology cobordisms of compact 
    $3$-manifolds with toroidal boundary.}
  \label{fig:hierarchy2}
\end{figure}

\begin{theorem}
  \label{thm:geom-hierarchy}
  Suppose that $\Ym$ and $\Yp$ are compact $3$-manifolds with empty or toroidal 
  boundary that belong to distinct classes in \fullref{fig:hierarchy1} or 
  \fullref{fig:hierarchy2}, such that there does not exist a sequence of arrows 
  from the class of $\Ym$ to the class of $\Yp$.  Then there does not exist a 
  ribbon $\Q$-homology cobordism from $\Ym$ to $\Yp$.
\end{theorem}

\begin{proof}
  We begin by inspecting \fullref{fig:hierarchy1}, which consists of two 
  columns corresponding to whether $\pi_1 (Y)$ is finite; we call them the 
  \emph{finite column} and the \emph{infinite column} respectively. Focusing on 
  each of these columns separately, successive application of 
  \fullref{lem:grpprops} shows that there are no arrows that point up. Of 
  course, one must check that the manifolds in each class indeed have 
  fundamental groups that are characterized by the property on the left. For 
  the finite column, $Y$ is a lens space if and only if $\pi_1 (Y)$ is cyclic, 
  and $\pi_1 (Y)$ is solvable unless it is the direct sum of a cyclic group and 
  the binary icosahedral group $P_{120}$ (in which case $Y$ is known as a 
  \emph{type-$\icosah$ manifold}); the only spherical $3$-manifold with 
  fundamental group $P_{120}$ is the Poincar\'e homology sphere $\PHS$. See 
  \cite[Section~1.7]{AFW} for a discussion. For the infinite column, the 
  classification by $\pi_1$ follows from \cite[Table~1.1 and Table~1.2]{AFW}; 
  the fact there are no arrows between $\RPthree \connsum \RPthree$ and $S^1 
  \cross S^2$ reflects the fact that their $\Q$-homologies have different 
  ranks.

  We now move on to arrows between the two columns. First, there are clearly no  
  arrows from the infinite to the finite column.  Also, the ranks of the 
  $\Q$-homologies obstruct any arrow from the finite column to $S^1 \cross 
  S^2$. The only remaining obstructions are as follows. There are no arrows
  \begin{enumerate}
    \item From lens spaces to $\RPthree \connsum \RPthree$. Indeed, 
      \fullref{lem:gordon-h1} implies that $H_1 (\Ym)$ is a subgroup of a 
      quotient of $H_1 (\RPthree \connsum \RPthree)$, and thus can only be the 
      trivial group, $\zeetwo$, or $\zeetwo \dirsum \zeetwo$. Since $\Ym$ is a 
      lens space, it has non-trivial cyclic $H_1$; thus, $H_1 (\Ym) \isom 
      \zeetwo$.  Suppose there exists a ribbon $\Q$-homology cobordism from 
      $\Ym$ to $\RPthree \connsum \RPthree$; then $(- \Ym) \connsum \RPthree 
      \connsum \RPthree$ bounds a $\Q$-homology ball. This implies that 
      $\abs{H_1 (- \Ym \connsum \RPthree \connsum \RPthree)} = 8$ is a perfect 
      square, which is a contradiction.

\footnotetext{For this class, $\pi_1 (Y)$ is solvable but not Abelian; it is 
  nilpotent if and only if it is a direct sum of a cyclic group and the 
  generalized quaternion group $Q_{2^n}$; all such manifolds $Y$ are prism 
  (i.e.\ type-$\dih$) manifolds. One could accordingly stratify the class into 
  two classes with an arrow between them.}

    \item From spherical manifolds that are not cyclically covered by $S^3$ to 
      $\RPthree \connsum \RPthree$.\footnote{An alternative proof can be given 
        here as follows. First deduce that $H_1 (\Ym) \isom \zeetwo \dirsum 
        \zeetwo$ with an argument involving perfect squares, which implies that 
        the ribbon $\Q$-homology cobordism is a $\Z$-homology cobordism. Then 
        observe that by \cite[Example~15]{Doi15}, the set of $d$-invariants of 
        $\Ym$ does not match that of $\RPthree \connsum \RPthree$, a 
        contradiction.}  Here, $\Ym$ is a non--lens space spherical manifold.  
      Recall that such a manifold has $\pi_1$ isomorphic to a central extension 
      of a polyhedral group, which in particular is a non-cyclic group, with 
      elements of order $4$, that does not embed into a dihedral group; see 
      \cite[Section~1.7]{AFW} and \cite[Section~6.2]{Orlik}.  Suppose there 
      exists a ribbon $\Q$-homology cobordism from $\Ym$ to $\RPthree \connsum 
      \RPthree$; then \fullref{thm:gordon} implies that $\pi_1 (\Ym)$ is a 
      subgroup of a quotient of $\pi_1 (\RPthree \connsum \RPthree) \isom 
      \zeetwo \fprod \zeetwo$.  However, it is an elementary exercise to see 
      that each quotient of $\zeetwo \fprod \zeetwo$ is either a cyclic group, 
      a dihedral group, or itself (which does not contain elements of order 
      $4$). In any case, $\pi_1 (\Ym)$ cannot be a subgroup of a quotient of 
      $\zeetwo \fprod \zeetwo$, which is a contradiction.
    \item From type-$\icosah$ manifolds to any manifold with solvable $\pi_1$; 
      for manifolds in the infinite column, these are exactly the ones with 
      virtually solvable $\pi_1$ (see \cite[Theorem~1.11.1]{AFW}), i.e.\ all 
      classes except the one in the last row.
  \end{enumerate}

  For \fullref{fig:hierarchy2},
  it suffices to observe that, in the first row, the $\Q$-homology of $T^2 
  \cross I$ differs from those of $S^1 \cross D^2$ and $K^2 \tprod I$, and 
  \fullref{lem:gordon-h1} shows that there is no ribbon $\Q$-homology cobordism 
  from $K^2 \tprod I$ to $S^1 \cross D^2$.
\end{proof}

\begin{remark}
  \label{rmk:S1D2-K2I}
  It is easy to construct a ribbon $\Q$-homology cobordism from $S^1 \cross   
  D^2$ to $K^2 \tprod I$.
\end{remark}

\begin{remark}
  \label{rmk:sol-lspace}
  Boyer, Gordon, and Watson \cite[Theorem~2]{BGW13} show that all $\Q$-homology 
  spheres with $\Sol$-geometry are $L$-spaces. By \fullref{cor:lspace}, there 
  do not exist ribbon $\zeetwo$-homology cobordisms from any $\Q$-homology 
  sphere that is not an $L$-space to a manifold that admits a $\Sol$-geometry.  
  Observe that this is consistent with \fullref{fig:hierarchy1}, since 
  $\Q$-homology spheres with spherical, ($S^2 \cross \R$)-, Euclidean, and 
  $\Nil$-geometry are also $L$-spaces \cite[Proposition~5]{BGW13}.
\end{remark}

\section{Statements of results on Floer homologies}
\label{sec:results_floer}

In the next few sections of this article, we will prove a number of results of the 
following flavor:  If $W \colon \Ym \to \Yp$ is a ribbon homology cobordism, 
then $F(\Ym)$ is a summand of $F (\Yp)$, where $F$ is a version of Floer 
homology (e.g.\ sutured instanton Floer homology, involutive Heegaard Floer 
homology, etc.).  In the theorems below, we give the precise statements, which 
have varying technical hypotheses and conclusions.  However, the rough idea is 
the same throughout and indeed quite simple, which is to show that the double 
$\double{W}$ of $W$ induces an isomorphism on Floer homology. All cobordism 
maps and isomorphisms can easily be checked to be graded; we leave this task to 
the reader, although we do use this fact in \fullref{thm:seifert} and 
\fullref{cor:hf-seifert} below.

We begin with results for instanton Floer homology. We start with Floer's 
original homology $\Io$ for $\Z$-homology spheres \cite{Floer:inst}.

\begin{theorem}
  \label{thm:main-i-o}
  Let $\Ym$ and $\Yp$ be $\Z$-homology spheres, and suppose that $W \colon \Ym 
  \to \Yp$ is a ribbon $\Q$-homology cobordism.  Then the cobordism map $\Io 
  (\double{W}) \colon \Io (\Ym) \to \Io (\Ym)$ is the identity map up to a 
  sign, and $\Io (W)$ includes $\Io (\Ym)$ into $\Io (\Yp)$ as a summand.
\end{theorem}

Next, we have an analogous statement for the framed instanton Floer homology 
$\Is$ \cite{KM:YAFT}.

\begin{theorem}
  \label{thm:main-i-s}
  Let $\Ym$ and $\Yp$ be closed $3$-manifolds, and suppose that $W \colon \Ym 
  \to \Yp$ is a ribbon $\Q$-homology cobordism.  Then the cobordism map $\Is 
  (\double{W}) \colon \Is (\Ym) \to \Is (\Ym)$ satisfies
  \[
    \Is (\double{W}) = \abs{H_1 (W, \Ym)} \cdot \Id_{\Is (\Ym)}
  \]
  up to a sign, and $\Is (W)$ includes $\Is (\Ym)$ into $\Is (\Yp)$ as a 
  summand.
\end{theorem}

\fullref{thm:main-i-s} implies the following corollary, which may also be 
proved using \fullref{thm:main-hf} below for ribbon $\zeetwo$-homology 
cobordisms.

\begin{corollary}
  \label{cor:norm}
  Let $\Ym$ and $\Yp$ be closed $3$-manifolds, and suppose that there exists a 
  ribbon $\Q$-homology cobordism from $\Ym$ to $\Yp$. Then the unit Thurston 
  norm ball of $\Ym$ includes that of $\Yp$.
\end{corollary}

\begin{proof}
  This follows from the fact that $\Is$ detects the Thurston norm 
  \cite{KM:sutured}, together with the fact that a ribbon $\Q$-homology 
  cobordism induces a concrete identification between $H_2 (\Ym; \Q)$ and $H_2 
  (\Yp; \Q)$.
\end{proof}

The following is an analogue for the sutured instanton Floer homology $\SHI$ 
\cite{KM:sutured}.  Here and below, by a cobordism between sutured manifolds, 
we mean a cobordism obtained by attaching interior handles to a product 
cobordism; this means that the $3$-manifolds have isomorphic sutured 
boundaries. This definition is narrower than the one used by Juh\'asz 
\cite{Juh16}. See \fullref{defn:sut-cob} for a precise definition.

\begin{theorem}
  \label{thm:main-i-sut}
  Let $\Msutm$ and $\Msutp$ be sutured manifolds, and suppose that $N \colon 
  \Msutm \to \Msutp$ is a ribbon $\Q$-homology cobordism.  Then the cobordism 
  map $\SHI (\double{N}) \colon \SHI \Msutm \to \SHI \Msutm$ satisfies
  \[
    \SHI (\double{N}) = \abs{H_1 (N, \Mm)} \cdot \Id_{\SHI \Msutm}
  \]
  up to a sign, and $\SHI (N)$ includes $\SHI \Msutm$ into $\SHI \Msutp$ as a 
  summand.
\end{theorem}

Recall that for a knot $K$ in a closed $3$-manifold $Y$,
the sutured instanton Floer homology of the exterior of $K$ is also denoted by 
$\KHI (Y, K)$ \cite{KM:sutured}.  By the isomorphism between $\KHI$ and the 
reduced singular knot instanton Floer homology $\In$
\cite{KM:unknot}, \fullref{thm:main-i-sut} implies the following result. 

\begin{corollary}
  \label{cor:main-i-knot}
  Let $\Ym$ and $\Yp$ be closed $3$-manifolds, and let $\Km$ and $\Kp$ be knots 
  in $\Ym$ and $\Yp$ respectively. Suppose that there exists a concordance $C 
  \colon \Km \to \Kp$ in a cobordism $W \colon \Ym \to \Yp$, such that the 
  exterior of $C$ is a ribbon $\Q$-homology cobordism.  Then the cobordism map 
  $\In (\double{W}, \double{C}) \colon \In (\Ym, \Km) \to \In (\Ym, \Km)$ 
  satisfies
  \[
    \In (\double{W}, \double{C}) = \abs{H_1 (W, \Ym)} \cdot \Id_{\In (\Ym, 
      \Km)}
  \]
  up to a sign, and $\In (W, C)$ includes $\In (\Ym, \Km)$ into $\In (\Yp, 
  \Kp)$ as a summand.
\end{corollary}

\begin{remark}
  \label{rmk:main-i-knot-ext}
  Sherry Gong has informed the authors of a direct proof of a version of 
  \fullref{cor:main-i-knot} with coefficients in $\Z$ for concordances in $\Ym 
  \cross I$, without appealing to the isomorphism between $\KHI$ and 
  $\In$.\footnote{Since the first appearance of this article, a version of 
    Gong's argument has appeared in the work of Kronheimer and Mrowka 
    \cite[Theorem~7.4]{KM-conc}.}
  Kang \cite{Kang} has very recently provided a general proof of 
  \fullref{cor:main-i-knot} for conic strong Khovanov--Floer theories for 
  concordances in $S^3 \cross I$, which may be used to recover the version of 
  \fullref{cor:main-i-knot} for ribbon concordances in $S^3 \cross I$.
\end{remark}

\begin{remark}
  \label{rmk:branched}
  One can easily see that the double cover of $S^3 \cross I$ branched along the 
  concordance is a ribbon $\zeetwo$-homology cobordism, and so 
  \fullref{thm:main-i-s} applies to show an inclusion of $\Is (\Sigma_2 (\Km))$ 
  into $\Is (\Sigma_2 (\Kp))$, for knots $\Km$ and $\Kp$ in $S^3$.  A similar 
  statement holds for surgeries along $\Kpm$. We omit these statements for 
  brevity.
\end{remark}

We also provide a version for equivariant instanton Floer homologies 
\cite{Don:YM-Floer, AD:CS-Th}. Denote by $\Ig$ any of the equivariant instanton 
Floer homologies $\It$, $\If$, and $\Ib$.\footnote{The homologies $\It$, $\If$, 
  and $\Ib$ may be viewed as analogues of $\HFp$, $\HFm$, and $\HFi$ 
  respectively.} (We adopt the notation in \cite{AD:CS-Th} for these 
homologies.)

\begin{theorem}
  \label{thm:main-i-eq}
  Let $\Ym$ and $\Yp$ be $\Z$-homology spheres, and suppose that $W \colon \Ym 
  \to \Yp$ is a ribbon $\Q$-homology cobordism.  Then the cobordism map $\Ig 
  (W)$ includes $\Ig (\Ym)$ into $\Ig (\Ym)$ as a summand.
\end{theorem}

\begin{remark}
  \label{rmk:main-i-eq}
  Equivariant instanton Floer homologies can be extended to $\Q$-homology 
  spheres (with certain auxiliary data) \cite{MMiller, AustinBraam}. We expect 
  (but do not prove) that \fullref{thm:main-i-eq} holds also for these 
  extensions.
\end{remark}

We now turn to Heegaard Floer homology \cite{OzsSza04:HF}.  Denote by $\HFg$ 
any of the Heegaard Floer homologies $\HFh$, $\HFp$, $\HFm$, and $\HFi$, and by 
$\Fg_W$ the corresponding cobordism map.

\begin{theorem}
  \label{thm:main-hf}
  Let $\Ym$ and $\Yp$ be closed $3$-manifolds, and suppose that $W \colon \Ym 
  \to \Yp$ is a ribbon $\Z/2$-homology cobordism. Then the cobordism map 
  $\Fg_W$ includes $\HFg (\Ym)$ into $\HFg (\Yp)$ as a summand.  In fact, 
  $\Fh_{\double{W}} \colon \HFh (\Ym) \to \HFh (\Ym)$ is the identity map.
\end{theorem}

\begin{remark}
  \label{rmk:main-hf-spinc}
  We also provide a $\SpinC$-refinement of \fullref{thm:main-hf}; see 
  \fullref{thm:main-hf-spinc} for the precise statement.
\end{remark}

As in instanton Floer theory, there is also a version for the sutured Heegaard 
Floer homology $\SFH$ \cite{Juh06}. We expect that the stated isomorphism below 
coincides with the cobordism map defined by Juh\'asz \cite{Juh16}, although we 
do not prove it.

\begin{theorem}
  \label{thm:main-hf-sut}
  Let $\Msutm$ and $\Msutp$ be sutured manifolds, and suppose that there exists 
  a ribbon $\zeetwo$-homology cobordism from $\Msutm$ to $\Msutp$.  Then $\SFH 
  \Msutm$ is isomorphic to a summand of $\SFH \Msutp$.
\end{theorem}

As in \fullref{cor:main-i-knot}, by the isomorphism 
\cite[Proposition~9.2]{Juh06} between the knot Heegaard Floer homology $\HFKh$ 
\cite{OzsSza04:HFK, Ras03:HFK} of a null-homologous knot and $\SFH$ of its 
exterior, \fullref{thm:main-hf-sut} immediately implies the following statement 
for such concordances. This recovers a version of the results in \cite{Zem19} 
and \cite{MilZem19} when the concordance is in $S^3 \cross I$;\footnote{Note 
  that the exterior of a concordance in $S^3 \cross I$ is a $\Z$-homology 
  cobordism.} again, we do not prove that the stated isomorphism coincides with 
the knot cobordism map.

\begin{corollary}[{cf.~\cite[Theorem~1.1]{Zem19} and 
    \cite[Theorem~1.2]{MilZem19}}]
  \label{cor:main-hf-knot}
  Let $\Ym$ and $\Yp$ be closed $3$-manifolds, and let $\Km$ and $\Kp$ be 
  null-homologous knots in $\Ym$ and $\Yp$ respectively. Suppose that there 
  exists a concordance from $\Km$ to $\Kp$ in a cobordism from $\Ym$ to $\Yp$, 
  whose exterior is a ribbon $\zeetwo$-homology cobordism.  Then $\HFKh (\Ym, 
  \Km)$ is isomorphic to a summand of $\HFKh (\Yp, \Kp)$.  \qed
\end{corollary}

\begin{remark}
  \label{rmk:main-i-knot-app}
  \fullref{cor:main-hf-knot} has been used to obtain a genus bound on knots 
  related by ribbon concordance \cite[Theorem~1.5]{Zem19} analogous to 
  \fullref{cor:norm}, and on band connected sums of knots 
  \cite[Theorem~1.6]{Zem19}; \fullref{cor:main-i-knot} provides an alternative 
  proof of these results using knot instanton Floer homology. It also recovers 
  the well-known theorem that, if a ribbon concordance exists in $S^3 \cross I$ 
  from $\Km$ to $\Kp$, where $\Km$ and $\Kp$ have the same genus, then the 
  fiberedness of $\Kp$ implies that of $\Km$.
\end{remark}

As explained in \fullref{rmk:branched}, one could also use 
\fullref{thm:main-hf} to obtain analogous statements for $\HFg$ of certain 
cyclic covers of $S^3$ branched along $\Kpm$, and for surgeries along $\Kpm$.

We also give an extension for the involutive Heegaard Floer homology $\HFIh$ 
\cite{HendricksManolescu}.

\begin{theorem}
  \label{thm:main-hf-inv}
  Let $\Ym$ and $\Yp$ be closed $3$-manifolds, and suppose that there exists a 
  ribbon $\Z$-homology cobordism from $\Ym$ to $\Yp$.  Then $\HFIh (\Ym)$ is 
  isomorphic to a summand of $\HFIh (\Yp)$.
\end{theorem}

The rough strategy for proving all of the theorems above is fairly 
straightforward.  First, a topological argument (\fullref{prop:surgery} below)
shows that $\double{W}$ is given by surgery along a collection of $m$ loops in 
$(\Ym \times I) \connsum m (S^1 \cross S^3)$.   By using surgery formulas, it 
can be shown that the induced map for $\double{W}$ is the same as that for the 
$4$-manifold obtained by surgering along the $m$ cores of the $S^1 \cross S^3$ 
summands, which is just $\Ym \cross I$.  Of course, this induces the identity 
map.

We will also outline an alternative proof of \fullref{thm:main-i-o} in 
\fullref{ssec:characters-i} that passes more directly through the fundamental 
group and \fullref{thm:gordon}.

\begin{remark}
  \label{rmk:sw}
  We expect the analogue of \fullref{thm:main-hf} to hold also for the monopole 
  Floer homology groups $\HMt$, $\HMf$, and $\HMb$ \cite{KM:monopole}.  Note 
  that by the isomorphisms between Heegaard and monopole Floer homologies 
  \cite{KLT1, CGH, TaubesECHSWF}, we already know that $\HMg (\Ym)$ is 
  isomorphic to a summand of $\HMg (\Yp)$. In order to prove that the 
  isomorphism coincides with the cobordism map, one could, for example, prove a 
  surgery formula analogous to \fullref{prop:surgery-maps} for monopole Floer 
  homology. Although we expect that this surgery formula holds for monopole 
  Floer homology (especially because an analogous result holds for a variation 
  of Bauer--Furuta invariants \cite[Example~1.4]{KhaLinSas19}), we do not give 
  a proof of this result for brevity.
    
  We also expect an analogue of \fullref{cor:casson} to hold for the 
  Mrowka--Ruberman--Saveliev invariant $\lSW$ \cite{MRS11}. Using the splitting 
  theorem \cite{LRS18}, we have
  \[
    \lSW (\DWbar) = - \Lef (\HMr (\double{W})) - h (\Ym),
  \]
  where $h$ is the monopole Fr{\o}yshov invariant. Since the Casson invariant 
  of $\Ym$ can alternatively be computed as $\chi (\HMr (\Ym)) + h (\Ym)$, we 
  would obtain that $\lSW (\DWbar)=- \casson (\Ym)$. In particular, we have 
  $\lSW (\DWbar)=-\lFO (\DWbar)$. This would verify \cite[Conjecture~B]{MRS11} 
  for the $4$-manifolds with the $\ZZ$-homology of $S^1 \cross S^3$ that have 
  the form $\DWbar$.
\end{remark}

\section{Topology of the double of a ribbon cobordism}
\label{sec:topology}

Recall that the \emph{double} $\double{W}$ of a cobordism $W \colon Y_1 \to 
Y_2$ is formed by gluing $W$ and $-W$ along $Y_2$. In analogy with the 
arguments used in ribbon concordance, our strategy to prove 
\fullref{thm:main-i-o}, \fullref{thm:main-i-s}, and \fullref{thm:main-hf} will 
be to prove the cobordism map on Floer homology induced by $\double{W}$ is an 
isomorphism, when $W$ is ribbon.  First, we need a topological description of 
$\double{W}$.  In what follows, we will use $\F$ to denote any field.   Note 
that a ribbon $\F$-homology cobordism has the same number of $1$- and 
$2$-handles.

\begin{proposition}
  \label{prop:surgery}
  Let $\Ym$ and $\Yp$ be compact $3$-manifolds, and suppose that $W \colon \Ym 
  \to \Yp$ is a ribbon cobordism, where the number of $1$-handles is $m$, and 
  that of $2$-handles is $\ell$. Then $\double{W}$ can be described by surgery 
  on $X \homeo (\Ym \times I) \connsum m (S^1 \times S^3)$ along $\ell$ 
  disjoint simple closed curves $\gamma_1, \dotsc, \gamma_\ell$.
  
  Suppose that, in addition, $W$ is also an $\F$-homology cobordism, and denote 
  by $\alpha_i \in H_1 (X)$ the homology class of the core of the 
  $i^{\text{th}}$ $S^1 \cross S^3$ summand.  Then, writing
  \[
    [\gamma_j] = \sigma_j + \sum_{i=1}^m c_{ij} \alpha_i, \quad \sigma_j \in 
    H_1 (\Ym),
  \]
  we have that the matrix $(c_{ij}) \tensor_{\Z} \F$ is invertible over $\F$, 
  and $\lvert \det (c_{ij}) \rvert = \abs{H_1 (W, \Ym)}$; in particular,
  \[
    [\gamma_1] \wedge \dotsb \wedge [\gamma_\ell] = \det (c_{ij}) \cdot 
    \alpha_1 \wedge \dotsb \wedge \alpha_\ell \in (\extprod^* (H_1 (X) / \Tors) 
    / \langle H_1(\Ym) / \Tors \rangle) \tensor_{\Z} \F,
  \]
  where $\langle H_1(\Ym) / \Tors \rangle$ is the ideal generated by $H_1(\Ym) 
  / \Tors$, is an equality of non-zero elements.
\end{proposition}

Of course, working with elements in $\extprod^* (H_1 (X) / \Tors ) / \langle 
H_1(\Ym) / \Tors \rangle$ is the same as first projecting $H_1 (X)$ to the 
submodule corresponding to the $S^1 \cross S^3$ summands and then working in 
the exterior algebra there.  See \fullref{fig:switch} for a schematic diagram 
when $m = 1$.

\begin{figure}[!htbp]
  \labellist
  \hair 2pt
  \pinlabel $\Ym$ at -7 25
  \pinlabel $\Ym$ at 170 25
  \pinlabel {$X \setminus \nbhdgamma$} at 82 25
  \pinlabel {$S^1 \times D^3$} at -15 100
  \pinlabel {$D^2 \times S^2$} at 178 100
  \endlabellist
  \includegraphics{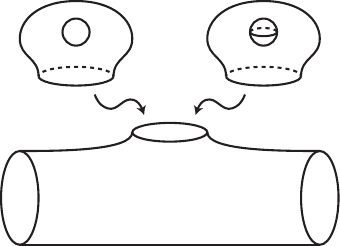}
  \caption{An illustration of \fullref{prop:surgery} in the case of $m = \ell = 
    1$.  Here, $\nbhdgamma$ denotes a neighborhood of $\gamma$. Reattaching the 
    $S^1 \cross D^3$ would yield $X \homeo (\Ym \cross I) \connsum  (S^1 \cross 
    S^3)$, while we may obtain $\double{W}$ by switching it for the $D^2 \cross 
    S^2$.}
  \label{fig:switch}
\end{figure}

Before we prove \fullref{prop:surgery}, we first establish an elementary fact.

\begin{lemma}
  \label{lem:surgery}
  Let $M_1$ and $M_2$ be $(n-1)$-manifolds, and suppose that $N \colon M_1 \to 
  M_2$ is a cobordism associated to attaching an $n$-dimensional $k$-handle 
  $h$. Then the double $\double{N}$ can be described by surgery on $M_1 \cross 
  [-1, 1]$ along some $S^{k-1} \subset M_1 \cross \set{0}$ given by the 
  attaching sphere of $h$.
\end{lemma}

\begin{proof}
  Write $\double{N} = (M_1 \cross [-1, 0]) \union h \union h' \union (M_1 
  \cross [0, 1])$, where $h'$ is the dual handle of $h$.  The cocore of $h$ and 
  the core of $h'$ together form an $S^{n-k}$ with trivial normal bundle, which 
  may be identified with $h \union h'$. (The case where $n = 4$ and $k = 2$ is 
  described, for example, in \cite[Example~4.6.3]{GomSti99}.) Note that $h$ 
  meets the lower $M_1 \cross [-1, 0]$, and $h'$ meets the upper $M_1 \cross 
  [0, 1]$, at the same attaching region $S^{k-1} \cross D^{n-k} \subset M_1 
  \cross \set{0}$, with the same framing.  Thus, removing $h \union h'$ from 
  $\double{N}$ would result in $(M_1 \cross [-1, 1]) \setminus (S^{k-1} \cross 
  D^{n-k} \cross (-\epsilon, \epsilon))$.  In other words, $\double{N}$ may be 
  formed by removing $S^{k-1} \cross D^{n-k} \cross (-\epsilon, \epsilon) 
  \homeo S^{k-1} \cross D^{n-k+1}$ from $M_1 \cross [-1, 1]$ and replacing it 
  with $h \union h' \homeo D^k \cross S^{n-k}$, which is the definition of 
  surgery.
\end{proof}

In the case where $n = 4$ and $k = 2$, the handles $h$ and $h'$ above can be 
described by a Kirby diagram consisting of a loop $\gamma$ with some (possibly 
non-zero) framing and the linking circle of $\gamma$ with zero framing; the 
fact that this corresponds to surgery is well known to experts; see, for 
example, \cite[p.~500]{Akbulut}.

\begin{proof}[Proof of \fullref{prop:surgery}]
  First, decompose $W$ into a cobordism $W_1$ from $\Ym$ to $\Yt \homeo \Ym 
  \connsum m (S^1 \cross S^2)$ and a cobordism $W_2$ from $\Yt$ to $\Yp$, 
  corresponding to the attachment of $1$- and $2$-handles respectively. Below, 
  we will compare $\double{W} = W_1 \union W_2 \union \orrev{W_2} \union 
  \orrev{W_1}$ with $W_1 \union \orrev{W_1}$.

  Applying \fullref{lem:surgery} to each of the $2$-handles in $W_2$, we see 
  that $W_2 \union \orrev{W_2}$ can be described by surgery on $\Yt \cross [-1, 
  1]$ along some $\gamma_1, \dotsc, \gamma_\ell$, where the $\gamma_i$'s are 
  given by the attaching circles of the $2$-handles.  (Perform isotopies and 
  handleslides first, if necessary, to ensure that the attaching regions of the 
  $2$-handles lie in $\Yt$ and are disjoint.)

  Note that $W_1 \union \orrev{W_1} \homeo W_1 \union (\Yt \cross [-1, 1]) 
  \union \orrev{W_1}$ is diffeomorphic to $X \homeo (\Ym \cross I) \connsum m 
  (S^1 \cross S^3)$.  Thus, we see that $\double{W} = W_1 \union W_2 \union 
  \orrev{W_2} \union \orrev{W_1}$ can be described by surgery on $X$ along 
  $\gamma_1, \dotsc, \gamma_\ell$.

  Finally, suppose $W$ is a ribbon $\F$-homology cobordism; then $m = \ell$.  
  Present the differential $\bdy_2 \colon C_2 (\Ym) \to C_1 (\Ym)$ by a matrix 
  $A$; then in the corresponding cellular chain complex of $W$, the 
  presentation matrix $Q$ of the differential $\bdy_2 \colon C_2 (W) \to C_1 
  (W)$ is of the form
  \[
    Q = \begin{pmatrix} A & B \\ 0 & C \end{pmatrix},
  \]
  where $C$ is an ($m \times m$)-matrix representing the attachment of the 
  $2$-handles in $W_2$. As the $\gamma_i$'s are given by the attaching circles 
  of these $2$-handles, we see that $C_{ij}$ is given by the algebraic 
  intersection number of $\gamma_j$ with $\set{p} \cross S^3$ in the 
  $i^\text{th}$ $S^1 \cross S^3$ summand, and so $C_{ij} = c_{ij}$.  Since $W$ 
  is an $\F$-homology cobordism, we have $H_1 (W, \Ym; \F) = 0$, implying that 
  $C \tensor_{\Z} \F \colon \F^m \to \F^m$ is surjective, and hence invertible.  
  It is now clear that $\lvert \det (C_{ij}) \rvert = \abs{H_1 (W, \Ym)}$, and 
  the equality in $(\extprod^* (H_1 (X) / \Tors) / \langle H_1 (\Ym) / \Tors 
  \rangle) \tensor_{\Z} \F$ is obvious.
\end{proof}

While it will not be used later in the paper, we conclude this section with the 
following geometric result, which may be of independent interest.

\begin{proposition}
  \label{prop:pos-curv}
  Suppose that $W$ is a compact $4$-manifold with connected boundary and a 
  ribbon handle decomposition. Then $W$ admits a metric with positive scalar 
  curvature.
\end{proposition}

\begin{proof}
  By \fullref{prop:surgery}, $\double{W}$ is obtained by surgery on a 
  collection of $\ell$ loops in $\bigconnsum m (S^1 \cross S^3)$. First, it is 
  well known that $S^1 \cross S^3$ has a p.s.c.\ metric.  By the work of Gromov 
  and Lawson \cite[Theorem~A]{GromovLawson}, $\bigconnsum m (S^1 \cross S^3)$ 
  admits a p.s.c.\ metric.  Next, surgery on loops is a codimension-$3$ 
  surgery, and so we may again apply the result of Gromov and Lawson to see 
  that $\double{W}$ admits a p.s.c.\ metric.  Since $W$ is a codimension-$0$ 
  submanifold of $\double{W}$, it inherits a p.s.c.\ metric as well.
\end{proof}

\section{Instanton Floer homology}
\label{sec:i}

\subsection{The Chern--Simons functional}
\label{ssec:i-CS}

Let $G$ be a compact, connected, simply connected, simple Lie group, and let 
$P$ be a principal $G$-bundle on $Y$. Any such bundle can be trivialized, and 
we fix one such trivialization. Denote by $\ad (P)$ the adjoint bundle 
associated to $P$; this vector bundle is induced by the adjoint action of $G$ 
on its Lie algebra $\fg$. The space of connections $\connspA (P)$ on $P$ is an 
affine space modeled on $\Dform^1 (Y; \ad (P))$, with a distinguished element 
$\trivconn$, which is the trivial connection (associated to the trivialization 
we chose).  Given a connection $B \in \connspA (P)$, let $A$ be the connection 
on the bundle $P \cross \R$ over $Y \cross \R$ that is equal to the pull-back 
of $B$ on $P \cross (-\infty, -1]$ and the pull-back of $\trivconn$ on $P 
\cross [1, \infty)$.  The \emph{Chern--Simons functional of $B$} is defined by 
the Chern--Weil integral
\begin{equation}
  \label{eq:CS}
  \CSt (B) = - \frac{1}{32 \pi^2 \coxeter} \int_{Y \cross \R} \tr (\ad (\curvA) 
  \wedge \ad (\curvA)),
\end{equation}
where $\curvA$ is the $\ad (P)$-valued curvature $2$-form, and $\ad (\curvA)$ 
is the corresponding induced element of $\End (\ad (P))$. The constant 
$\coxeter$ is the \emph{dual Coxeter number}, which depends on $G$; it is equal 
to $N$ when $G = \SU (N)$.
 
Let $\gauge_G$ be the space of smooth maps from $Y$ to $G$. This space can be 
identified with the group of automorphisms of $P$, known as the \emph{gauge 
  group}; in particular, any $g \in \gauge_G$ acts on $\connspA (P)$ by mapping 
a connection $A$ to its pull-back $g^* (A)$. The integral in \eqref{eq:CS} is 
not necessarily invariant with respect to this $\gauge_G$-action; however, it 
always changes by multiples of a fixed constant, and the normalization in 
\eqref{eq:CS} is chosen such that the change in $\CSt$ is always an integer.  
In particular, if we denote by $\connspB (P)$ the quotient of $\connspA (P)$ by 
this $\gauge_G$-action, then \eqref{eq:CS} induces a map $\CS: \connspB (P) \to 
\R / \Z$. An important feature of $\CS$ is that it is a topological function, 
in that its definition does not require a metric on $Y$.

It is not hard to see from the definition that a connection $B$ is a critical 
point of $\CSt$ if and only if $B$ has vanishing curvature, i.e.\ if $B$ is 
flat. Given a flat connection, one may take its holonomy along closed loops in 
$Y$ to obtain a homomorphism $\rho \colon \pi_1 (Y) \to G$, i.e.\ an element of 
the representation variety $\repvarG (Y)$.  This is not necessarily a 
one-to-one correspondence, but if we quotient the space of flat connections by 
the gauge group action and quotient $\repvarG (Y)$ by conjugation, we do get an 
identification of the isomorphism classes of flat connections with the 
character variety $\charvarG (Y)$. In other words, $\charvarG (Y)$ is the set 
of critical points of $\CS$.  Further, the set of critical values of the 
Chern--Simons functional $\CS$ is a finite set, which is a topological 
invariant of $Y$.

In the definition of the Chern--Simons functional, the assumptions on the Lie 
group $G$ are not essential. Indeed, we may take $G$ to be a compact, 
connected, simple Lie group that is possibly not simply connected, with 
universal cover $\Gt$. An important example to keep in mind is when $G = 
\SOthree$ and $\Gt = \SUtwo$.  Instead of a trivial principal bundle, we 
consider a possibly non-trivial principal $G$-bundle on $Y$. We may still form 
the space of connections $\connspA (P)$ as before, and we may form the 
configuration space $\connspB (P)$ by quotienting $\connspA (P)$ by the 
$\gauge_{\Gt}$-action (rather than the $\gauge_G$-action).  There is no longer 
a distinguished element $\trivconn \in \connspA (P)$.  Instead, we arbitrarily 
choose a connection $B_0 \in \connspA (P)$, which plays the role of $\trivconn$ 
in the definitions of $\CSt$; this determines an $\R/\Z$-valued functional 
$\CS$ on $\connspB (P)$ that is well defined up to addition by a constant 
(representing the indeterminacy of the choice of $B_0$).  The critical points 
of $\CS$ are isomorphism classes of flat connections on $P$.  Moreover, the set 
of (relative) values of the Chern--Simons functional at this set of critical 
points is a topological invariant of the pair $(Y, P)$.

\begin{proof}[Proof of {\fullref{cor:chern-simons}}]
  Let $W \colon \Ym \to \Yp$ be a ribbon $\Q$-homology cobordism. Let 
  $\alpha_-$ be a flat connection on $\Ym$, whose holonomy gives an element 
  $\rhom \in \repvar_{G} (\Ym)$.  By \fullref{prop:character-embedding}, we may 
  extend $\rhom$ to an element $\rho_W \in \repvar_{G} (W)$, which pulls back 
  to an element $\rhop \in \repvar_{G} (\Yp)$. We may then choose a 
  corresponding flat connection $\alpha_+$ on $\Yp$.  By Auckly \cite{Auckly}, 
  the Chern--Simons invariants of $\alpha_-$ and $\alpha_+$ agree.
\end{proof}

\subsection{An overview of instanton Floer theory}
\label{ssec:i-intro}

In this section, we review the two main versions of instanton Floer homology 
and develop some properties of the associated cobordism maps. (Other versions 
will be discussed later in this section.)  Throughout, we work only with 
coefficients in $\Q$. We begin with Floer's original version of instanton 
Floer homology \cite{Floer:inst}, which associates to any $\Z$-homology sphere 
$Y$ a $\Z/8$-graded vector space $\Io (Y)$. To a $\Q$-homology cobordism $W 
\colon Y_1 \to Y_2$ of $\Z$-homology spheres, the theory associates a 
homomorphism $\Io (W) \colon \Io (Y_1) \to \Io (Y_2)$ of vector spaces 
\cite{Don:YM-Floer}.\footnote{The homomorphism $\Io (W)$ is also defined for 
  more general cobordisms $W$; see \cite{Don:YM-Floer} for details. We focus on 
  $\Q$-homology cobordisms here for ease of exposition, as this specialization 
  suffices for our purposes.}

The vector space $\Io (Y)$ is the homology of a chain complex $(\Co (Y), d)$.  
The chain complex $\Co (Y)$ is defined roughly as the Morse homology of the 
Chern--Simons functional $\CS$ with the Lie group $G = \SUtwo$ and the trivial 
bundle on $Y$. Recall from \fullref{ssec:i-CS} that the critical set of $\CS$ 
is exactly the space of isomorphism classes of flat connections; in this setup, 
all non-trivial flat connections are irreducible. Here, a connection is 
\emph{irreducible} if its isotropy group is $\set{\pm 1}$; when the connection 
is flat, this is equivalent to the condition that the associated representation 
is irreducible.

In order to achieve Morse--Smale transversality, one perturbs the Chern--Simons 
functional. The critical set of the perturbed Chern--Simons functional still 
contains the trivial connection; the other critical points are no longer 
necessarily flat, but the perturbation can be chosen to be small, which 
guarantees that the non-trivial critical points are still (isomorphism classes 
of) irreducible connections.\footnote{For simplicity, it is customary to blur 
  the line between connections and isomorphism classes of connections (i.e.\ 
  connections up to the gauge group action).  From now on, we will often follow 
  this custom; for example, by an irreducible element of $\fC (Y)$, we will 
  mean an isomorphism class of irreducible connections.} We denote the set of 
all non-trivial critical points by $\fC (Y)$.\footnote{Although it is not 
  reflected in the notation, the set $\fC (Y)$ depends on the choice of 
  perturbation of the Chern--Simons functional.} Then $\Co (Y)$ is the 
$\Q$-vector space generated by the elements of $\fC (Y)$, equipped with the 
differential $d$, where the coefficients $\eval{d (\alpha)}{\beta}$ are given 
by the signed count of index-$1$ gradient flow lines of the perturbation of 
$\CS$ that are asymptotic to $\alpha$ and $\beta$.  A useful observation, which 
is also essential in the development of the analytical aspects of the theory, 
is that the gradient flow lines of (a perturbation of) $\CS$ may be viewed as 
the solutions of (a corresponding perturbation of) the ASD (anti--self-dual) 
equation for the trivial $\SUtwo$-bundle on $Y \cross \R$.

The cobordism map $\Io (W) \colon \Io (Y_1) \to \Io (Y_2)$ is also defined with 
the aid of the ASD equation.  We first attach cylindrical ends to $W$ and fix a 
Riemannian metric on this new manifold, which we also denote by $W$ by abuse of 
notation.  For any pair $(\alpha_1, \alpha_2) \in \fC (Y_1) \cross \fC (Y_2)$, 
we may form a moduli space $\moduli (W; \alpha_1, \alpha_2)$ of connections 
that satisfy a perturbed ASD equation for the trivial $\SUtwo$-bundle on $W$ 
and that are asymptotic to $\alpha_1$ and $\alpha_2$ on the ends. Here, the 
perturbation of the ASD equation is chosen such that it is compatible with the 
perturbations of the Chern--Simons functionals of $Y_1$ and $Y_2$, and 
guarantees that each connected component of $\moduli (W; \alpha_1, \alpha_2)$ 
is a smooth manifold, of possibly different dimensions.  We write $\moduli (W; 
\alpha_1, \alpha_2)_d$ for the union of the $d$-dimensional connected 
components of $\moduli (W; \alpha_1, \alpha_2)$. The value of $d$ mod $8$ is 
determined by $\alpha_1$ and $\alpha_2$. We then define a chain map $\Co (W) 
\colon \Co (Y_1) \to \Co (Y_2)$ by
\begin{equation}
  \label{eq:cob-map-I}
  \Co (W) (\alpha_1) = \sum_{\alpha_2 \in \fC (Y_2)} \# \moduli (W; \alpha_1, 
  \alpha_2)_0 \cdot \alpha_2 \in \Co (Y_2).
\end{equation}
Here, $\# \moduli (W; \alpha_1, \alpha_2)_0$ is the signed count of the 
elements of $\moduli (W; \alpha_1, \alpha_2)_0$. The homomorphism $\Io (W) 
\colon \Io (Y_1)\to \Io (Y_2)$ is the map induced by $\Co (W)$ at the level of 
homology.  It turns out that this map depends only on $W$ and is independent of 
the choice of Riemannian metric on $W$ and perturbation of the ASD equation.

A variation of instanton Floer homology is obtained by replacing the trivial 
$\SUtwo$-bundles with non-trivial $\SOthree$-bundles. Fix a closed $3$-manifold 
$Y$. The isomorphism class of an $\SOthree$-bundle $P$ on $Y$ is determined by 
its second Stiefel--Whitney class $w = w_2 (P) \in H^2 (Y; \zeetwo)$. As 
described in \fullref{ssec:i-CS}, we may define a Chern--Simons functional 
$\CS_w$ on the configuration space $\connspB (P)$ of connections on $P$ up to 
gauge group action.  We say that $(Y, w)$ is an \emph{admissible pair} if the 
pairing of $w$ with $H_2 (Y)$ is not trivial. This condition guarantees that 
the set of critical points of $\CS_w$, or equivalently, the set of flat 
connections on $P$, consists only of irreducible elements of $\connspB (P)$.  
This assumption considerably simplifies the analytical aspects of gauge theory 
and allows us to define an instanton Floer homology $\Io (Y, w)$ for an 
admissible pair, analogous to instanton Floer homology of a $\Z$-homology 
sphere.  As in the previous case, we apply a small perturbation to $\CS_w$ to 
obtain a Morse--Smale functional with the critical set $\fC (Y, w)$. Again, the 
critical points of the perturbed functional are no longer necessarily flat, but 
they remain irreducible.  We define $\Co (Y, w)$ to be the $\Q$-vector space 
generated by $\fC (Y, w)$, equipped with a differential $d$ defined using 
gradient flow lines of the perturbed Chern--Simons functional. 

Instanton Floer homology of admissible pairs is also functorial with respect to 
cobordisms. Let $(Y_1, w_1)$ and $(Y_2, w_2)$ be admissible pairs, let $W 
\colon Y_1 \to Y_2$ be an arbitrary cobordism (i.e.\ not necessarily a 
$\Q$-homology cobordism), and let $c \in H^2 (W; \zeetwo)$ be a cohomology 
class whose restriction to $Y_i$ is equal to $w_i$. The cohomology class $c$ 
determines an $\SOthree$-bundle on $W$, and solutions to a perturbed ASD 
equation for connections on this bundle that are asymptotic to $\alpha_1 \in 
\fC (Y_1, w_1)$ and $\alpha_2 \in \fC (Y_2, w_2)$ give rise to the moduli space 
$\moduli (W, c; \alpha_1, \alpha_2)$. As in the previous case, the perturbation 
of the ASD equation is chosen such that it is compatible with the perturbations 
of the Chern--Simons functionals of $(Y_1, w_1)$ and $(Y_2, w_2)$ and that each 
component of $\moduli (W, c; \alpha_1, \alpha_2)$ a smooth manifold. As in 
\eqref{eq:cob-map-I}, these moduli spaces can be used to define a homomorphism 
$\Io (W, c) \colon \Io (Y_1, w_1) \to \Io (Y_2, w_2)$.  In general, this map is 
defined only up to a sign; this sign can be determined if we fix a homology 
orientation on $W$, which is an orientation of $\topwedge H^1 (W; \Q) \tensor 
\topwedge H^+ (W; \Q) \tensor \topwedge H^1 (Y_2; \Q)$.  Here $H^+ (W; \Q)$ is 
the subspace of $H^2 (W; \Q)$ represented by $L^2$ self-dual harmonic $2$-forms 
on $W$.  (See, for example,  \cite{KM:YAFT} for more details on how to use 
homology orientations to remove the sign ambiguity of $\Io (W, c)$.) In 
particular, for a $\Q$-homology cobordism $W$, there is a canonical choice of 
homology orientation.

There are more general cobordism maps defined for instanton Floer homology of 
admissible pairs. Let $\A (W)$ be the $\Z$-graded algebra $\Sym^* (H_2 (W; \Q) 
\dirsum H_0 (W; \Q)) \tensor \extprod^* (H_1 (W; \Q))$, where the elements in 
$H_i (W; \Q)$ have degree $4 - i$.  For any $z \in \A (W)$ with degree $i$, a 
standard construction gives rise to a cohomology class $\mu (z)$ of degree $i$ 
in $\moduli (W, c; \alpha_1, \alpha_2)_d$, represented by a linear combination 
of submanifolds $V (W, c; \alpha_1, \alpha_2;z)_{d-i}$ of codimension $i$; see, 
for example, \cite[Chapter~5]{DK:book}. Then the homomorphism $\Co (W, c; z) 
\colon \Co (Y_1, w_1) \to \Co (Y_2, w_2)$ defined by
\begin{equation}
  \label{eq:cob-map-I-sharp}
  \Co (W, c; z) (\alpha_1) = \sum_{\alpha_2 \in \fC (Y_2, w_2)} \# V (W, c; 
  \alpha_1, \alpha_2; z)_0 \cdot \alpha_2 \in C (Y_2, w_2)
\end{equation}
is a chain map, and the induced homomorphism $\Io (W, c; z)$ at the level of 
homology is independent of the choice of metric, perturbation, and the 
representative submanifold $V (W,c; \alpha_1, \alpha_2; z)_0$. The homomorphism 
$\Io (W, c; z)$ depends linearly on $z$, and is again defined up to a sign that 
can be fixed using a homology orientation on $W$.  It is also functorial: Let 
$(Y_1, w_1)$, $(Y_2, w_2)$, and $(Y_3, w_3)$ be admissible pairs, $W \colon Y_1 
\to Y_2$ and $W' \colon Y_2 \to Y_3$ be cobordisms equipped with homology 
orientations, and $c_\circ$ be an element of $H^2 (W' \circ W; \zeetwo)$ whose 
restrictions to $W$ and $W'$ are denoted by $c$ and $c'$ respectively, and fix 
$z \in \A (W)$ and $z' \in \A (W')$; then $\Io (W' \circ W,c_\circ; z \cdot 
z')$, defined using the composed homology orientation, is equal to $\Io (W', 
c'; z') \circ \Io (W, c; z)$.

It is natural to ask whether for a cobordism $W$ between $\Z$-homology spheres, 
the definition of the cobordism map $\Io (W)$ can also be extended to a 
homomorphism $\Io (W; z)$ for $z \in \A (W)$. In this context, it would also be 
useful to define $\Io (W; z)$ when $W$ is not a $\Q$-homology cobordism, e.g.\ 
when $b_1 (W) > 0$; to do so, we would also need to make use of homology 
orientations to remove the sign ambiguity.  In general, the main obstruction to 
defining this extension is the existence of reducible ASD connections on $W$: 
One can still define a subspace $V (W; \alpha_1, \alpha_2; z)_0$ of $\moduli 
(W; \alpha_1, \alpha_2)_d$ in the case that $\deg (z) = d$,  but $V (W; 
\alpha_1, \alpha_2; z)_0$ might not be compact because of the existence of 
reducible connections.  Thus one cannot proceed easily, as in 
\eqref{eq:cob-map-I-sharp}, to define $\Io (W; z)$.  In the case that $b^+ (W) 
> 1$, the cobordism map $\Io (W; z) \colon \Io (Y_1) \to \Io (Y_2)$ is defined 
for any $z \in \A (W)$; see \cite[Chapter 6]{Don:YM-Floer}.  For our purposes, 
we need to consider the case where $b^+ (W) = 0$ and the degree of $z$ is 
sufficiently small.  The following compactness result provides the essential 
analytical input to define $\Io (W; z)$ in this context.

\begin{lemma}
  \label{lem:compactness-I}
  Let $Y_1$ and $Y_2$ be $\Z$-homology spheres, and let $\alpha_1 \in \fC 
  (Y_1)$ and $\alpha_2 \in \fC (Y_2)$. Suppose that 
  $W \colon Y_1 \to Y_2$ is a cobordism with $b_1 (W) = m$ and $b^+ (W) = 0$, 
  and that $\set{A_i}_{i=1}^\infty$ is a sequence of connections on $W$ each 
  representing an element of $\moduli (W; \alpha_1, \alpha_2)_d$, where $d \leq 
  3 m + 4$.  Then there are $\alpha_1' \in \fC (Y_1)\cup \{\Theta\}$, $\alpha_2'\in \fC 
  (Y_2)\cup \{\Theta\}$, a finite set of points $\set{p_1, \dotsc, p_\ell} \subset W$, and an 
  irreducible connection $A_0$ on $W$ representing an element of $\moduli (W; 
  \alpha_1', \alpha_2')_{d'}$,
  such that
  \begin{enumerate}
    \item $0 \leq d'\leq d - 8 \ell$; and
    \item after possibly passing to a subsequence and changing each connection
      $A_i$ by an action of the gauge group, the sequence of connections 
      $\set{A_i}$ converges in $\mathcal{C}^\infty$-norm to $A_0$ on any 
      compact subspace of the complement of $\set{p_1, \dotsc, p_\ell}$.  
  \end{enumerate}
\end{lemma}

\begin{proof}
  This is a consequence of the standard compactness theorem for the solutions 
  of the ASD equation on manifolds with cylindrical ends (see, for example, 
  \cite[Chapter~5]{Don:YM-Floer}), together with the following observation.  If 
  the chosen perturbations of the Chern--Simons functionals of $Y_1$ and $Y_2$ 
  and of the ASD equation on $W$ are small enough, then any reducible ASD 
  connection on $W$ is a (singular) element of a moduli space of the form 
  $\moduli (W; \trivconn, \trivconn)_{e}$, where $\trivconn$ is the trivial 
  connection, and $e \geq 3 m - 3$. A straightforward index computation shows that such reducible connections do not appear as limits of a 
  sequence in $\moduli (W; \alpha_1, \alpha_2)_{d}$ when $d \leq 3 m + 4$.
\end{proof}

Suppose that $W$ is a cobordism as in the statement of 
\fullref{lem:compactness-I}. We equip $W$ with a homology orientation by fixing 
an orientation for the vector space $H^1 (W; \Q)$. Suppose also that $z \in 
\A(W)$ has degree at most $3 m+3$.  \fullref{lem:compactness-I} together with a 
standard counting argument shows that the moduli space $V (W; \alpha_1, 
\alpha_2; z)_{0}$ is compact.  Thus we may use a formula similar to  
\eqref{eq:cob-map-I-sharp} to define the cobordism map $\Io (W; z) \colon \Io 
(Y_1) \to \Io (Y_2)$. A standard argument shows that this map is independent of 
the choice of metric, perturbation, and representative submanifold for the 
cohomology class associated to $z$. 

\subsection{Surgery and cobordism maps in instanton Floer theory}
\label{ssec:i-surgery}

We first start with two basic propositions, in which
we will relate certain cobordism maps associated to two cobordisms $X$ and $Z$, 
where $Z$ is the result of surgery on $X$ along a loop $\gamma$.
First, we have a surgery formula for instanton Floer homology of admissible 
pairs.

\begin{proposition}
  \label{prop:surgery-i-adm}
  Let $(Y_1, w_1)$ and $(Y_2, w_2)$ be admissible pairs, and let $X \colon Y_1 
  \to Y_2$ be a cobordism. Suppose that $\gamma \subset \Int (X)$ is a loop 
  with neighborhood $\nbhd{\gamma} \homeo \gamma \cross D^3$, and denote by $Z$ 
  the result of surgery on $X$ along $\gamma$.  Fix a properly embedded surface 
  $S \subset \Int (X)$ supported away from $\nbhd{\gamma}$, such that the 
  cohomology class $c_X \in H^2 (X; \zeetwo)$ dual to $[S]$ restricts to $w_1$ 
  and $w_2$ on $Y_1$ and $Y_2$ respectively, and denote by $c_Z$ the class in 
  $H^2 (Z; \zeetwo)$ determined by $[S]$. Suppose that $z_X \in \A (X)$ admits 
  representatives for its homology classes that are supported away from 
  $\nbhd{\gamma}$, and denote by $z_Z$ the class in $\A (Z)$ determined by 
  these representatives.  Then up to a sign,
  \[
    \Io (X, c_X; [\gamma] \cdot z_X) = \Io (Z, c_Z; z_Z).
  \]
\end{proposition} 

\begin{proof}
  This is essentially \cite[Theorem~7.16]{Don:YM-Floer}, and the same proof 
  works in this set up.
\end{proof}

Similarly, we have a surgery formula for instanton Floer homology of 
$\Z$-homology spheres.

\begin{proposition}
  \label{prop:surgery-i}
  Let $Y_1$ and $Y_2$ be $\Z$-homology spheres, and suppose that $X \colon Y_1 
  \to Y_2$ is a cobordism with $b_1 (X) = m$ and $b^+ (X) = 0$.  Suppose that 
  $\gamma \in \Int (X)$ is a loop with neighborhood $\nbhd{\gamma} \homeo 
  \gamma \cross D^3$, and denote by $Z$ the result of surgery on $X$ along 
  $\gamma$.  Suppose that $z_X \in \A (X)$ has degree at most $3 m - 3$ and 
  admits representatives for its homology classes that are supported away from 
  $\nbhd{\gamma}$, and denote by $z_Z$ the class in $\A (Z)$ determined by 
  these representatives.  Then up to a sign,
  \[
    \Io (X; [\gamma] \cdot z_X) = \Io (Z; z_Z).
  \]
\end{proposition}

\begin{proof}
  This is again essentially \cite[Theorem~7.16]{Don:YM-Floer}.
\end{proof}
 
\begin{remark}
  \label{rmk:surgery-i-adm}
  While we do not provide a proof, we expect that it is possible to remove the 
  sign ambiguities in \fullref{prop:surgery-i-adm} and 
  \fullref{prop:surgery-i}, which would then remove the sign ambiguities in 
  \fullref{thm:main-i-o}, \fullref{thm:main-i-s}, \fullref{thm:main-i-sut}, and 
  \fullref{cor:main-i-knot}.  In the case that $[\gamma] = 0$, both sides of 
  the equation vanish.  In the case that $[\gamma] \neq 0$, we would have to 
  choose homology orientations.  Note that, in this case, $H^+ (X; \Q) \isom 
  H^+ (Z; \Q)$ and $H_1 (X; \Q) \isom \langle [\gamma] \rangle \dirsum H_1 (Z; 
  \Q)$.  Fix a homology orientation $\homor_Z$ on $Z$; we may set $\homor_X = 
  \omega \wedge \homor$, where $\omega \in H^1 (X; \Q)$ is determined by 
  $[\gamma]$.  With this choice, we expect the equations in 
  \fullref{prop:surgery-i-adm} and \fullref{prop:surgery-i} to hold without a 
  sign adjustment.
\end{remark}

We now use the propositions above to study ribbon homology cobordisms. First, 
we verify an analogue of \fullref{thm:main-i-o} for admissible pairs, which we 
will use in the following subsections.

\begin{theorem}
  \label{thm:main-i-adm}
  Let $(\Ym, \wm)$ and $(\Yp, \wpl)$ be admissible pairs, and suppose that $W 
  \colon \Ym \to \Yp$ is a ribbon $\Q$-homology cobordism. Fix a properly 
  embedded surface $S \subset \Int (W)$ supported away from the cocores in a 
  ribbon handle decomposition of $W$, such that the cohomology class $c_W \in 
  H^2 (W; \zeetwo)$ dual to $[S]$ restricts to $\wm$ and $\wpl$ on $\Ym$ and 
  $\Yp$ respectively, and denote by $c_{\double{W}} \in H^2 (\double{W}; 
  \zeetwo)$ and $c_{\Ym \cross I} \in H^2 (\Ym \cross I; \zeetwo)$ the 
  cohomology classes determined by $\double{S}$. Then up to a sign, the 
  cobordism map $\Io (\double{W}, c_{\double{W}}) \colon \Io (\Ym, \wm) \to \Io 
  (\Ym, \wm)$ satisfies
  \[
    \Io (\double{W}, c_{\double{W}}) = \abs{H_1 (W, \Ym)} \cdot \Io (\Ym \cross 
    I, c_{\Ym \cross I}).
  \]
  In particular, if $c_{\Ym \cross I}$ is the pull-back of $\wm$, then up to a 
  sign,
  \[
    \Io (\double{W}, c_{\double{W}}) = \abs{H_1 (W, \Ym)} \cdot \Id_{\Io (\Ym, 
      \wm)},
  \]
  and $\Io (W, c_W)$ includes $\Io (\Ym, \wm)$ into $\Io (\Yp, \wpl)$ as a 
  summand.
\end{theorem}

\begin{proof}
  By \fullref{prop:surgery}, $\double{W}$ is described by surgery on $X \homeo 
  (\Ym \times I) \connsum m (S^1 \times S^3)$ along $m$ disjoint circles 
  $\gamma_1, \dotsc, \gamma_m$, with
  \[
    [\gamma_1] \wedge \dotsb \wedge [\gamma_m] = \det (c_{ij}) \cdot \alpha_1 
    \wedge \dotsb \wedge \alpha_m \in (\extprod^* (H_1 (X) / \Tors) / \langle 
    H_1 (\Ym) / \Tors \rangle) \tensor_{\Z} \Q,
  \]
  where $\alpha_i \in H_1 (X)$ is the homology class of the core of the 
  $i^\text{th}$ $S^1 \cross S^3$ summand, $c_{ij}$ is the multiplicity of 
  $\alpha_i$ in $[\gamma_j]$, and $\lvert \det (c_{ij}) \rvert = \abs{H_1 (W, 
    \Ym)}$.  Applying \fullref{prop:surgery-i-adm} with $Z = \double{W}$, we 
  have that, up to a sign,
  \[
    \Io (X, c_X; [\gamma_1] \wedge \dotsb \wedge [\gamma_m]) = \Io (\double{W}, 
    c_{\double{W}}).
  \]
  We claim that
	\begin{equation}
    \label{eq:term-type-1}
    \Io (X, c_X; [\gamma_1] \wedge \dotsb \wedge [\gamma_m]) = \det (c_{ij}) 
    \cdot \Io (X, c_X; \alpha_1 \wedge \dotsb \wedge \alpha_m),
  \end{equation}
  where we are using $\homor_{X, \gamma}$ on both sides of the equation; 
  indeed, by the linearity of $\Io$, it suffices to show that $\Io (X, c_X; 
  \xi) = 0$ for $\xi \in \Lambda^m (H_1 (X; \Q)) \intersect \ideal{H_1 (\Ym; 
    \Q)}$.  To see this, we may apply \fullref{prop:surgery-i-adm} in the 
  opposite direction to see that $\Io (X, c_X; \xi) = \Io (Z', c_{Z'})$ for 
  some cobordism $Z'$ with at least one $S^1 \cross S^3$ connected summand; the 
  general vanishing theorem for connected sums implies that this map is zero.  
  (The interested reader may compare this argument with the penultimate 
  paragraph of the proof of \fullref{thm:main-hf} in 
  \fullref{ssec:hf-surgery}.)
  Note that $\det (c_{ij}) = \abs{H_1 (W, \Ym)}$ up to a sign.

  Applying \fullref{prop:surgery-i-adm} again with $Z = \Ym \cross I$, we see 
  that up to a sign,
  \[
    \Io (X, c_X; \alpha_1 \wedge \dotsb \wedge \alpha_m) = \Io (\Ym \cross I, 
    c_{\Ym \cross I}).
  \]
  This completes our proof.
\end{proof}

Similarly, we prove \fullref{thm:main-i-o}.

\begin{proof}[Proof of \fullref{thm:main-i-o}]
  The proof is completely analogous to that of \fullref{thm:main-i-adm}, 
  without the need to keep track of the cohomology classes or consider elements 
  of $H_1 (\Ym)$.
\end{proof}

\begin{proof}[Proof of {\fullref{cor:casson}}]
  A standard gluing argument shows that the signed count of elements in the 
  moduli space of index-$0$ (perturbed) ASD connections on $\DWbar$ is equal to 
  $2 \Lef (\Io (\double{W} \colon \Io (\Ym) \to \Io (\Ym))$. (See \cite[Theorem 
  6.7]{Don:YM-Floer} for a similar gluing result.) By definition, the former 
  count is equal to $4\lFO (\DWbar) $, and by \fullref{thm:main-i-o}, the 
  Lefschetz number $\Lef (\Io (\double{W})$ is the Euler characteristic of $\Io 
  (\Ym)$, which is precisely twice the Casson invariant of $\Ym$.
\end{proof}

\subsection{Framed instanton Floer theory}
\label{ssec:i-fr}

Instanton Floer homology of admissible pairs can be used to define a 3-manifold 
invariant called framed instanton Floer homology \cite{KM:YAFT}. First, by a 
\emph{framed manifold}, we mean a closed $3$-manifold with a framed basepoint.  
Fix $(T^3, u)$ to be the admissible pair of the $3$-dimensional torus and the 
element of $H^2 (T^3; \zeetwo)$ given by the dual of $S^1 \cross \set{q} 
\subset T^3$ for some point $q \in T^2$.  Let $Y$ be a framed manifold with a 
framed basepoint $p \in Y$.  Then define $\Ys$ to be $Y \connsum T^3$, where 
the connected sum takes place in a neighborhood of $p$, and let $\ws \in H^2 
(\Ys; \zeetwo)$ be the class induced by the trivial class in $Y$ and $u$ in 
$T^3$.  Let $x \in \A (\Ys \cross I)$ be the class of degree $4$ determined by 
the homology class of a point in $\Ys \cross I$. The operator $\mu (x) = \Io 
(\Ys \cross I, \pi_1^* (\ws); x)$ acts on the $\Z/8$-graded vector space $\Io 
(\Ys, \ws)$, and satisfies $\mu (x)^2 = 4 \cdot \Id_{\Io (\Ys, \ws)}$ 
\cite[Corollary 7.2]{KM:sutured}. The \emph{framed instanton Floer homology} of 
$Y$, denoted by $\Is (Y)$, is defined to be the kernel of $\mu(x) - 2$; it 
inherits a $\Z/4$-grading from $\Io (\Ys, \ws)$.  This flavor of instanton 
Floer homology is conjectured to agree with the hat flavor of Heegaard Floer 
homology, when both are computed over $\Q$.

Framed instanton Floer homology is functorial with respect to cobordisms of 
framed manifolds. Given framed $3$-manifolds $Y_1$ and $Y_2$ with framed 
basepoints $p_1$ and $p_2$ respectively, a \emph{framed cobordism} $W \colon 
Y_1 \to Y_2$ is a cobordism together with a choice of an embedded framed path 
in $W$ between $p_1$ and $p_2$. A framed cobordism $W \colon Y_1 \to Y_2$ can 
be used to define a cobordism $\Ws \colon \Ys_1 \to \Ys_2$ by taking the sum 
with $T^3 \cross I$ along a regular neighborhood of the framed path in $W$.  A 
homology orientation on $W$ induces a homology orientation on $\Ws$ in an 
obvious way.  Moreover, the dual of $S^1 \cross \set{q} \cross I \subset T^3 
\cross I$ defines a cohomology class $c \in H^2 (\Ws; \zeetwo)$ that restricts 
to $\ws_1$ and $\ws_2$ on $\Ys_1$ and $\Ys_2$ respectively.  The functoriality 
of instanton Floer homology of admissible pairs implies that
\[
  \Io (\Ws, c) \comp \Io (\Ys_1 \times I, \pi_1^* (\ws_1); x_1) = \Io (\Ys_2 
  \times I, \pi_1^* (\ws_2); x_2) \comp \Io (\Ws, c).
\]
In particular, $\Io (\Ws, c)$ gives rise to a homomorphism $\Is (W) \colon \Is 
(Y_1) \to \Is (Y_2)$.

\begin{proof}[Proof of \fullref{thm:main-i-s}]
  Let $W \colon \Ym \to \Yp$ be a ribbon $\Q$-homology cobordism of framed 
  $3$-manifolds.  We also denote by $\wm^\sharp$ and $\wpl^\sharp$ the 
  cohomology classes in $\Ym^\sharp$ and $\Yp^\sharp$ induced by $u$ 
  respectively.  Then $(\Ym^\sharp, \wm^\sharp)$, $(\Yp^\sharp, \wpl^\sharp)$, 
  $W^\sharp$, and $S^1 \cross \set{q} \cross I \subset W^\sharp$ satisfy the 
  conditions of \fullref{thm:main-i-adm}, and we can thus apply it to conclude 
  that, up to a sign, $\Io (\double{W}^\sharp, c)$ is equal to multiplication 
  by $\abs{H_1 (W,\Ym)}$. Since this map clearly respects the eigenspace 
  decomposition of $\mu (x)$, we obtain the analogous statement for $\Is 
  (\double{W})$.
\end{proof}

\subsection{Sutured instanton Floer theory}
\label{ssec:i-sut}

We first define what we mean by a cobordism of sutured manifolds. Note that 
this definition is narrower than the one used by Juh\'asz \cite{Juh16}.

\begin{definition}
  \label{defn:sut-cob}
  Let $(M_1, \sut_1)$ and $(M_2, \sut_2)$ be sutured manifolds. A 
  \emph{cobordism} $N \colon (M_1, \sut_1) \to (M_2, \sut_2)$ is a $4$-manifold 
  $N$ obtained by a sequence of interior handle attachments on $M_1 \times I$.  
  In particular, there is a natural diffeomorphism of $\bdy M_1$ and $\bdy M_2$ 
  that identifies $\sut_1$ with $\sut_2$.
\end{definition}

If $Y$ is a framed $3$-manifold, then we can define a sutured manifold $(M, 
\sut)$, where $M$ is the complement of a regular neighborhood of the basepoint 
diffeomorphic to the $3$-ball, and $\alpha$ is the equator in $\bdy M$. A 
framed cobordism $W \colon Y_1 \to Y_2$ of framed $3$-manifolds then induces a 
cobordism of the sutured manifolds associated to $Y_1$ and $Y_2$.

More generally, the theory of instanton Floer homology of admissible pairs can 
be also used to define a functorial invariant of sutured manifolds, 
generalizing the framed instanton Floer construction. Instanton homology of 
sutured manifolds is defined using \emph{closures} of sutured manifolds, which 
we now recall.

Let $(M, \sut)$ be a sutured manifold whose set of sutures $\sut$ has $d$ 
elements.  Denote by $F_{g, d}$ the genus-$g$ surface with $d$ boundary 
components. Fix an arbitrary $g \geq 0$; we glue $(M, \sut)$ to the product 
sutured manifold $F_{g, d} \cross [-1, 1]$ by identifying $A (\sut)$ with 
$(\bdy F_{g, d}) \cross [-1, 1]$. The resulting space has two boundary 
components $\Rh_\pm \homeo R_\pm (\sut) \union (F_{g, d} \cross \set{\pm 1})$, 
which are closed surfaces of the same genus; we choose a diffeomorphism $\phi$ 
of these two boundary components that fixes some point $p \in F_{g, d}$, and 
glue $\Rh_{\pm}$ together by $\phi$ to obtain a closed $3$-manifold $\Mh$. Then 
$\set{p} \cross [-1, 1] \subset F_{g, d} \cross [-1,1]$ determines a closed 
curve in $\Mh$, and we write $w \in H^2 (\Mh; \zeetwo)$ for its Poincar\'e 
dual.  The image of $\Rh_\pm$ gives rise to an embedded oriented surface $R$ of 
a certain genus $g'$ in $\Mh$ with $g' \geq g$, and $(\Mh, w)$ is an admissible 
pair because the pairing of $w$ with $R$ is not trivial. At this point, we 
require $g' \geq 1$; this could be ensured by the sufficient (but not 
necessary) condition that we choose $g \geq 1$.  Then, $R$ induces an 
endomorphism
\[
  \mu (R) = \Io (\Mh \cross I, \pi_1^* (w); R) \colon \Io (\Mh, w) \to \Io 
  (\Mh, w).
\]
If $g' > 1$, then the instanton homology of $(M, \sut)$ is defined by
\[
  \SHI (M, \sut) = \ker (\mu (R) - (2g' - 2))
\]
In the case that $g' = 1$, the operator $\mu (R)$ acts trivially and the 
definition of $\SHI (M, \sut)$ should be modified using the operator $\mu (x) = 
\Io (\Mh \cross I, \pi_1^* (w); x)$, where $x \in \A (\Mh \cross I)$ is the 
class given by a point. Thus, if $g' = 1$, we define
\[
  \SHI (M, \sut) = \ker (\mu (x) - 2)
\]
In any case, the key fact is that this construction of $\SHI (M, \sut)$ above 
is independent of all choices made in the process. (The interested reader may 
compare the above with the proof of \fullref{thm:main-hf-sut} in 
\fullref{ssec:hf-sut}, in the context of sutured Heegaard Floer theory.)

We also have an analogous construction for a cobordism of sutured manifolds $N 
\colon (M_1, \sut_1) \to (M_2, \sut_2)$. First, fix $g \geq 0$, and glue the 
product of an interval and the product sutured manifold $F_{g,d} \cross [-1, 
1]$ to $N$ to obtain a cobordism of manifolds with boundary, where the induced 
cobordism of the boundary components is the trivial cobordism $(\Rh_+ \cross I) 
\disjunion (\Rh_- \cross I)$ to itself.  Using the diffeomorphism $\phi$ of 
$\Rh_+$ and $\Rh_-$, we identify $\Rh_+ \cross I$ with $\Rh_- \cross I$ to 
obtain a cobordism $\Nh$ from a closure $\Mh_1$ of $(M_1, \sut_1)$ to a closure 
$\Mh_2$ of $(M_2, \sut_2)$. (As before, we require that the image $R$ of 
$\Rh_\pm$ has genus $g' \geq 1$.) Also, the product of an interval and $\set{p} 
\cross [-1, 1]$ determines a properly embedded surface in $\Nh$, whose 
Poincar\'e dual $c \in H^2 (\Nh; \zeetwo)$ restricts to $w_i \in H^2 (\Mh_i; 
\zeetwo)$ for $i \in \set{1, 2}$. Thus, we obtain a cobordism map $\Io (\Nh, c) 
\colon \Io (\Mh_1, w_1) \to \Io (\Mh_2, w_2)$ of admissible pairs.  It turns 
out that $\Io (\Nh, c)$ respects the eigenspace decompositions of $\Io (\Mh_1, 
w_1)$ and $\Io (\Mh_2, w_2)$, and so we obtain a homomorphism $\SHI (N) \colon 
\SHI (M_1, \sut_1) \to \SHI (M_2, \sut_2)$ simply by restricting to the 
($+2$)-eigenspace.

\begin{proof}[Proof of \fullref{thm:main-i-sut}]
  This follows directly from \fullref{thm:main-i-adm} together with the 
  description of sutured instanton Floer homology as the eigenspace of the 
  instanton Floer homology for an admissible pair.
\end{proof}

The sutured instanton homology of the sutured manifold associated to a framed 
$3$-manifold $Y$ is isomorphic to $\Is (Y)$. In fact, the manifold $\Ys$ can be 
obtained as a closure of the sutured manifold associated to $Y$, where we use 
the product sutured manifold $F_{1,1} \cross I$ in the construction of the 
closure.

\begin{proof}[Proof of \fullref{cor:main-i-knot}]
  The main idea of this proof is that the known isomorphism between $\KHI$ and 
  $\In$ is natural with respect to cobordism maps. To simplify the exposition, 
  we focus on the cobordism maps associated to $(\double{W}, \double{C})$ 
  below.

  To make this precise, we first recall an explicit description of $\KHI (Y, 
  K)$, as contained in \cite[Section~5.1 and Section~7.6]{KM:sutured}.
  Let $K$ be a knot in a closed, oriented $3$-manifold $Y$; first, we associate 
  to the pair $(Y, K)$ the sutured manifold $(Y \setminus \nbhd{K}, \sut)$, 
  where $Y \setminus \nbhd{K}$ is the exterior of $Y$, and $\sut \subset \bdy 
  (Y \setminus \nbhd{K})$ consists of two sutures that are oppositely oriented 
  meridians. Then $\KHI (Y, K)$ is defined as $\SHI (Y \setminus \nbhd{K}, 
  \sut)$. As described earlier in this subsection, $\SHI (Y \setminus \nbhd{K}, 
  \sut)$ is in turn defined by taking a closure; we choose to work with the 
  closure associated to $F_{0,2}$, the genus-$0$ surface with $2$ boundary 
  components, and denote this closure $\Mh$ by $\knotclos_K$.
  According to \cite[Section~5.1]{KM:sutured}, the closed $3$-manifold 
  $\knotclos_K$ admits an equivalent description. It is formed by gluing 
  $F_{1,1} \cross S^1$ to $Y \setminus \nbhd{K}$, with $\bdy F_{1,1} \cross 
  \set{p}$ being identified with a longitude of $K$ on $\bdy \nbhd{K}$, and 
  $\set{q} \cross S^1$ being identified with the meridian of $K$ on $\bdy 
  \nbhd{K}$.  In this new description, the element $w \in H^2 (\knotclos_K; 
  \zeetwo)$ is the Poincar\'e dual of the oriented loop $\gamma \cross \set{p} 
  \subset F_{1,1} \cross S^1 \subset \knotclos_K$, where $\gamma \subset 
  F_{1,1}$ is some oriented, non-separating loop.  The embedded oriented 
  surface $R$ is then $\gamma' \cross S^1 \subset F_{1,1} \cross S^1 \subset 
  \knotclos_K$, where $\gamma'$ is another non-separating loop in $F_{1,1}$ 
  that intersects $\gamma$ at exactly one point. This, in particular, means 
  that $R$ has genus $g' = 1$, and so
  \[
    \KHI (Y, K) = \SHI (Y \setminus \nbhd{K}, \sut) = \ker (\mu (x) - 2),
  \]
  where $\mu (x) \colon \Io (\knotclos_K, w) \to \Io (\knotclos_K, w)$ is a 
  degree-$4$ operator determined by a point $x \in \knotclos_K$. By 
  \cite[Corollary~7.2]{KM:sutured}, one can see that $\mu (x)$ has eigenvalues 
  $\pm 2$ (so that $\mu (x)^2 = 4 \cdot \Id_{\Io (\knotclos_K, w)}$), each of 
  whose eigenspace has half the dimension of $\Io (\knotclos_K, w)$.  In 
  particular, one concludes that the dimension of $\KHI (Y, K)$ is half that of 
  $\Io (\knotclos_K, w)$.

  Next, we recall the isomorphism between $\KHI$ and $\In$. In 
  \cite[Section~5]{KM:unknot}, a degree-$4$ involution $\psi_K \colon \Io 
  (\knotclos_K, w) \to \Io (\knotclos_K, w)$ is constructed, whose associated 
  quotient is denoted $\Io (\knotclos_K, w)^\psi$; then, using a version of 
  Floer's Excision Theorem, it is shown that there is an isomorphism 
  $\Phi_{(Y,K)} \colon \In (Y, K) \to \Io (\knotclos_K, w)^\psi$. From this, 
  one again concludes that the dimension of $\In (Y, K)$ is half that of $\Io 
  (\knotclos_K, w)$, and thus that $\In (Y, K)$ is isomorphic to $\KHI (Y, K)$.

  Let $(\Ypm, \Kpm)$ and $(W, C) \colon (\Ym, \Km) \to (\Yp, \Kp)$ be as in the 
  statement. We now argue that the isomorphism between $\KHI$ and $\In$ is 
  natural with respect to cobordism maps associated to $(\double{W}, 
  \double{C})$. To begin, let $N \colon (\Ym \setminus \nbhd{\Km}, \sutm) \to 
  (\Yp \setminus \nbhd{\Kp}, \sutp)$ be the cobordism of sutured manifolds, in 
  the sense of \fullref{defn:sut-cob}, obtained by removing a regular 
  neighborhood of $C$ from $W$; obviously, $N$ is a ribbon $\Q$-homology 
  cobordism. Then, the cobordism map
  \[
    \KHI (\double{W}, \double{C}) \colon \KHI (\Ym, \Km) \to \KHI (\Ym, \Km)
  \]
  is defined as the cobordism map
  \[
    \SHI (\double{N}) \colon \SHI (\Ym \setminus \nbhd{\Km}, \sutm) \to \SHI 
    (\Ym \setminus \nbhd{\Km}, \sutm).
  \]
  By \fullref{thm:main-i-sut}, we know that, up to a sign,
  \[
    \SHI (\double{N}) = \abs{H_1 (N, \Ym \setminus \nbhd{\Km})} \cdot \Id_{\SHI 
      (\Ym \setminus \nbhd{\Km}, \sutm)} = \abs{H_1 (W, \Ym)} \cdot \Id_{\SHI 
      (\Ym \setminus \nbhd{\Km}, \sutm)},
  \]
  which in particular implies that it is a degree-$0$ map. Passing to the 
  closure, this homomorphism is in turn induced by a cobordism map
  \[
    \Io (\double{\Nh}, c) \colon \Io (\knotclos_{\Km}, \wm) \to \Io 
    (\knotclos_{\Km}, \wm)
  \]
  of admissible pairs that commutes with $\mu (x_-) \colon \Io 
  (\knotclos_{\Km}, \wm) \to \Io (\knotclos_{\Km}, \wm)$, where $x_- \in 
  \knotclos_{\Km}$. In particular, $\SHI (\double{N})$ is the restriction of 
  $\Io (\double{\Nh}, c)$ to the ($+2$)-eigenspace of $\mu (x_-)$. Taking into 
  account the facts that $\Io (\knotclos_{\Km}, \wm)$ is a $\Z / 8$-graded 
  vector space and that $\mu (x_-)$ is a degree-$4$ map, we conclude that $\Io 
  (\double{\Nh}, c)$ itself must satisfy
  \[
    \Io (\double{\Nh}, c) = \abs{H_1 (W, \Ym)} \cdot \Id_{\Io (\knotclos_{\Km}, 
      \wm)}
  \]
  up to a sign.

  Now $\Io (\double{\Nh}, c)$ commutes with the degree-$4$ involution 
  $\psi_{\Km} \colon \Io (\knotclos_{\Km}, \wm) \to \Io (\knotclos_{\Km}, \wm)$ 
  \cite[Section~5]{KM:unknot}, and thus induces a map
  \[
    \Io (\double{\Nh}, c)^\psi \colon \Io (\knotclos_{\Km}, \wm)^\psi \to \Io 
    (\knotclos_{\Km}, \wm)^\psi
  \]
  on the quotients. Clearly, this must also satisfy
  \[
    \Io (\double{\Nh}, c)^\psi = \abs{H_1 (W, \Ym)} \cdot \Id_{\Io 
      (\knotclos_{\Km}, \wm)^\psi}
  \]
  up to a sign. Finally, we claim that $\Phi_{(\Ym, \Km)} \colon \In (\Ym, \Km) 
  \to \Io (\knotclos_{\Km}, \wm)^\psi$ intertwines $\Io (\double{\Nh}, c)^\psi$ 
  with $\In (\double{W}, \double{C}) \colon \In (\Ym, \Km) \to \In (\Ym, \Km)$:
  \[
    \Phi_{(\Ym, \Km)} \comp \Io (\double{\Nh}, c)^\psi = \In (\double{W}, 
    \double{C}) \comp \Phi_{(\Ym, \Km)}.
  \]
  Indeed, this claim follows from the fact that the excision map $\Phi_{(\Ym, 
    \Km)}$ is itself a cobordism map, meaning that the two sides of the 
  identity above can be interpreted as two homomorphisms associated to 
  diffeomorphic cobordisms. This implies that the desired result that
  \[
    \In (\Ym, \Km) = \abs{H_1 (W, \Ym)} \cdot \Id_{\In (\Ym, \Km)}
  \]
  holds, up to a sign.
\end{proof}

\subsection{Equivariant instanton Floer theory}
\label{ssec:i-eq}

For a $\Z$-homology sphere $Y$, one can define a stronger invariant that 
contains the information of $\Io (Y)$ and $\Is (Y)$. Let $(\Co (Y), d)$ be the 
instanton Floer chain complex whose homology is equal to $\Io (Y)$. We consider 
a larger chain complex $(\Cto (Y), \dto)$ defined by $\Cto (Y) = \Co (Y) 
\dirsum \Q \dirsum \Co (Y) [3]$, where $\Co (Y) [3]$ denotes the complex $\Co 
(Y)$ with the $\Z/8$-grading shifted up by $3$. The complex $\Cto (Y)$ is 
equipped with a  $\Z/8$-grading on $\widetilde{C}(Y)$ by assigning degree $0$ 
to the summand $\Q$. With respect to the direct sum decomposition of $\Cto (Y)$ 
above, the differential $\dto$, which has degree $-1$, has the matrix form
\begin{equation}
  \label{eq:d-tilde}
  \dto =
  \begin{pmatrix}
    d & 0 & 0\\
    D_1 & 0 & 0\\
    U & D_2 & -d
  \end{pmatrix},
\end{equation}
where $U \colon \Co (Y) \to \Co (Y) [4]$ is a degree-preserving map, $D_1$ is a 
functional on $\Co (Y)$ that is not zero only on elements of degree $1$, and 
$D_2 (1)$ is a degree-$4$ element in $\Co (Y)$. We refer the reader to 
\cite{Don:YM-Floer,Fro:h-inv} for more details on the definition of $U$, $D_1$, 
and $D_2$. Here we use the same conventions as in \cite{AD:CS-Th}, where an 
exposition of the definition of $(\Cto (Y), \dto)$ is given. The characterizing 
feature of the special form of $\dto$ in \eqref{eq:d-tilde} is that it 
anti-commutes with the endomorphism of $\Cto (Y)$ given by
\[
  \chi =
  \begin{pmatrix}
    0 & 0 & 0\\
    0 & 0 & 0\\
    1 & 0 & 0
  \end{pmatrix}.
\]
We call a chain complex $(\widetilde{C}, \widetilde d)$ over $\Q$ whose 
differential has the form in \eqref{eq:d-tilde} an 
\emph{$\cSO$-complex}.\footnote{For a topological space with an 
  $\SOthree$-action that has a unique fixed point, one can form an 
  $\cSO$-complex whose homology is the homology of the space. This justifies 
  the terminology $\cSO$-complex.} 

The chain complex $(\Cto (Y), \dto)$ depends on some auxiliary choices, namely 
the Riemannian metric on $Y$ and the perturbation of the Chern--Simons 
functional of $Y$. However, the chain homotopy type of $(\Cto (Y), \dto)$ is an 
invariant of $Y$ in an appropriate sense. Suppose $(\Cto{}' (Y), \dto')$ is the 
chain complex that is obtained from another set of auxiliary choices.  Then 
there is a degree-$0$ chain map $\lambdat \colon \Cto (Y) \to \Cto{}' (Y)$, 
such that
\begin{equation}
  \label{eq:C-tilde}
  \lambdat =
  \begin{pmatrix}
    \lambda & 0 & 0\\
    \Delta_1 & 1 & 0\\
    \mu & \Delta_2 & \lambda\\
  \end{pmatrix}.
\end{equation}
Notice that the map $\lambdat$ commutes with $\chi$.  Similarly, there is a 
degree-$0$ chain map $\lambdat' \colon \Cto{}' (Y) \to \Cto (Y)$ with the same 
form as \eqref{eq:C-tilde}, with $\lambda'$ playing the role of $\lambda$.  
Moreover, there are degree-$1$ maps $\Kt \colon \Cto (Y) \to \Cto (Y)$ and 
$\Kt' \colon \Cto{}' (Y) \to \Cto{}' (Y)$ that anti-commute with $\chi$, such 
that
\[
  \Kt \comp \dto + \dto \comp \Kt = \lambdat' \comp \lambdat - \Id_{\Cto (Y)}, 
  \qquad
  \Kt' \comp \dto' + \dto' \comp \Kt' = \lambdat \comp \lambdat' - \Id_{\Cto{}' 
    (Y)}.
\]
As is customary in Floer theories, the existence of the maps $\lambdat$ and 
$\lambdat'$ is a consequence of a more general functoriality of the theory.  In 
fact, for any $\Z$-homology cobordism $W \colon Y_1 \to Y_2$, there is a chain 
map $\lambdat (W) \colon \Cto (Y_1) \to \Cto (Y_2)$ of the form in 
\eqref{eq:C-tilde}.  In particular, this morphism contains in its data a chain 
map $\lambda (W) \colon \Co (Y_1) \to \Co (Y_2)$, which induces the cobordism 
map $\Io (W) \colon \Io (Y_1) \to \Io (Y_2)$ on the level of homology. 

For general $\cSO$-complexes, a chain map of the form \eqref{eq:C-tilde}, with 
the number $1$ possibly replaced by a non-zero rational number, is called an 
\emph{$\cSO$-morphism}. An \emph{$\cSO$-homotopy} from an $\cSO$-morphism 
$\lambdat_1$ to another $\cSO$-morphism $\lambdat_2$ is given by a map $\Kt$ of 
degree $1$ that anti-commutes with $\chi$, such that
\[
  \Kt \comp \dto + \dto \comp \Kt = \lambdat_2 - \lambdat_1,
\]
and we say that two $\cSO$-complexes $\widetilde{C}$ and $\widetilde{C}'$ are 
\emph{$\cSO$-homotopy equivalent} if there are $\cSO$-morphisms $\lambdat 
\colon \widetilde{C} \to \widetilde{C}'$ and $\lambdat' \colon \widetilde{C}' 
\to \widetilde{C}$ such that $\lambdat' \comp \lambdat$ and $\lambdat \comp 
\lambdat'$ are $\cSO$-homotopic to identity maps. In other words, the 
discussion above shows that the $\cSO$-homotopy type of $(\Cto (Y), \dto)$ is 
an invariant of $Y$.

The $\cSO$-homotopy type of the complex $(\Cto (Y), \dto)$ contains the 
information of the instanton homology groups $\Io (Y)$ and $\Is (Y)$. It is 
clear from the definition that $\Io (Y)$ is the homology of the quotient complex 
$(\Co (Y), d)$, whose chain homotopy type can be recovered from the 
$\cSO$-homotopy type of $(\Cto (Y), \dto)$.  The homology of the chain complex 
$(\Cto (Y), \dto + 4 \chi)$ is also isomorphic to $\Is (Y)$ \cite{Sca:sur}.  

One could extract from $(\Cto (Y), \dto)$ several other homologies, which are 
analogous to $\HFm$, $\HFp$, and $\HFi$ in Heegaard Floer theory respectively.  
Following \cite{Don:YM-Floer, AD:CS-Th}, consider the $\Z/8$-graded chain 
complexes $(\Cf (Y), \df)$ and $(\Ct (Y), \dt)$ defined by
\begin{align*}
  \Cf (Y) & = \Co (Y) [3] \dirsum \Q [x], & \df \paren{\alpha, \sum_{i=0}^N a_i 
    x^i} & = \paren{d \alpha - \sum_{i=0}^{N} U^i D_2 (a_i), 0},\\
  \Ct (Y) & = \Co (Y) \dirsum (\Qlocpoly / \Q [x]), & \dt \paren{\alpha, 
    \sum_{i=-\infty}^{-1} a_i x^i} & = \paren{d \alpha, \sum_{i=-\infty}^{-1} 
    D_1 U^{-i-1} (\alpha) x^i}.
\end{align*}
Here, the degree of $x$ is defined to be $-4$. The homology of $(\Cf (Y), \df)$ 
and $(\Ct (Y), \dt)$ are denoted by $\If (Y)$ and $\It (Y)$ respectively.   
They are modules over the polynomial ring $\Q [x]$, with the action of $x$ 
given by the endomorphisms
\begin{align*}
  x & \colon \Cf (Y) \to \Cf (Y), & x \cdot \paren{\alpha, \sum_{i=0}^N a_i 
    x^i} & = \paren{U \alpha, D_1 (\alpha) + \sum_{i=0}^{N} a_i x^{i+1}},\\
  x & \colon \Ct (Y) \to \Ct (Y), & x \cdot \paren{\alpha, 
    \sum_{i=-\infty}^{-1} a_i x^i} & = \paren{U \alpha + D_2 (a_{-1}), 
    \sum_{i=-\infty}^{-2} a_i x^{i+1}}.
\end{align*}
We define $(\Cb (Y), \db)$ to be the $\Q [x]$-module $\Qlocpoly$ with the 
trivial differential; although it is independent of $Y$, it is convenient to 
consider it and its homology $\Ib (Y) \isom \Qlocpoly$. Together, $\If (Y)$, 
$\It (Y)$, and $\Ib (Y)$ are called the \emph{equivariant instanton Floer 
  homologies} of $Y$.

The modules $\If (Y)$, $\It (Y)$ and $\Ib (Y)$ fit into an exact triangle
\begin{equation}
  \label{eq:ex-triangle}
  \xymatrix{
    \It (Y) \ar[rr]^{j_*} & & \If (Y) \ar[dl]^{ i_*}\\
    & \Ib (Y) \ar[ul]^{p_*}
  }
\end{equation}
where the module homomorphisms are induced by the maps
\begin{align*}
  i & \colon \Cf (Y) \to \Cb (Y), & i \paren{\alpha, \sum_{i=0}^{N} a_i x^i} & 
  = \sum_{i=-\infty}^{-1} D_1 U^{-i-1} (\alpha) x^{i} + \sum_{i=0}^{N} a_i 
  x^i,\\
  j & \colon \Ct (Y) \to \Cf (Y), & j \paren{\alpha, \sum_{i=-\infty}^{-1} a_i 
    x^{i}} & = (-\alpha, 0),\\
  p & \colon \Cb (Y) \to \Ct (Y), & p \paren{\sum_{i=-\infty}^{N} a_i x^i} & = 
  \paren{\sum_{i=0}^{N} U^{i} D_2 (a_i), \sum_{i=-\infty}^{-1} a_i x^{i}}.
\end{align*}

As is apparent from the definitions, the construction of the equivariant 
instanton homologies and the exact triangle \eqref{eq:ex-triangle} from $(\Cto 
(Y), \dto)$ is completely algebraic and does not require any additional 
geometric input. In particular, for any $\cSO$-complex $(\widetilde{C}, 
\widetilde{d})$, one can define the chain complexes $(\widehat{C}, \df)$, 
$(\widecheck{C}, \dt)$, their homologies $\widehat{I}$, $\widecheck{I}$, and 
the analogue of the exact triangle \eqref{eq:ex-triangle}.  These constructions 
are functorial; given an $\cSO$-morphism $\lambdat \colon \widetilde{C} \to 
\widetilde{C}'$, there are corresponding chain maps $\widehat{\lambda} \colon 
\widehat{C} \to \widehat{C}'$, $\widecheck{\lambda} \colon \widecheck{C} \to 
\widecheck{C}'$, and $\overline{\lambda} \colon \overline{C} \to 
\overline{C}'$, which induce module homomorphisms $\widehat{\lambda}_* \colon 
\widehat{I} \to \widehat{I}'$, $\widecheck{\lambda}_* \colon \widecheck{I} \to 
\widecheck{I}'$, and $\overline{\lambda}_* \colon \overline{I} \to 
\overline{I}'$ that commute with the exact triangles associated to 
$\widetilde{C}$ and $\widetilde{C}'$, as explained in 
\cite[Section~2.3]{AD:CS-Th}.  An $\cSO$-homotopy between two $\cSO$-morphisms 
$\widetilde{\lambda}_1$ and $\widetilde{\lambda}_2$ induces a homotopy between 
$\widehat{\lambda}_1$ and $\widehat{\lambda}_2$, a homotopy between 
$\widecheck{\lambda}_1$ and $\widecheck{\lambda}_2$, and a homotopy between 
$\overline{\lambda}_1$ and $\overline{\lambda}_2$.  Moreover, the maps 
corresponding to the composition $\widetilde{\lambda}' \comp 
\widetilde{\lambda}$ of two $\cSO$-morphisms are equal to the compositions of 
the maps corresponding to $\widetilde{\lambda}'$ and $\widetilde{\lambda}$.  As 
a consequence of this functoriality, the equivariant instanton homologies $\If 
(Y)$, $\It (Y)$, $\Ib(Y)$ and the exact triangle \eqref{eq:ex-triangle} are 
invariants of $Y$, and do not depend on the auxiliary choices in the definition 
of $(\widetilde{C} (Y), \widetilde{d})$.

We now turn our attention to the behavior of equivariant instanton Floer 
homologies under ribbon $\Q$-homology cobordisms. (Recall from 
\fullref{rmk:zhs} that $\Q$-homology cobordisms between $\Z$-homology spheres 
are in fact $\Z$-homology cobordisms.) The key statement is the following 
proposition about the associated $\cSO$-complexes.

\begin{proposition}
  \label{prop:iso-double-tilde}
  Let $\Ym$ and $\Yp$ be $\Z$-homology spheres, and suppose that $W \colon \Ym 
  \to \Yp$ is a ribbon $\Q$-homology cobordism. Then the $\cSO$-morphism  
  $\lambdat (\double{W}) \colon \Cto (\Ym) \to \Cto (\Ym)$ is $\cSO$-homotopic 
  to an $\cSO$-isomorphism.
\end{proposition}

\begin{proof}
  Write the differential $\dto$ of the $\cSO$-complex $(\Cto (\Ym), \dto)$ as 
  in \eqref{eq:d-tilde}, with the maps $d$, $U$, $D_1$, and $D_2$, and write 
  the $\cSO$-morphism $\lambdat (\double{W})$ as in \eqref{eq:C-tilde}, with 
  the maps $\lambda (\double{W})$, $\Delta_1$, $\Delta_2$, and $\mu$. Since we 
  are working with chain complexes over a field, our argument in 
  \fullref{ssec:i-surgery} shows that there is a chain homotopy $K \colon \Co 
  (\Ym) \to \Co (\Ym)$ such that $K \comp d + d \comp K = \Id_{\Co (\Ym)} - 
  \lambda (\double{W})$.  Defining the map $\Kt \colon \Cto (\Ym) \to \Cto 
  (\Ym)$ by
  \[
    \Kt =
    \begin{pmatrix}
      K & 0 & 0\\
      0 & 0 & 0\\
      0 & 0 & -K\\
    \end{pmatrix},
  \]
  we immediately see that $\Kt$ anti-commutes with $\chi$; moreover, we can 
  compute that
  \[
    \Kt \comp \dto + \dto \comp \Kt + \lambdat (\double{W}) =
    \begin{pmatrix}
      \Id_{\Co (\Ym)} & 0 & 0\\
      * & 1 & 0\\
      * & * & \Id_{\Co (\Ym)}
    \end{pmatrix},
  \]
  and so $\Kt$ is an $\cSO$-homotopy between $\lambdat (\double{W})$ and 
  $\widetilde{q} = \Kt \comp \dto + \dto \comp \Kt + \lambdat (\double{W})$, 
  which is clearly invertible over $\Q$.
\end{proof}

\begin{proof}[Proof of \fullref{thm:main-i-eq}]
  \fullref{prop:iso-double-tilde} and the discussion above it together imply 
  that $\If (\double{W}) = \widehat{q}_*$, $\It (\double{W}) = 
  \widecheck{q}_*$, and $\Ib (\double{W}) = \overline{q}_*$, where $\widehat{q} 
  \colon \Cf (\Ym) \to \Cf (\Ym)$, $\widecheck{q} \colon \Ct (\Ym) \to \Ct 
  (\Ym)$, and $\overline{q} \colon \Cb (\Ym) \to \Cb (\Ym)$ are the chain maps 
  corresponding to some $\cSO$-isomorphism $\widetilde{q} \colon \Cto (\Ym) \to 
  \Cto (\Ym)$.  It is clear that $\widehat{q}_*$, $\widecheck{q}_*$, and 
  $\overline{q}_*$ are $\Q [x]$-module isomorphisms.
\end{proof}

\subsection{A character variety approach to 
  \texorpdfstring{\fullref{thm:main-i-o}}{the instanton Floer statement}}
\label{ssec:characters-i}

In this subsection, we sketch a different approach to prove 
\fullref{thm:main-i-o} and \fullref{thm:main-i-s}. For simplicity, we focus on 
the proof of \fullref{thm:main-i-o}.  In particular, let $\Ym$ and $\Yp$ be 
$\Z$-homology spheres, and suppose that $W \colon \Ym \to \Yp$ is a ribbon 
$\Q$-homology cobordism. (See \fullref{rmk:zhs}.) Our approach in this section 
is based on the relationship between the character varieties of $\Ym$ and 
$\Yp$. A key component of our proof is an energy argument that also appears in 
\cite{braam-donaldson, fukaya, dfl}.

Fix a Riemannian metric on $\Ym$ and a cylindrical metric on $\double{W}$ that 
is compatible with the metric on $\Ym$. For simplicity, we first assume that 
these metrics allow us to define the instanton Floer homology $\Io (\Ym)$ and 
the cobordism map $\Io (\double{W})$ without perturbing the Chern--Simons 
functional of $\Ym$ or the ASD equation on $\double{W}$. In particular, $\Io 
(\Ym)$ is the homology of a chain complex $(\Co (\Ym), d)$, where $\Co (\Ym)$ 
is generated by gauge equivalence classes of non-trivial flat connections, or 
equivalently, non-trivial elements of the character variety of $\Ym$. The 
cobordism map $\Io (\double{W})$ is defined using the moduli spaces $\moduli 
(\double{W}; \alpha_1, \alpha_2)$, i.e.\ gauge equivalence classes of solutions 
of the (unperturbed) ASD equation
\begin{equation}
  \label{eq:ASD}
  F(A)^+ = 0,
\end{equation}
where $A$ is an $\SUtwo$-connection on $\double{W}$ asymptotic to the 
non-trivial flat $\SUtwo$-connections $\alpha_1$ and $\alpha_2$ on the ends of 
$\double{W}$. 

Let $\connspB (\double{W}; \alpha_1, \alpha_2)$ be the space of gauge 
equivalence classes of connections on $\double{W}$ that are asymptotic to 
$\alpha_1$ and $\alpha_2$ on the ends (that may or may not satisfy 
\eqref{eq:ASD}).  For a connection $A$ representing an element of $\connspB 
(\double{W}; \alpha_1, \alpha_2)$, the topological energy of $A$, given by the 
Chern--Weil integral
\[
  \energy (A) = \frac{1}{8 \pi^2} \int_{\double{W}} {\tr} (F (A) \wedge F (A)),
\]
can easily be verified to be invariant under the action of the gauge group, and 
also under continuous deformation of $A$.  Moreover, \eqref{eq:ASD} implies 
that, for connections $A$ that represent an element in $\moduli (\double{W}; 
\alpha_1, \alpha_2)$, we always have $\energy (A) \geq 0$, and $\energy (A) = 
0$ if and only if $A$ is a flat connection.  We will also need the following 
fact, which says that the topological energy of $A$ determines the dimension of 
the component of the moduli space $\moduli (\double{W}; \alpha_1, \alpha_2)$ 
that contains $A$.

\begin{lemma}
  \label{lem:index-formula}
  There exists a function $\epsilon$ that associates to each non-trivial (i.e.\ 
  irreducible) flat connection $\alpha$ on $\Ym$ a real number $\epsilon 
  (\alpha)$, such that the equality
  \[
    d = 8 \energy (A) + \epsilon (\alpha_1) - \epsilon (\alpha_2)
  \]
  holds whenever $[A] \in \moduli (\double{W}; \alpha_1, \alpha_2)_d$.
\end{lemma}

\begin{proof}
  To each connection $A$ representing an element of $\moduli (\double{W}; 
  \alpha_1, \alpha_2)$, we may associate the ASD operator $\asdopa$, which is 
  an elliptic operator; if $A$ represents an element of $\moduli (\double{W}; 
  \alpha_1, \alpha_2)_d$, then $d$ is equal to the index of $\asdopa$.
  Therefore, it suffices to show that, for some choice of $\epsilon$,
	\begin{equation}
    \label{eq:ind-formula-id}
    \ind (\asdopa) = 8 \energy (A) + \epsilon (\alpha_1) - \epsilon(\alpha_2).
	\end{equation}
  We first verify this formula when $\alpha_1 = \alpha_2$. Since index and 
  topological energy are invariant under continuous deformation, we may assume 
  without loss of generality that the connection $A$ is the pull-back of a 
  fixed flat connection $\alpha_1$ on the cylindrical ends of $\double{W}$. In 
  particular, $A$ induces a connection $\Abar$ on the closed $4$-manifold 
  $\DWbar$ obtained by gluing the incoming and the outgoing ends of 
  $\double{W}$ by the identity. Clearly, the topological energy of $A$ and 
  $\Abar$ are equal to each other.  Moreover, the additive property of indices 
  with respect to gluing (see \cite[Chapter~3]{Don:YM-Floer}) implies that 
  $\ind (\asdopa) = \ind (\asdopabar)$.  Since $\DWbar$ has the same 
  $\Z$-homology as $S^1 \cross S^3$, the standard index theorems for the ASD 
  operator on closed $4$-manifolds imply that
	\[
    \ind (\asdopabar) = 8 \energy (\Abar).
	\]
  This shows that \eqref{eq:ind-formula-id} holds in the case that $\alpha_1 = 
  \alpha_2$.
	
  In the more general case, we fix an arbitrary irreducible flat connection 
  $\alpha_1$ on $\Ym$, and set $\epsilon (\alpha_1) = 0$. For any other 
  irreducible flat connection $\alpha$, we take an arbitrary connection $B$ on 
  $\Ym \cross \R$ that is equal to the pull-backs of representatives of 
  $\alpha_1$ and $\alpha$ on $(-\infty, -1] \cross \Ym$ and $[1, \infty) \cross 
  \Ym$ respectively, and define
	\[
    \epsilon (\alpha) = \ind (\asdop_B) - \frac{1}{\pi^2} \int_{\Ym \cross \R} 
    \tr (F (B) \wedge F (B)).
	\]
  One can check that $\epsilon$ is well defined, and it only depends on the 
  gauge equivalence class of $\alpha$.
  Another application of the additive property of the index of ASD operators 
  with respect to gluing completes the proof of the lemma.
\end{proof}

\begin{lemma}
  \label{lem:flat-ASD}
  If the moduli space $\moduli (\double{W}; \alpha_1, \alpha_2)_0$ is not 
  empty, then either $\epsilon (\alpha_2) > \epsilon (\alpha_1)$, or $\alpha_1 
  = \alpha_2$.  Moreover, the moduli space $\moduli (\double{W}; \alpha_1, 
  \alpha_1)_0$ consists of an odd number of flat connections.
\end{lemma}

\begin{proof}
  \fullref{lem:index-formula} implies that, if there is an element $[A]$ in 
  $\moduli (\double{W}; \alpha_1, \alpha_2)_0$, then $\epsilon (\alpha_2) \geq 
  \epsilon (\alpha_1)$.  Moreover, if $\epsilon (\alpha_1) = \epsilon 
  (\alpha_2)$, then the connection $A$ has to be flat, which is to say that $A$ 
  represents an element of the character variety $\charvarSUtwo (\double{W})$.  
  In particular, \fullref{prop:character-embedding-2} implies that $\alpha_1 = 
  \alpha_2$. By assumption, any element of $\moduli (\double{W}; \alpha_1, 
  \alpha_1)_0$ is cut out regularly, and we do not need to perform any 
  perturbation. Regularity of a flat connection $A$ on $\double{W}$ is 
  equivalent to the property that $H^1 (\double{W}; \Ad_{A})$ is trivial.  The 
  proof of \fullref{prop:character-embedding-2} implies that the 
  $\SUtwo$-representations of $\pi_1 (\double{W})$ that extend a given 
  representation of $\pi_1 (\Ym)$ is the set of solutions of $K (g_1, \dotsc, 
  g_m) = 1$, where $K \colon \SUtwo^m \to \SUtwo^m$ is a map of degree $\pm 1$.  
  Since the solutions of these equations are cut out transversely, the number 
  of solutions of this extension problem is an odd integer.
\end{proof}	

\fullref{lem:flat-ASD} implies that if we sort flat connections on $\Ym$ based 
on their $\epsilon$-values, then the chain map $\Co (\double{W})$ is upper 
triangular with non-zero diagonal entries. In particular, $\Io (\double{W})$ is 
an isomorphism.

In general, we need to consider perturbations of the Chern--Simons functional 
of $\Ym$ and the ASD equation on $\double{W}$. There are standard functions on 
the space of connections on $\Ym$ that give rise to perturbed Chern--Simons 
functionals of $\Ym$ (see \cite[Chapter~5]{Don:YM-Floer}) that are sufficient 
to define the instanton Floer homology $\Io (\Ym)$. Any such perturbation can 
be extended to a perturbation of the ASD equation on $\double{W}$ that is 
\emph{time independent} in the sense defined by Braam and Donaldson 
\cite{braam-donaldson}.  The main point of considering such perturbations is 
that, even after we slightly modify the definition of topological energy, the 
solutions of the perturbed ASD equation will still have non-negative 
topological energy.  Having fixed the above, another technical issue would be 
to know whether the solutions of the perturbed ASD equation with vanishing 
topological energy are cut out regularly. If we happen to know that our chosen 
perturbation has this additional property, then we can proceed as above to show 
that the map $\Io (\double{W})$ is an isomorphism.  However, the authors have 
not checked whether there is a time-independent perturbation with this 
property.

\section{Heegaard Floer homology}
\label{sec:hf}

\subsection{Surgery and cobordism maps in Heegaard Floer theory}
\label{ssec:hf-surgery}

In light of \fullref{prop:surgery}, our strategy to prove \fullref{thm:main-hf} 
will be to show that the cobordism map for $\double{W}$ is actually just 
determined by that for $X \homeo (\Ym \cross I) \connsum m (S^1 \cross S^3)$ 
and the homology classes of the $\gamma_i$'s, and hence must agree with that of 
$\Ym \cross I$.  We will first focus on $\HFh$; it will be shown later in the 
proof of \fullref{thm:main-hf} that this is sufficient to recover the result 
for the other flavors.  The necessary tool is \fullref{prop:surgery-maps} 
below, which shows the behavior of the Heegaard Floer cobordism maps under 
surgery along circles, and is the counterpart of \fullref{prop:surgery-i-adm} 
and \fullref{prop:surgery-i} for Heegaard Floer homology.  This statement is 
known to experts, and can be derived from the link cobordism TQFT of Zemke; see 
\fullref{rmk:surgery-maps} below.  A closely related result is also already 
established in \cite[Example~1.4]{KhaLinSas19}.  For completeness, we provide a 
proof in this subsection. Note that we do not assume $3$- and $4$-manifolds to 
be connected in this subsection.

Recall that given a connected $\SpinC$-cobordism $(W, \spinct) \colon (Y_1, 
\spinc_1) \to (Y_2, \spinc_2)$ between closed, connected $3$-manifolds, 
Ozsv\'ath and Szab\'o \cite{OzsSza06} define cobordism maps
\[
  \Fg_{W, \spinct} \colon \HFg (Y_1, \spinc_1) \tensor (\extprod^* (H_1 (W) / 
  \Tors) \tensor \zeetwo) \to \HFg (Y_2, \spinc_2).
\]
These maps have the property that
\begin{equation}
  \label{eq:H1-action}
  \Fg_{W, \spinct} (x \tensor \xi) = \Fg_{W,\spinct} (\xi_1 \cdot x) + \xi_2 
  \cdot \Fg_{W, \spinct} (x),
\end{equation}
whenever $\xi \in H_1(W) / \Tors$ satisfies $\xi = \iota_1(\xi_1) - 
\iota_2(\xi_2)$, where $\xi_i \in H_1(Y_i) / \Tors$ and $\iota_i$ is induced by 
inclusion; see \cite[p.~186]{OzsSza03}. We may also sum over all 
$\SpinC$-structures on $W$, and obtain a total map
\[
  \Fg_{W} \colon \HFg (Y_1) \tensor (\extprod^* (H_1 (W) / \Tors) \tensor 
  \zeetwo) \to \HFg (Y),
\]
satisfying a property analogous to \eqref{eq:H1-action}. We are now ready to 
state:

\begin{proposition}
  \label{prop:surgery-maps}
  Let $Y_1$ and $Y_2$ be closed, connected $3$-manifolds, and let $X \colon Y_1 
  \to Y_2$ be a connected cobordism. Suppose that $\gamma_1, \dotsc, 
  \gamma_\ell \subset \Int (X)$ are loops with disjoint neighborhoods 
  $\nbhdgammai \homeo \gamma_i \cross D^3$, and denote by $Z$ the result of 
  surgery on $X$ along $\gamma_1, \dotsc, \gamma_\ell$. Then for $x \in \HFh 
  (Y_1)$,
  \begin{equation}
    \label{eq:surgery-maps}
    \Fh_{X} (x \tensor ([\gamma_1] \wedge \dotsb \wedge [\gamma_\ell])) = 
    \Fh_{Z} (x).
  \end{equation}
  Thus, $\Fh_Z$ depends only on $X$ and $[\gamma_1] \wedge \cdots \wedge 
  [\gamma_\ell] \in \extprod^* (H_1 (X) / \Tors) \tensor \zeetwo$.
\end{proposition}

\begin{remark}
  \label{rmk:surgery-maps}
  A surgery formula for link cobordisms and link Floer homology, similar to 
  \fullref{prop:surgery-maps}, is provided by Zemke 
  \cite[Proposition~5.4]{Zem19a}. One may obtain \fullref{prop:surgery-maps} 
  via an identification, also provided by Zemke \cite[Theorem~C]{Zem18}, of 
  link cobordism maps with maps induced by cobordisms between $3$-manifolds.  
  In this paper, we instead provide a direct proof without mentioning any link 
  cobordism theory, in the interest of providing a self-contained discussion.
\end{remark}

Before giving the proof, we describe the idea informally.  Surgery on 
$\gamma_i$ is the result of removing a copy of $S^1 \cross D^3$ and replacing 
it with $D^2 \cross S^2$.  The cobordism map for $D^2 \cross S^2$ agrees with 
that of $S^1 \cross D^3$ if one contracts the latter map by the generator of 
$H_1$. Composing with the cobordism map for $X \setminus (\bigdisjunion 
\nbhdgammai)$, the result follows.  However, to prove this carefully, we must 
cut and re-glue several different codimension-$0$ submanifolds, and thus need 
to use the graph TQFT framework by Zemke \cite{Zem-graph-hat}.  Below, we give a brief 
review of the necessary elements.

Let $Y$ be a possibly disconnected $3$-manifold, and let $\pts$ be a set of 
points in $Y$ with at least one point in each component.  Let $W \colon Y_1 \to 
Y_2$ be a cobordism, and let $\graph$ be a graph embedded in $W$ with $\bdy 
\graph = \pts_1 \union \pts_2$.  Then, Zemke \cite{Zem-graph-hat} constructs Heegaard 
Floer homology groups $\HFh (Y_i, \pts_i)$ and cobordism maps $\Fh_{W, \graph} 
\colon \HFh(Y_1, \pts_1) \to \HFh(Y_2, \pts_2)$.

In a later paper, Zemke \cite{Zem-graph-minus} constructs cobordism maps \[
  F^A_{W, \graph, \spinct} \colon \HFg(Y_1, \pts_1, \spinct \rvert_{Y_1}) \to 
  \HFg(Y_2, \pts_2, \spinct \rvert_{Y_2}), \qquad F^B_{W, \graph, \spinct} 
  \colon \HFg(Y_1, \pts_1, \spinct \rvert_{Y_1}) \to \HFg(Y_2, \pts_2, 
  \spinct \rvert_{Y_2})
\]
for each $\spinct \in \SpinC (W)$, for various flavors $\HFg$ of Heegaard 
Floer homology groups. One may also take the sum over all $\spinct \in \SpinC 
(W)$ to obtain maps $F^A_{W, \graph}$ and $F^B_{W, \graph}$. In this theory, 
for $\HFm$, the graph $\graph$ needs to be equipped with a cyclic ordering of 
the edges adjacent to each vertex; however, for $\HFh$, the map is 
independent of this choice of a cyclic ordering 
\cite[Lemma~4.5]{Zem-graph-minus}.  Furthermore, for $\HFh$, the maps 
$F^A_{W, \graph, \spinct}$ and $F^B_{W, \graph, \spinct}$ coincide, as can be 
seen by combining \cite[Lemma~5.7]{Zem-graph-minus} and the definitions of 
the type-$A$ and type-$B$ graph action maps \cite[Equation~(7.1) and 
Equation~(7.2)]{Zem-graph-minus}.

As pointed out to the authors by Ian Zemke, for $\HFh$, the maps $F^A_{W, 
  \graph} = F^B_{W, \graph}$ in fact agree with $\Fh_{W, \graph}$. Indeed, it 
suffices to check this for maps associated to $4$-dimensional $1$-, $2$-, and 
$3$-handles, as well as maps associated to three elementary graph cobordisms: 
free-stabilization cobordisms, free-destabilization cobordisms, and 
wye-shaped cobordisms \cite[Figure~1.1~($\Gamma$-1) and 
($\Gamma$-2)]{Zem-graph-hat}. For handles, the definitions of $\Fh_{W, 
  \graph}$ \cite[Section~2.4 and Section~3]{Zem-graph-hat} and $F^A_{W, 
  \graph} = F^B_{W, \graph}$ \cite[Section~8 and Section~9]{Zem-graph-minus} 
coincide, as they are ultimately equal to the maps described by Ozsv\'ath and 
Szab\'o \cite{OzsSza06}. For graph cobordisms, $\Fh_{W, \graph}$ is computed 
in \cite[Section~4]{Zem-graph-hat}, while $F^A_{W, \graph} = F^B_{W, \graph}$ 
are computed in \cite[Section~4]{Zem-graph-duality}. This equivalence between 
the two graph TQFTs helps us establish some of the properties for $\Fh_{W, 
  \graph}$ in the following theorem.

\begin{theorem}[Zemke \cite{Zem-graph-hat, Zem-graph-minus}]
  \label{thm:zemke}
  The cobordism maps $\Fh_{W, \Gamma}$ satisfy the following.
  \begin{enumerate}
    \item \label{item:disj-union} Under disjoint union, we have that $\HFh (Y_1 
      \disjunion Y_2, \pts_1 \disjunion \pts_2) = \HFh (Y_1, \pts_1) \tensor 
      \HFh (Y_2, \pts_2)$, and $\Fh_{(W_1, \graph_1) \disjunion (W_2, 
        \graph_2)} = \Fh_{(W_1, \graph_1)} \tensor \Fh_{(W_2, \graph_2)}$.
    \item \label{item:gluing} Given $(W, \graph) \colon (Y_1, \pts_1) \to (Y_2, 
      \pts_2)$ and $(W', \graph') : (Y_2, \pts_2) \to (Y_3, \pts_3)$, then 
      $\Fh_{W', \graph'} \comp \Fh_{W, \graph} = \Fh_{W \union W', \graph 
        \union \graph'}$; see \cite[Theorem~1.2~(2)]{Zem-graph-hat}.
    \item \label{item:spinc} $\Fh_{W,\graph}$ admits a decomposition by 
      $\SpinC$-structures in the usual way. In particular, $\Fh_{W, \graph} = 
      \sum_{\spinct \in \SpinC (W)} \Fh_{W, \graph, \spinct}$, and
      \[
        \Fh_{W', \graph', \spinct_{W'}} \comp \Fh_{W, \graph, \spinct_W} = 
        \tallsum{\sum_{\substack{\spinct \in \SpinC (W \union W')\\ \spinct 
              \rvert_{W} = \spinct_W, \, \spinct \rvert_{W'} = \spinct_{W'}}}} 
        \Fh_{W \union W', \graph \union \graph', \spinct};
      \]
      see \cite[Theorem~C]{Zem-graph-minus}. (We take the convention that this 
      equation remains valid when $\spinct_{W} \rvert_{Y_2} \neq \spinct_{W'} 
      \rvert_{Y_2}$, in which case both sides of the equation are identically 
      zero.)
    \item \label{item:exterior} If $\lambda$ is an arc from the boundary of 
      some $B^4 \subset W$ to $\graph$, then $\Fh_{W, \graph} (x) = \Fh_{W 
        \setminus B^4, \graph \union \lambda} (x \tensor y)$, where $y$ is the 
      generator of $\HFh (\bdy B^4)$; see 
      \cite[Proposition~11.1]{Zem-graph-minus}.
    \item \label{item:osz} Suppose that $Y_1$ and $Y_2$ are connected, $\pts_1$ 
      and $\pts_2$ each consist of a single point, and $\graph$ is a path. Then 
      $\Fh_{W} (x) = \Fh_{W, \graph} (x)$, where $\Fh_{W}$ is the original 
      Ozsv\'ath--Szab\'o cobordism map; see 
      \cite[Theorem~1.2~(1)]{Zem-graph-hat}.  (Implicitly, the 
      Ozsv\'ath--Szab\'o cobordism map requires a choice of basepoints and a 
      choice of path, but the injectivity statement in \fullref{thm:main-hf} is 
      independent of both choices.)
    \item \label{item:H1} Suppose again that $Y_1$ and $Y_2$ are connected, 
      $\pts_1$ and $\pts_2$ each consist of a single point, and $\graph$ is a 
      path.  Let $\gamma$ be a simple closed loop in $\Int (W)$ that intersects 
      $\graph$ at a single point.  Then $\Fh_{W} (x \tensor [\gamma]) = \Fh_{W, 
        \graph \union \gamma} (x)$, where the left-hand side is the 
      Ozsv\'ath--Szab\'o cobordism map defined above; see 
      \cite[Lemma~4.3]{Zem-graph-hat}.
    \item \label{item:action} As a special case of \eqref{item:H1}, Let $Y$ be 
      connected and let $\pts$ consist of a single point.  Consider $\graph = 
      \pts \cross I \subset Y \cross I$.  Choose a simple closed loop $\gamma$ 
      in $Y$ based at $\pts$ and let $\graph_\gamma$ be the graph obtained by 
      appending $\gamma \cross \set{1/2}$ to $\graph$.  Denote the cobordism 
      map $\Fh_{Y \cross I, \graph_\gamma}$ by $\Fc (\gamma)$.  Then, $\Fc 
      (\gamma)$ depends only on $[\gamma] \in H_1 (Y)$.  Furthermore, $\Fc 
      (\gamma * \gamma') = \Fc (\gamma) + \Fc (\gamma')$ and $\Fc (\gamma) 
      \comp \Fc (\gamma) = 0$.  Here, $\gamma * \gamma'$ is a simple closed 
      loop in the based homotopy class of the concatenation.
  \end{enumerate}
\end{theorem}

We now need a slight generalization of \fullref{thm:zemke}~\eqref{item:H1}, 
i.e.\ \cite[Lemma~4.3]{Zem-graph-hat}, which will allow us to analyze the 
effect on the cobordism map of appending multiple loops to a path.  We begin 
with the identity cobordism.  

\begin{lemma}
  \label{lem:product-loops}
  Suppose that $Y$ is connected, and that $\pts$ consists of a single point.  
  Suppose that $\graph$ is a graph obtained by taking $\pts \cross I \subset Y 
  \cross I$ and appending to it $\ell$ disjoint simple closed curves $\gamma_1, 
  \dotsc, \gamma_\ell$, which each intersect $\pts \cross I$ only at a single 
  point.  Then
  \[
    \Fh_{Y \cross I} (x \tensor ([\gamma_1] \wedge \dotsb \wedge 
    [\gamma_\ell])) = \Fh_{Y \cross I, \graph} (x),
  \]
  where the left-hand side is the Ozsv\'ath--Szab\'o cobordism map.
\end{lemma}

\begin{proof}
  This is implicit in the work of Zemke \cite{Zem-graph-hat}, but we give the 
  proof for completeness.  By a homotopy, and hence isotopy, in $Y \cross I$, 
  we may arrange that $\gamma_i \subset Y \cross \set{i / (\ell+1)}$.  
  Therefore, using \fullref{thm:zemke}~\eqref{item:gluing}, we can write 
  $\Fh_{Y \cross I, \graph}$ as a composition of the maps $\Fc (\gamma_i)$.  
  Viewing $\Fc$ as a function from $H_1 (Y)$ to $\End_{\zeetwo} (\HFh (Y))$,
  \fullref{thm:zemke}~\eqref{item:action} implies that this descends to the 
  exterior algebra.
\end{proof}

We move on to more general cobordisms.

\begin{lemma}
  \label{lem:gen-loops}
  Suppose that $Y_1$ and $Y_2$ are connected, and that $\pts_1$ and $\pts_2$ 
  each consist of a single point. Let $W \colon Y_1 \to Y_2$ be a connected 
  cobordism.  Suppose that $\graph$ is a graph obtained by taking a path 
  $\alpha$ from $\pts_1$ to $\pts_2$ and appending to it $\ell$ disjoint simple 
  closed loops $\gamma_1, \dotsc, \gamma_\ell$, which each intersect $\alpha$ 
  only at a single point.  Then
  \[
    \Fh_{W} (x \tensor ([\gamma_1] \wedge \dotsb \wedge [\gamma_\ell])) = 
    \Fh_{W, \graph} (x),
  \]
  where the left-hand side is the Ozsv\'ath--Szab\'o cobordism map.
\end{lemma}

\begin{proof}
  We may decompose $(W, \graph)$ as a composition of three cobordisms: $(W_1, 
  \graph_1)$, where $W_1$ consists only of 1-handles and $\graph_1$ is a path; 
  $(\bdy W_1 \cross I, \graph_*)$, where $\graph_*$ consists of a graph in 
  $\bdy W_1 \cross I$ as in the statement of \fullref{lem:product-loops}; and 
  $(W_2, \graph_2)$, where $W_2$ consists of $2$- and $3$-handles, and 
  $\graph_2$ is again a path. The result now follows from 
  \fullref{lem:product-loops} together with 
  \fullref{thm:zemke}~\eqref{item:gluing}.
\end{proof}

With this generalization, we may now complete the proof of 
\fullref{prop:surgery-maps}.

\begin{proof}[Proof of \fullref{prop:surgery-maps}]
  Let $\nbhdgammai \homeo \gamma_i \cross D^3$ be a neighborhood of $\gamma_i$, 
  and let $P = X \setminus (\bigdisjunion_i \nbhdgammai)$.
  Let $X' = X \setminus (B^4_1 \disjunion \dotsb \disjunion B^4_\ell)$, where 
  $B^4_i \subset \Int (\nbhdgammai)$.     We construct a properly embedded 
  graph $\graphXp$ in $X'$ as follows; see \fullref{fig:chain-A}.

  \vspace*{\factualfontsize}

  \begin{figure}[!htbp]
    \labellist
    \footnotesize\hair 2pt
    \pinlabel {$p_0$} at 13 12
    \pinlabel {$p_1$} at 70 16
    \pinlabel {$p_2$} at 145 16
    \pinlabel {$p_\ell$} at 230 16
    \pinlabel {$p_{\ell+1}$} at 282 16

    \pinlabel {$\alpha_0$} at 45 16
    \pinlabel {$\alpha_1$} at 110 16
    \pinlabel {$\alpha_2$} at 175 16
    \pinlabel {$\alpha_\ell$} at 255 16

    \pinlabel {$\beta_1$} at 70 35
    \pinlabel {$\beta_2$} at 145 35
    \pinlabel {$\beta_\ell$} at 230 35

    \pinlabel {$q_1$} at 70 50
    \pinlabel {$q_2$} at 145 50
    \pinlabel {$q_\ell$} at 232 50

    \pinlabel {$\gamma_1$} at 58 85
    \pinlabel {$\gamma_2$} at 133 85
    \pinlabel {$\gamma_\ell$} at 220 85

    \pinlabel {$r_1$} at 97 85
    \pinlabel {$r_2$} at 172 85
    \pinlabel {$r_\ell$} at 259 85

    \pinlabel {$\delta_1$} at 94 70
    \pinlabel {$\delta_2$} at 170 70
    \pinlabel {$\delta_\ell$} at 256 70

    \pinlabel {$\epsilon_1$} at 91 100
    \pinlabel {$\epsilon_2$} at 167 100
    \pinlabel {$\epsilon_\ell$} at 253 100

    \pinlabel {$s_1$} at 71 103
    \pinlabel {$s_2$} at 147 103
    \pinlabel {$s_\ell$} at 233 103

    \pinlabel {$S^3$} [b] at 78 111
    \pinlabel {$S^3$} [b] at 154 111
    \pinlabel {$S^3$} [b] at 241 111

    \pinlabel {\normalsize $Y_1$} at -9 27
    \pinlabel {\normalsize $Y_2$} at 323 27

    \endlabellist
    \includegraphics[scale=1.0]{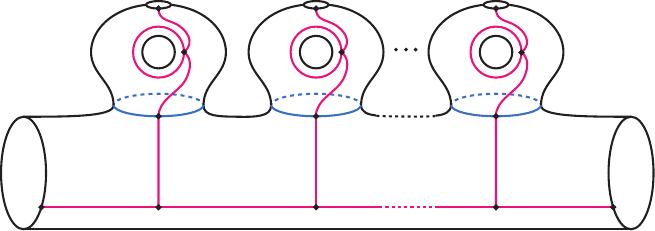}
    \caption{The embedded graph $\graphXp$ in $X'$.}
    \label{fig:chain-A}
  \end{figure}

  We begin with the vertex set.  Choose $\ell$ points $p_1, \dotsc, p_\ell$ in 
  the interior of $P$, and points $p_0$ and $p_{\ell+1}$ in $Y_1$ and $Y_2$ 
  respectively.  Choose $\ell$ points $q_1, \dotsc q_\ell$ with $q_i \in \bdy 
  \nbhdgammai$, which are copies of $S^1 \cross S^2$.  Choose $\ell$ points 
  $r_1, \dotsc, r_\ell$ with $r_i \in \gamma_i$.  Finally, let $s_i$ be a point 
  in $S^3_i = \partial B^4_i$ for each $i$.

  Now we define the edge sets.  Choose any collection of embedded arcs 
  $\alpha_0, \dotsc, \alpha_\ell$ with $\alpha_i \subset P$ connecting $p_i$ 
  and $p_{i+1}$.  Let $\beta_i \subset P$ be an arc from $p_i$ to $q_i$.  
  Connect $q_i$ and $r_i$ by arcs $\delta_i$, and $r_i$ and $s_i$ by arcs 
  $\epsilon_i$, in $\nbhdgammai \setminus B^4_i$.  We may choose the edges 
  above in such a way that their interiors are mutually disjoint, avoid the 
  $\gamma_i$, and are contained in the interior of $X'$.  Then, the edge set of 
  $\graphXp$ consists of the edges $\alpha_i$, $\beta_i$, $\gamma_i$, 
  $\delta_i$, and $\epsilon_i$.  In accordance with 
  \fullref{thm:zemke}~\eqref{item:disj-union}, we view the cobordism map for 
  $(X', \graphXp)$ as a map \[\Fh_{X', \graphXp} \colon \HFh (Y_1) \tensor 
    \paren{\bigtensor^\ell_{i=1} \HFh (S^3_i)} \to \HFh(Y_2).\]

  It follows from \fullref{lem:gen-loops} as well as  
  \fullref{thm:zemke}~\eqref{item:disj-union} and \eqref{item:exterior} that
  \[
    \Fh_X (x \tensor ([\gamma_1] \wedge \dotsb \wedge [\gamma_\ell])) = 
    \Fh_{X', \graphXp} (x \tensor y_1 \tensor \dotsb \tensor y_\ell),
  \]
  where $y_i$ is the generator of $\HFh(S^3_i)$.  (We can first contract the 
  homology elements, and then contract the arcs $\beta_i \union \delta_i \union 
  \epsilon_i$.)  Let $\graphP$ be the intersection of $\graphXp$ with $P$, 
  which can alternatively be obtained by excising the $\gamma_i, \delta_i,$ and 
  $\epsilon_i$ arcs.  

  Note that $Z = P \union (\bigdisjunion_i (D^2 \cross S^2)_i)$.  Here, we 
  suppress the choice of gluing from the notation.  Similarly, we let $Z' = Z 
  \setminus (B^4_1 \disjunion \dotsb \disjunion B^4_\ell)$ where $B^4_i \subset 
  (D^2 \cross S^2)_i$; then $Z' = P \union (\bigdisjunion_i R_i)$, where each 
  $R_i$ is a punctured $D^2 \cross S^2$.  Let $\zeta_i$ be an arc in $R_i$ that 
  connects $q_i$ and $s_i$; then we define $\graphRi$ in $R_i$ to be $\zeta_i$, 
  and define $\graphZp$ in $Z'$ as the union of the arcs $\alpha_i, \beta_i, 
  \zeta_i$. See \fullref{fig:chain-B} for an illustration of $(Z', \graphZp)$.

  \vspace*{\factualfontsize}

  \begin{figure}[!htbp]
    \labellist
    \footnotesize \hair 2pt
    \pinlabel {$p_0$} at 13 12
    \pinlabel {$p_1$} at 70 16
    \pinlabel {$p_2$} at 145 16
    \pinlabel {$p_\ell$} at 230 16
    \pinlabel {$p_{\ell+1}$} at 282 16

    \pinlabel {$\alpha_0$} at 45 16
    \pinlabel {$\alpha_1$} at 110 16
    \pinlabel {$\alpha_2$} at 175 16
    \pinlabel {$\alpha_\ell$} at 255 16

    \pinlabel {$\beta_1$} at 70 35
    \pinlabel {$\beta_2$} at 145 35
    \pinlabel {$\beta_\ell$} at 230 35

    \pinlabel {$q_1$} at 70 50
    \pinlabel {$q_2$} at 145 50
    \pinlabel {$q_\ell$} at 232 50

    \pinlabel {$\zeta_1$} at 95 85
    \pinlabel {$\zeta_2$} at 171 85
    \pinlabel {$\zeta_\ell$} at 258 85

    \pinlabel {$s_1$} at 71 103
    \pinlabel {$s_2$} at 147 103
    \pinlabel {$s_\ell$} at 233 103

    \pinlabel {$S^3$} [b] at 78 111
    \pinlabel {$S^3$} [b] at 154 111
    \pinlabel {$S^3$} [b] at 241 111

    \pinlabel {\normalsize $Y_1$} at -9 27
    \pinlabel {\normalsize $Y_2$} at 323 27

    \endlabellist
    \includegraphics[scale=1.0]{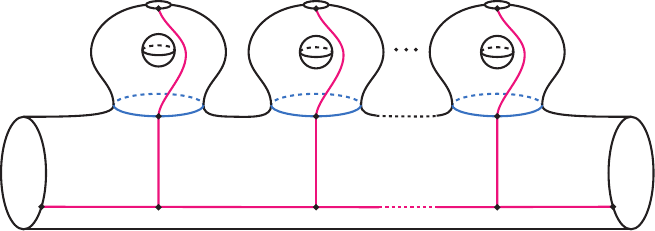}
    \caption{The embedded graph $\graphZp$ in $Z'$.}
    \label{fig:chain-B}
  \end{figure}

  Viewing the cobordism map for $(Z', \graphZp)$ as a map $\Fh_{Z', \graphZp} 
  \colon \HFh (Y_1) \tensor (\bigtensor_i \HFh (S^3_i)) \to \HFh (Y_2)$, we 
  have
  \[
    \Fh_Z (x) = \Fh_{Z', \graphZp} (x \tensor y_1 \tensor \dotsb \tensor 
    y_\ell),
  \]
  again by \fullref{thm:zemke}~\eqref{item:exterior}. Thus, 
  \eqref{eq:surgery-maps} will follow if we can show
  \[
    \Fh_{X', \graphXp} (x \tensor y_1 \tensor \dotsb \tensor y_\ell) = \Fh_{Z', 
      \graphZp} (x \tensor y_1 \tensor \dotsb \tensor y_\ell).
  \]

  To do so, let $Q_i = \nbhdgammai \setminus B^4_i$, and let $\graphQi$ be the 
  intersection of $\graphXp$ with $Q_i$. Both $(Q_i, \graphQi)$ and $(R_i, 
  \graphRi)$ are cobordisms from $(S^3, s_i)$ to $(S^1 \cross S^2, q_i)$; see 
  \fullref{fig:caps}.

  \begin{figure}[!htbp]
    \begin{subfigure}[b]{0.45\textwidth}
      \centering
      \labellist
      \footnotesize \hair 2pt
      \pinlabel {$s_i$} at 22 48
      \pinlabel {$r_i$} at 62 24
      \pinlabel {$q_i$} at 103 48

      \pinlabel {$\epsilon_i$} at 37 30
      \pinlabel {$\gamma_i$} at 61 70
      \pinlabel {$\delta_i$} at 83 30

      \pinlabel {\normalsize $S^3$} at -8 50
      \pinlabel {\normalsize $S^1 \times S^2$} at 145 50
      \endlabellist
      \includegraphics[scale=1.0]{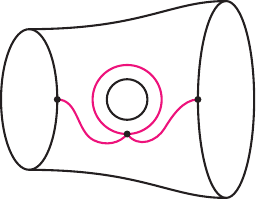}
      \caption{The cobordism $(Q_i, \graphQi)$.}
      \label{fig:caps1}
    \end{subfigure}
    ~
    \begin{subfigure}[b]{0.45\textwidth}
      \centering
      \labellist
      \footnotesize \hair 2pt
      \pinlabel {\normalsize $S^3$} at -8 50
      \pinlabel {\normalsize $S^1 \times S^2$} at 145 50

      \pinlabel {$s_i$} at 22 48
      \pinlabel {$q_i$} at 103 48

      \pinlabel {$\zeta_i$} at 62 27

      \endlabellist
      \includegraphics[scale=1.0]{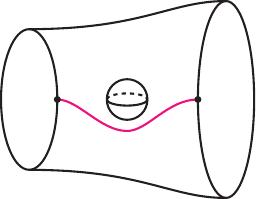}
      \caption{The cobordism $(R_i, \graphRi)$.}
      \label{fig:caps2}
    \end{subfigure}
    \caption{The cobordisms $(Q_i, \graphQi)$ and $(R_i, \graphRi)$.}
    \label{fig:caps}
  \end{figure}

  Viewing $(P, \graphP)$ as a cobordism from $(Y_1, p_0) \disjunion 
  (\bigdisjunion_i (S^1 \cross S^2)_i, q_i) \to (Y_2, p_{\ell+1})$, by 
  \fullref{thm:zemke}~\eqref{item:disj-union} and \eqref{item:gluing}, we have 
  that
  \[
    \Fh_{X', \graphXp} = \Fh_{P, \graphP} \comp \paren{\Id_{\HFh (Y_1)} \tensor 
      \Fh_{Q_1, \graph_{Q_1}} \tensor \dotsb \tensor \Fh_{Q_\ell, 
        \graph_{Q_\ell}}}
  \]
  and
  \[
    \Fh_{Z', \graphZp} = \Fh_{P, \graphP} \comp \paren{\Id_{\HFh (Y_1)} \tensor 
      \Fh_{R_1, \graph_{R_1}} \tensor \dotsb \tensor \Fh_{R_\ell, 
        \graph_{R_\ell}}}.
  \]
  Thus, we need only to show that $\Fh_{Q_i, \graphQi} = \Fh_{R_i, \graphRi}$ 
  for each $i$. On the one hand, \fullref{thm:zemke}~\eqref{item:H1} together 
  with \eqref{eq:H1-action} imply that
  \[
    \Fh_{Q_i, \graphQi} (y_i) = \Fh_{Q_i} (y_i \tensor [\gamma_i]) = [\gamma_i] 
    \cdot \Fh_{Q_i} (y_i).
  \]
  Since $Q_i$ is simply a $1$-handle attachment to $S^3$, its cobordism map, by 
  Ozsv\'ath and Szab\'o's definition, sends $y_i$ to the topmost generator of 
  $\HFh (S^1 \cross S^2)$ (see \cite[p.~364]{OzsSza06}), and the action by 
  $[\gamma_i]$ sends this to the bottommost generator (see 
  \cite[Proposition~6.4]{PropApp}).  On the other hand, 
  \fullref{thm:zemke}~\eqref{item:osz} implies that
  \[
    \Fh_{R_i, \graphRi} (y_i) = \Fh_{R_i} (y_i).
  \]
  Since $R_i$ is simply a $0$-framed $2$-handle attachment along the unknot in 
  $S^3$, its cobordism map sends $y_i$ to the bottommost generator of $\HFh 
  (S^1 \cross S^2)$. (One could directly compute the map from the definition on 
  \cite[pp.~356--357]{OzsSza06} using the standard Heegaard triple diagram for 
  $R_i$. Alternatively, one could observe that $\Fh_{R_i}$ must not be zero 
  because of the exact triangle for surgery along the unknot; since the 
  cobordism map respects the $H_1$-action, and the $H_1$-action on $\HFh (S^3)$ 
  is trivial, this means that $\Fh_{R_i} (y_i)$ is in the kernel of the 
  $H_1$-action on $\HFh (S^1 \cross S^2)$.)  Consequently, $\Fh_{Q_i, \graphQi} 
  (y_i) = \Fh_{R_i, \graphRi} (y_i)$ as desired.
\end{proof}

\begin{proof}[Proof of \fullref{thm:main-hf}]
  We consider the hat flavor first. Consider the double $\double{W}$ of $W$.  
  Then, by \fullref{prop:surgery}, $\double{W}$ is described by surgery on $X 
  \homeo (\Ym \cross I) \connsum m (S^1 \cross S^3)$ along $m$ circles 
  $\gamma_1, \dotsc, \gamma_m$, where $[\gamma_1] \wedge \dotsb \wedge 
  [\gamma_m] = \alpha_1 \wedge \dotsb \wedge \alpha_m \in (\extprod^* (H_1 (X) 
  / \Tors) / \langle H_1 (\Ym) / \Tors \rangle) \tensor_{\Z} \zeetwo$, and 
  $\alpha_i$ is the homology class of the core of the $i^{\text{th}}$ $S^1 
  \cross S^3$ summand.  Note that the same description is true of $\Ym \cross 
  I$; in this case, the surgery is performed along the core circles $\gamma'_i$ 
  of the $(S^1 \cross S^3)$'s themselves.  

  Applying \fullref{prop:surgery-maps} with $Z = \double{W}$, we have that
  \[
    \Fh_{X} (x \tensor ([\gamma_1] \wedge \dotsb \wedge [\gamma_m])) = 
    \Fh_{\double{W}} (x) = \Fh_{-W} \comp \Fh_{W} (x).
  \]
  Now consider $\Ym \cross I$ as surgery on $X$ along the cores $\gamma'_i$.  
  Applying \fullref{prop:surgery-maps} again, this time with $Z = \Ym \cross I$, 
  we have
  \[
    \Fh_{X} (x \tensor ([\gamma'_1] \wedge \dotsb \wedge [\gamma'_m])) = 
    \Fh_{\Ym \cross I} (x) = \Id_{\HFh (\Ym)}.
  \]
  Since $[\gamma_1] \wedge \dotsb \wedge [\gamma_m] = [\gamma'_1] \wedge \dotsb 
  \wedge [\gamma'_m]$ in $\extprod^*(H_1 (X) / \Tors) / \langle H_1(\Ym)/\Tors 
  \rangle \tensor_{\Z} \zeetwo$, by the linearity of $\Fh$, it suffices to show 
  that $\Fh_{X} (x \tensor \xi) = 0$ for $x \in \HFh (\Ym)$ and $\xi \in 
  \extprod^m (H_1(X) / \Tors) \intersect \langle H_1 (\Ym) / \Tors \rangle 
  \tensor_{\Z} \zeetwo$.  Indeed, this will imply that $\Fh_{-W} \circ \Fh_{W} 
  = \Id_{\HFh (\Ym)}$, and we have the desired result for $\HFh$.  

  Note that $\extprod^m (H_1 (X) / \Tors) \intersect \langle H_1 (\Ym) / \Tors 
  \rangle$ is generated by elements of the form $\omega \wedge (\bigwedge_{i 
    \in I} \alpha_i)$, where $\omega$ is a wedge of elements in $H_1 (\Ym) / 
  \Tors$ and $I \subsetneq \{1,\ldots, m\}$; we would like to show that if $\xi 
  \in \langle H_1(\Ym)/\Tors \rangle \tensor_{\Z} \zeetwo$ is of this form, 
  then $\Fh(x \tensor \xi) = 0$ for $x \in \HFh(\Ym)$.  Therefore, let $\xi = 
  \omega \wedge (\bigwedge_{i \in I} \alpha_i)$ be of this form. The idea is 
  that $\xi$ misses at least one $S^1 \cross S^3$ summand, and the cobordism 
  map associated to a twice punctured $S^1 \cross S^3$, without an 
  $H_1$-action, is identically zero. Concretely, choose $j \in \set{1, \dotsc, 
    m} \setminus I$, and write $X = T_j \union_{S^3} V$, where $T_j$ is the 
  $j^{\text{th}}$ $S^1 \cross S^3$ summand punctured once, and $V = ((\Ym 
  \cross I) \connsum (m-1) (S^1 \cross S^3)) \setminus B^4$; then $\xi$ 
  determines a graph $\graph_\xi$ in $X$ such that $\Fh_X (x \tensor \xi) = 
  \Fh_{X, \graph_\xi} (x)$, and we may assume that $\graph_\xi \intersect T_j = 
  \emptyset$. Let $X' = X \setminus B^4_j$, where $B^4_j \subset \Int (T_j)$.  
  As in \fullref{thm:zemke}~\eqref{item:exterior}, choose an arc $\lambda$ from 
  $\bdy B^4_j$ to $\graph_{\xi}$ that intersects $\bdy T_j = \bdy V$ once, and 
  let $\graph_{\xi}' = \graph_{\xi} \union \lambda$; then we have
  \[
    \Fh_{X, \graph_\xi} (x) = \Fh_{X', \graph_\xi'} (x \tensor y_j),
  \]
  where $y_j$ is the generator of $\HFh (\bdy B^4_j)$. Writing $T_j' = T_j 
  \setminus B^4_j$, it is also clear that
  \[
    \Fh_{X', \graph_\xi'} = \Fh_{V, \graph_{\xi}' \intersect V} \comp 
    \paren{\Id_{\HFh (\Ym)} \tensor \Fh_{T_j', \lambda \intersect T_j'}}.
  \]
  Since $\lambda \intersect T_j'$ is simply a path, $\Fh_{T_j', \lambda 
    \intersect T_j'} = \Fh_{T_j'}$. Note that $T_j' \homeo (S^3 \cross I) 
  \connsum (S^1 \cross S^3)$ is obtained by adding a $1$-handle and a 
  $3$-handle to $S^3 \cross I$, and so a direct computation shows that 
  $\Fh_{T_j'} (y_j) = 0$; thus, $\Fh_X (x \tensor \xi) = 0$, as desired.

  To obtain the analogous result for the other flavors of Heegaard Floer 
  homology, we use that the long exact sequences relating the various flavors 
  are natural with respect to cobordism maps.  It is straightforward to see 
  that only an isomorphism on $\HFp$ can induce the identity map on $\HFh$, and 
  similarly for $\HFm$.  Finally, only an isomorphism on $\HFi$ can induce an 
  isomorphism on both $\HFp$ and $\HFm$.
\end{proof}

\subsection{A \texorpdfstring{$\SpinC$}{SpinC}-refinement of 
  \texorpdfstring{\fullref{thm:main-hf}}{the Heegaard Floer statement}}
\label{ssec:spinc}

We now provide a $\SpinC$-refinement of \fullref{thm:main-hf}. First, observe 
that any $\SpinC$-structure $\spinct$ on a cobordism $W \colon Y_1 \to Y_2$ can 
be extended to a $\SpinC$-structure $\double{\spinct}$ on $\double{W}$, since 
$\spinct$ on $W$ and $\spinct$ on $-W$ coincide on the intersection $W 
\intersect -W = Y_2$. We now have the following observation when $W$ is a 
ribbon $\Q$-homology cobordism.

\begin{lemma}
  \label{lem:unique-spinc}
  Let $\Ym$ and $\Yp$ be closed $3$-manifolds, and suppose that $W \colon \Ym 
  \to \Yp$ is a ribbon $\Q$-homology cobordism. If a $\SpinC$-structure 
  $\spincp$ on $\Yp$ can be extended to a $\SpinC$-structure $\spinct$ on $W$, 
  then the extension is unique; moreover, in this case, $\double{\spinct}$ is 
  the unique $\SpinC$-structure on $\double{W}$ that restricts to $\spincp$ on 
  $\Yp$.
\end{lemma}

\begin{proof}
  For the first statement, consider
  \[
    H^2 (W, \Yp) \to H^2 (W) \to H^2 (\Yp)
  \]
  from the long exact sequence of the pair $(W, \Yp)$. By the Poincar\'e 
  Duality, $H^2 (W, \Yp) \isom H_2 (W, \Ym)$. Take a ribbon handle 
  decomposition of $W$; since $W$ is a $\Q$-homology cobordism, the numbers $m$ 
  of $1$- and $2$- handles are the same, and the differential $\bdy_2 \colon 
  C_2 (W, \Ym) \to C_1 (W, \Ym)$ in the cellular chain complex is given by a 
  homomorphism $R \colon \Z^m \to \Z^m$ such that $R \tensor_{\Z} \Q$ is an 
  isomorphism. This means that $R$, and hence $\bdy_2$, are injective, and so 
  $H_2 (W, \Ym) = 0$. Thus, the map $H^2 (W) \to H^2 (\Yp)$ induced by 
  inclusion is injective, proving that any extension $\spinct$ of $\spincp$ is 
  unique.

  For the second statement, consider
  \[
    H^1 (W) \dirsum H^1 (-W) \to H^1 (\Yp) \to H^2 (\double{W}) \to H^2 (W) 
    \dirsum H^2 (-W)
  \]
  from the Mayer--Vietoris exact sequence; we wish to prove the first map is 
  surjective. In fact, we will prove that the map $H^1 (W) \to H^1 (\Yp)$ is an 
  isomorphism. To do so, consider the map $H_1 (\Yp) \to H_1 (W)$. Since $W$ is 
  a $\Q$-homology cobordism, we have $\rk_{\Z} H_1 (\Yp) = \rk_{\Z} H_1 (W)$, 
  and we denote this number by $k$; then the map in question is given by some 
  homomorphism $R' \colon \Z^k \dirsum T_1 \to \Z^k \dirsum T_2$, where $T_1$ 
  and $T_2$ are torsion, with matrix
  \[
    R' =
    \begin{pmatrix}
      A & 0\\
      B & C
    \end{pmatrix}.
  \]
  Viewing $W$ upside down, it is built from $\Yp$ by adding $2$- and 
  $3$-handles, which implies that $R'$ is surjective; in particular, $A \colon 
  \Z^k \to \Z^k$ is also surjective, and thus an isomorphism. By the Universal 
  Coefficient Theorem, the map $H^1 (W) \to H^1 (\Yp)$ is exactly given by the 
  transpose $A^\mathrm{T} \colon \Z^k \to \Z^k$, which is also an isomorphism.  
  Returning to the exact sequence, we see that the third map is injective, 
  showing that the extension from $\spinct$ to $\double{\spinct}$ is unique.
\end{proof}

We are now ready to state the following refinement of \fullref{thm:main-hf}.

\begin{theorem}
  \label{thm:main-hf-spinc}
  Let $\Ym$ and $\Yp$ be closed $3$-manifolds, and suppose that $W \colon \Ym 
  \to \Yp$ is a ribbon $\zeetwo$-homology cobordism. Fix a $\SpinC$-structure 
  $\spinc$ on $\Ym$. Then the sum of cobordism maps
  \[
    \bigparen{\sum_{\substack{\spinct \in \SpinC (W)\\ \spinct \rvert_{\Ym} = 
          \spinc}}}{\Fg_{W, \spinct}} \colon \HFg (\Ym, \spinc) \to 
    \tallsum{\bigdirsum_{\substack{\spinct \in \SpinC (W)\\ \spinct 
          \rvert_{\Ym} = \spinc}}} \HFg (\Yp, \spinct \rvert_{\Yp})
  \]
  includes $\HFg (\Ym, \spinc)$ into the codomain as a summand. In fact,
  \[
    \bigparen{\sum_{\substack{\spinct \in \SpinC (W)\\ \spinct \rvert_{\Ym} = 
          \spinc}}}{\Fh_{\double{W}, \double{\spinct}}} \colon \HFh (\Ym, 
    \spinc) \to \HFh (\Ym, \spinc)
  \]
  is the identity map.
\end{theorem}

\begin{proof}
  The assertion that the first map is injective is a direct consequence of 
  \fullref{thm:main-hf}, since it is simply the restriction of $\Fg_W$ to the 
  summand $\HFg (\Ym, \spinc)$ of $\HFg (\Ym)$. (However, in writing the 
  codomain as the direct sum above, we have implicitly used the fact that for 
  distinct $\spinct_1, \spinct_2 \in \SpinC (W)$, their restrictions $\spinct_1 
  \rvert_{\Yp}, \spinct_2 \rvert_{\Yp} \in \SpinC (\Yp)$ are distinct, which is 
  a consequence of \fullref{lem:unique-spinc}.) The second assertion is 
  obtained by restricting the identity map $\Fh_{\double{W}}$ in 
  \fullref{thm:main-hf} to the summand $\HFh (\Ym, \spinc)$, and observing that 
  all $\SpinC$-structures on $\double{W}$ are of the form $\double{\spinct}$, 
  which follows from \fullref{lem:unique-spinc}.
\end{proof}

With the additional condition that $W$ is a $\Z$-homology cobordism, a $\SpinC$ 
structure $\spincm$ on $\Ym$ determines a unique $\spinct$ on $W$, and hence a 
unique $\spincp$ on $\Yp$.   We have:

\begin{corollary}
  \label{cor:spinc-Z-cob}
  Let $\Ym$ and $\Yp$ be $3$-manifolds, and suppose that $W \colon \Ym \to \Yp$ 
  is a ribbon $\Z$-homology cobordism. Fix a $\SpinC$-structure $\spincm$ on 
  $\Ym$, and let $\spinct$ and $\spincp$ be the corresponding $\SpinC$ 
  structures on $W$ and $\Yp$ respectively. Then the cobordism map $\Fg_{W, 
    \spinct} \colon \HFg (\Ym, \spincm) \to \HFg (\Yp, \spincp)$ includes 
  $\HFg(\Ym, \spincm)$ into $\HFg(\Yp, \spincp)$ as a summand. \qed
\end{corollary}

\subsection{Sutured Heegaard Floer theory}
\label{ssec:hf-sut}

First, we mention that the definition of a cobordism of sutured manifolds is 
given in \fullref{defn:sut-cob}.

We now use \fullref{thm:main-hf-spinc} to prove the sutured analogue.

\begin{proof}[Proof of \fullref{thm:main-hf-sut}]
  Recall from Lekili's work \cite[Theorem 24]{Lek13} that the sutured Floer 
  homology of a sutured manifold $(M, \sut)$ can be described in terms of the 
  Heegaard Floer homology of the sutured closure $\widehat{M} = M \union 
  (F_{g,d} \cross [-1, 1])$ and a closed surface $R$ in $\widehat{M}$ obtained 
  from $F_{g,d}$, where $F_{g,d}$ is a surface of genus $g \geq 2$ and $d$ 
  boundary components.  For more details on the construction of $\widehat{M}$ 
  and $R$, see \fullref{ssec:i-sut}. Then we have the isomorphism
  \[
    \SFH (M, \sut) \isom \bigdirsum_{\langle c_1 (\spinc), R \rangle = 2g - 2} 
    \HFp (\widehat{M}, \spinc).
  \]
  Now, given a ribbon $\zeetwo$-homology cobordism $N \colon \Msutm \to \Msutp$ 
  between sutured manifolds, we can attach $F_{g,d} \cross [-1, 1] \cross I$ to 
  obtain a ribbon $\zeetwo$-homology cobordism $\widehat{N}$ between the 
  sutured closures.  Furthermore, for any $\SpinC$-structure $\spinct$ on 
  $\widehat{N}$,
  \[
    \eval{c_1 \paren{\spinct \vert_{\widehat{M}_-}}}{\brac{R_{\widehat{M}_-}}} 
    = \eval{c_1 \paren{\spinct 
        \vert_{\widehat{M}_+}}}{\brac{R_{\widehat{M}_+}}}.
  \]
  Consequently, the desired result follows from \fullref{thm:main-hf-spinc}.
\end{proof}

\subsection{Involutive Heegaard Floer theory}
\label{ssec:hf-inv}

We now extend our work in \fullref{ssec:hf-surgery} to prove 
\fullref{thm:main-hf-inv}.  Recall that $\HFIh(Y)$ is defined as the homology 
of the mapping cone of $1 + \iota$, where $\iota$ is a chain homotopy 
equivalence on $\CFh (Y)$ coming from $\SpinC$-conjugation.  Since we are 
working over $\zeetwo$, we have that $\HFIh(Y)$ is in fact isomorphic to the 
homology of the mapping cone of $1 + \iota_* \colon \HFh(Y) \to \HFh(Y)$.
Unfortunately, the theory of cobordism maps is not fully developed in the 
theory, but we can still compare the involutive Heegaard Floer homologies under 
ribbon homology cobordisms.

\begin{proof}[Proof of \fullref{thm:main-hf-inv}]
  Fix a self-conjugate $\SpinC$-structure $\spinc_-$ on $\Ym$, which determines 
  a unique $\SpinC$-structure $\spinct$ on $W$ and a unique $\spinc_+$ on 
  $\Yp$.
  Then we have the commutative diagram
  \[
    \xymatrix{
      \HFh(\Ym, \spinc_-) \ar[r]^{\Fh_{W, \spinct}} \ar[d]^{1 + \iota_*} & 
      \HFh(\Yp, \spinc_+) \ar[r]^{\Fh_{-W, \spinct}} \ar[d]^{1 + \iota_*}& 
      \HFh(\Ym, \spinc_-) \ar[d]^{1 + \iota_*}\\
      \HFh(\Ym, \spinc_-) \ar[r]^{\Fh_{W, \spinct}} & \HFh(\Yp, \spinc_+) 
      \ar[r]^{\Fh_{-W, \spinct}} & \HFh(\Ym, \spinc_-).
    }
  \]
  The result now follows from \fullref{thm:main-hf-spinc}.
\end{proof}

\section{Some specific obstructions}
\label{sec:specific}

In this section, we derive some obstructions to ribbon homology cobordisms from 
our results on character varieties (\fullref{sec:character}) and Floer 
homologies (\fullref{sec:results_floer}).

\subsection{Ribbon cobordisms between Seifert fibered homology spheres}
\label{ssec:character-seifert}

First, we prove \fullref{thm:seifert}, a statement on ribbon $\Q$-homology 
cobordisms between Seifert fibered homology spheres.  Since 
\fullref{thm:seifert}~\eqref{it:casson} and \eqref{it:plumbing} will follow 
easily from instanton or Heegaard Floer homology, our main goal is to prove the 
following, which is a restatement of \fullref{thm:seifert}~\eqref{it:fibers}.  The complete proofs of
\fullref{thm:seifert} and \fullref{cor:montesinos} are given at the end of this subsection.

\begin{theorem}
  \label{thm:number-fibers}
  Suppose that there exists a ribbon $\Z$-homology cobordism from the Seifert 
  fibered homology sphere $\Sigma(a_1, \dotsc, a_n)$ to $\Sigma(a_1', \dotsc, 
  a_m')$.  Then the numbers of fibers satisfy $n \leq m$.
\end{theorem}

Our strategy will be to use \fullref{prop:zariski} with $G = \SUtwo$. We begin 
by mentioning a basic fact about $\SUtwo$-representations. Every representation 
$\rho \colon \pi \to \SUtwo$ is either trivial, Abelian, or irreducible, and 
$\dimR H^0 (\pi; \Ad_\rho)$ is respectively $3$, $1$, or $0$, according to this 
trichotomy.

We now review some useful facts from the work of Fintushel and Stern 
\cite{FintushelStern} on $\SUtwo$-representations for Seifert fibered homology 
spheres (see also the work of Boyer \cite{Boyer}).  To fix our notation, the 
Seifert fibered homology sphere $\Sigma (a_1, \dotsc, a_n)$ has base orbifold 
$S^2$ and presentation $(b; (a_1, b_1), \dotsc, (a_n, b_n))$, where we do not 
require that $0 < b_i < a_i$, but do require that $a_i$ and $b_i$ are 
relatively prime, and that
\[
  -b + \sum_{i=1}^n \frac{b_i}{a_i} = \frac{1}{a_1 \dotsm a_n}.
\]
Then the fundamental group of $\Sigma(a_1,\ldots,a_n)$ is given by 
\[
  \pi_1 (\Sigma (a_1, \dotsc, a_n)) \isom \grppre{x_1, \dotsc, x_n, h}{h \text{ 
      central}, \, x_i^{a_i} = h^{-b_i}, \, x_1 \dotsm x_n = h^{-b}}.
\]

\begin{theorem}[Fintushel and Stern 
  {\cite{FintushelStern}}]
  \label{thm:fintushel-stern}
  Suppose that $\rho \in \repvar (\Sigma (a_1, \dotsc, a_n))$ is irreducible.  
  Then
  \begin{enumerate}
    \item \label{it:fs1} \cite[Lemma~2.1]{FintushelStern} $\rho(h) = \pm 1$;
    \item \label{it:fs2} \cite[Lemma~2.2]{FintushelStern} $\rho(x_i) \neq \pm 1$ for at least 
      three values of $i \in \set{1, \dotsc, n}$; and
    \item \label{it:fs3} \cite[Proposition~2.5]{FintushelStern} $\dimR H^1 
      (\Sigma(a_1, \dotsc, a_n); \Ad_\rho) = 2 t - 6$, where $t$ is the number 
      of $x_i$'s such that $\rho (x_i) \neq \pm 1$.
  \end{enumerate}
\end{theorem}

We also recall their recipe for constructing conjugacy classes of irreducible 
$\SUtwo$-representations. For this, it will be useful to think of elements of 
$\SUtwo$ as unit quaternions.  After choosing the sign of $\rho (h)$, we may 
choose an integer $\ell_1$ and define $\rho (x_1) = e^{i \pi \ell_1 / a_1}$, as 
long as $0 \leq \ell_1 \leq a_1$ and $(-1)^{\ell_1} = (\rho (h))^{-b_1}$.  
Next, for each $q \in \set{2, \dotsc, n}$, we choose an integer $\ell_q$ with 
analogous constraints and consider $e^{i \pi \ell_q / a_q}$; we will eventually 
define $\rho (x_q)$ to be some element conjugate to this. (The choice of the 
integers $\ell_q$ is also subject to 
\fullref{thm:fintushel-stern}~\eqref{it:fs2}.) Note that once we choose $\rho 
(x_2), \dotsc, \rho (x_{n-1})$, they will determine $\rho (x_n)$ by the 
equation
\begin{equation}
  \label{eqn:fs-eq1}
  \rho (x_1) \dotsm \rho (x_n) = (\rho (h))^{-b};
\end{equation}
the difficulty, however, lies in ensuring that $\rho (x_n)$ is also conjugate 
to $e^{i \pi \ell_n / a_n}$ for some integer $\ell_n$ with analogous 
constraints. Plugging this last condition into \eqref{eqn:fs-eq1}, we see that 
$\rho (x_2), \dotsc, \rho (x_{n-1})$ must satisfy
\begin{equation}
  \label{eqn:fs-eq2}
  (\rho (x_1) \dotsm \rho (x_{n-1}))^{a_n} = (\rho (h))^{- b a_n + b_n} = \pm 
  1.
\end{equation}
To fulfill this condition, let $S_q$ denote the set of elements in $\SUtwo$ 
conjugate to $e^{i \pi \ell_q / a_q}$, for each $q \in \set{2, \dotsc, n - 1}$.  
If $e^{i \pi \ell_q / a_q} \neq \pm 1$, then $S_q$ is a copy of $S^2$.  In any 
case, consider the map $\phi \colon S_2 \cross \dotsb \cross S_{n-1} \to [0, 
\pi]$ given by
\[
  \phi (s_2, \dotsc, s_{n-1}) = \Arg (\rho(x_1) s_2 \cdots s_{n-1}),
\]
where $\Arg (z)$ is defined to be the value $\theta \in [0, \pi]$ such that $z$ 
is conjugate to $e^{i \theta}$. If $\pi \ell_n' / a_n$ is in the image of 
$\phi$, then, for each integer $\ell_n'$ such that $0 
\leq \ell_n' \leq a_n$ and $(-1)^{\ell_n'} = (\rho (h))^{-b a_n + b_n}$, there 
exists some choice of $\rho (x_2), \dotsc, \rho (x_{n-1})$ such that $\rho 
(x_1) \dotsm \rho (x_{n-1})$ is conjugate to $e^{i \pi \ell_n' / a_n}$ (and 
hence \eqref{eqn:fs-eq2} holds). This determines a well-defined representation $\rho$. Finally, note 
that since the Abelianization of $\pi_1 (\Sigma (a_1, \dotsc, a_n))$ is 
trivial, there are no non-trivial Abelian $\SUtwo$-representations, and every 
non-trivial representation is irreducible.

We now proceed towards the proof of \fullref{thm:number-fibers}. The main 
technical proposition we will prove is the following.  While this is well 
known, we include a direct proof for completeness.

\begin{proposition}
  \label{prop:seifert-zariski}
  Suppose that $Y$ is the Seifert fibered homology sphere $\Sigma(a_1, \dotsc, 
  a_n)$. Then there exists an irreducible $\rho \in \repvar (Y)$ such that $H^1 
  (Y; \Ad_\rho)$ has the maximal dimension possible, i.e.\ $2n - 6$.  
\end{proposition}

We now briefly describe our strategy to prove this proposition. By 
\fullref{thm:fintushel-stern}~\eqref{it:fs3}, we would like to show that $\pi_1 
(\Sigma (a_1, \dotsc, a_n))$ admits an irreducible $\SUtwo$-representation that 
does not send any $x_i$ to $\pm 1$. The idea is to reduce to the case where 
there are exactly three singular fibers; in other words, we will construct such a 
representation from an irreducible $\SUtwo$-representation of $\pi_1 (\Sigma 
(a_1, a_2, a_3 \dotsm a_n))$, by a pinching argument. A subtlety here is that 
for this pinching argument to work, we will require the representation of 
$\pi_1 (\Sigma (a_1, a_2, a_3 \dotsm a_n))$ to be of a certain form; to show 
that this exists, we will assume the primality of the $a_i$'s.  Thus, we begin 
by reducing to the case that all the $a_i$'s are prime.

\begin{lemma}
  \label{lem:seifert-prime}
  Let $r \in \Z_{\geq 0}$. Suppose that $\pi_1 (Y)$ admits an irreducible 
  representation $\rho$ such that $\dimR H^1 (Y; \Ad_{\rho}) = r$ for $Y \homeo 
  \Sigma (a_1, \dotsc, a_n)$.  Then the same holds for $Y \homeo \Sigma (a_1, 
  \dotsc, a_{n-1}, k a_n)$, where $k$ is relatively prime to $a_1, \dotsc, a_{n-1}$.
\end{lemma}

\begin{proof}
  Fix a presentation $(b; (a_1, b_1), \dotsc, (a_{n-1}, b_{n-1}), (k a_n, 
  b_n))$ for $\Sigma (a_1, \dotsc, a_{n-1}, k a_n)$; then
  \[
    \Sigma(a_1, \dotsc, a_n)  = S^2 (k b; (a_1, k b_1), \dotsc, (a_{n-1}, k 
    b_{n-1}), (a_n, b_n)).
  \]
  We denote by $x_i$ and $y_i$ the respective generators of $\pi_1 (\Sigma 
  (a_1, \dotsc, a_{n-1}, k a_n))$ and $\pi_1 (\Sigma (a_1, \dotsc, a_n))$ 
  associated to the singular fibers, but we abusively write $h$ for the central 
  generator in both groups.  Consider the homomorphism $\phi \colon \pi_1 
  (\Sigma (a_1, \dotsc, a_{n-1}, k a_n)) \to \pi_1 (\Sigma (a_1, \dotsc, a_n))$ 
  defined by
  \[
    \phi (h) = h^k, \quad \phi (x_i) = y_i \text{ for all } i \in \set{1, 
      \dotsc, n},
  \]
  which can be easily checked to be well-defined. (For completeness, we note 
  that $\phi$ is induced by the $k$-fold cover of $\Sigma (a_1, \dotsc, a_n)$ 
  branched over the singular fiber of order $a_n$.)

  Since $\rho \in \repvar (\Sigma (a_1, \dotsc, a_n))$ is irreducible, so is 
  $\phi^* \rho \in \repvar (\Sigma (a_1, \dotsc, a_{n-1}, k a_n))$, and the 
  number of $x_i$'s such that $\phi^* \rho (x_i) = \pm 1$ is exactly the same 
  as the number of $y_i$'s such that $\rho (y_i) = \pm 1$.  Therefore, by 
  \fullref{thm:fintushel-stern}~\eqref{it:fs3}, we conclude that $\dimR H^1 
  (\Sigma (a_1, \dotsc, a_{n-1}, k a_n); \Ad_{\phi^* \rho}) = \dimR H^1 (\Sigma 
  (a_1, \dotsc, a_n); \Ad_{\rho}) = r$.
\end{proof}

Next, we reduce to the case where there are exactly three singular fibers.  
Given $\Sigma(a_1, \dotsc, a_n)$, let $p = a_3 \dotsm a_n$. Denote the 
generators for $\pi_1 (\Sigma (a_1, \dotsc, a_n))$ corresponding to the 
singular fibers by $x_1, \dotsc, x_n$, and those for $\pi_1 (\Sigma (a_1, a_2, 
p))$ by $y_1$, $y_2$, and $z$ respectively; we continue to write $h$ for the 
central generator. Note that the $a_i$'s are not assumed to be prime in the 
following lemma; their primality will instead be used later.

\begin{lemma}
  \label{lem:seifert-pinch}
  Suppose that $n \geq 4$. Then there exists a surjective homomorphism
  \[
    f \colon \pi_1 (\Sigma (a_1, \dotsc, a_n)) \to \pi_1 (\Sigma (a_1, a_2, p))
  \]
  such that, for every irreducible $\rho \in \repvar (\Sigma (a_1, a_2, p))$ 
  where $\rho (z)$ is conjugate to $e^{i \pi \ell_z / p}$ with $\ell_z$ 
  relatively prime to $p$,
  we have that $f^* \rho \in \repvar (\Sigma (a_1, \dotsc, a_n))$ is 
  irreducible and $f^* \rho (x_i) \neq \pm 1$ for all $i$; in other words, 
  $\dimR H^1 (\Sigma (a_1, \dotsc, a_n); \Ad_{f^* \rho})$ is maximal.
\end{lemma}

Recall that there exists a degree-$1$ map from $\Sigma(a_1, \dotsc, a_n)$ to 
$\Sigma(a_1, a_2, a_3 \dotsm a_n)$ given by pinching along a suitable vertical 
torus in the Seifert fibration.  The homomorphism $f$ above is induced by this 
map.

\begin{proof}
  Fix a presentation $(b; (a_1, b_1), \dotsc, (a_n, b_n))$ for $\Sigma(a_1, 
  \dotsc, a_n)$, and let $q = \sum_{i=3}^n p b_i / a_i$. Then
  \[
    \Sigma (a_1, a_2, p) = S^2 (b; (a_1, b_1), (a_2, b_2), (p, q)).
  \]
  (Since the $a_i$'s are pairwise relatively prime, $p$ and $q$ are relatively 
  prime.) The two fundamental groups are
  \begin{gather*}
    \pi_1 (\Sigma (a_1, \dotsc, a_n)) \isom \grppre{x_1, \dotsc, x_n, h}{h 
      \text{ central}, \, x_i^{a_i} = h^{-b_i}, \, x_1 \dotsm x_n = h^{-b}},\\
    \pi_1 (\Sigma (a_1, a_2, p)) \isom \grppre{y_1, y_2, z, h}{h \text{ 
        central}, \, y_i^{a_i} = h^{-b_i}, \, z^p = h^{-q}, \, y_1 y_2 z = 
      h^{-b}}.
  \end{gather*}
  With these presentations, we now define $f$. Since the $a_i$'s are pairwise 
  relatively prime, for each $i \geq 3$, we may choose an integer $\eta_i$ such 
  that $\eta_i p / a_i \equiv 1$ mod $a_i$. Clearly, $\eta_i$ and $a_i$ are 
  relatively prime for each $i$. We define
  \[
    f (x_1) = y_1, \quad f (x_2) = y_2, \quad f (x_i) = z^{\alpha_i} 
    h^{\beta_i} \, \text{ for } i \geq 3, \quad f (h) = h,
  \]
  where $\alpha_i = \eta_i p / a_i$ and $\beta_i = (\eta_i q - b_i) / a_i$.  
  (Note that $f$ does not depend on the choice of $\eta_i$.) Observe that 
  $\sum_{j=3}^n \alpha_j \equiv 1$ mod $a_i$ for each $i \geq 3$, which implies 
  that $\sum_{j=3}^n \alpha_j \equiv 1$ mod $p$; using this fact, it is 
  straightforward to check that $f$ is a well-defined group homomorphism.

  We now claim that, for an irreducible $\rho \in \repvar (\Sigma (a_1, a_2, 
  p))$ satisfying the conditions in the lemma, we have $f^* \rho (x_i) \neq \pm 
  1$ for $i = 1, \dotsc, n$.  This is clear for $i = 1$ and $i = 2$.  For $i 
  \geq 3$, suppose that $f^* \rho (x_i) = \pm 1$ for some $i \geq 3$; then 
  $\rho (z)^{\eta_i p/a_i} = \pm 1$, implying that $e^{i \pi \ell_z \eta_i / 
    a_i} = \pm 1$, which is a contradiction since $\eta_i$ is relatively prime 
  to $a_i$.
\end{proof}

We now demonstrate the existence of an irreducible $\rho \in \repvar (\Sigma 
(a_1, a_2, p))$ satisfying the conditions of \fullref{lem:seifert-pinch}, in 
the case that the $a_i$'s are pairwise prime.

\begin{proposition}
  \label{prop:seifert-composite}
  Suppose that $n \geq 4$, and that $a_1 < \dotsb < a_n$ are positive prime 
  numbers. There exists an irreducible $\rho \in \repvar (\Sigma (a_1, a_2, 
  p))$ such that $\rho (z)$ is conjugate to $e^{i \pi \ell_z / p}$, where 
  $\ell_z$ is relatively prime to $p$.
\end{proposition}

\begin{proof}
  We continue to write $\Sigma (a_1, a_2, p) = S^2 (b; (a_1, b_1), (a_2, b_2), 
  (p, q))$, and use the same presentation for $\pi_1 (\Sigma (a_1, a_2, p))$ as 
  before.  First, we make some general observations.  Recall the construction 
  of irreducible $\SUtwo$-representations in the paragraph after 
  \fullref{thm:fintushel-stern}.  Observe that $\rho (z)$ is conjugate to $e^{i 
    \pi \ell_z / p}$ for some $\ell_z$ relatively prime to $p$ if and only if 
  $\rho (y_1) \rho (y_2)$ is conjugate to $e^{i \pi \ell_z' / p}$ for some 
  $\ell_z'$ relatively prime to $p$; thus, the goal is to show that after 
  picking appropriate values for $\rho (h)$, $\ell_1$, and $\ell_2$ satisfying 
  $(-1)^{\ell_j} = (\rho (h))^{-b_j}$ (for us to define $\rho (y_1) = e^{i \pi 
    \ell_1 / a_1}$ and decree $\rho (y_2)$ to be conjugate to $e^{i \pi \ell_2 
    / a_2}$), there exists an integer $\ell_z'$ such that
  \begin{enumerate}
    \item \label{it:pinch-prime} $\ell_z'$ is relatively prime to $p$;
    \item \label{it:pinch-sign} $(-1)^{\ell_z'} = (\rho (h))^{- b p + q}$; and
    \item \label{it:pinch-image} $\pi \ell_z' / p$ is in the image of $\phi 
      \colon S_2 \to [0, \pi]$.
  \end{enumerate}
  Since $\phi$ is continuous, to satisfy \eqref{it:pinch-image}, we simply need 
  to exhibit choices $s_2, s_2' \in S_2$ (i.e.\ elements $s_2$ and $s_2'$ that 
  are conjugate to $e^{i \pi \ell_2 / a_2}$) such that
  \begin{equation}
    \label{eq:squeezephi}
    \Arg (\rho (y_1) s_2) \leq \frac{\pi \ell_z'}{p} \leq \Arg (\rho (y_1) 
    s_2').
  \end{equation}
  Our strategy will be to find $s_2$, $s_2'$, and two values of $\ell_z'$ of 
  opposite parities satisfying \eqref{it:pinch-prime} and 
  \eqref{it:pinch-image}; then exactly one of them will satisfy 
  \eqref{it:pinch-sign}. Finally, by construction, $\rho$ is not trivial, and 
  thus is irreducible.

  First, we consider the special case that $a_1 = 2$.  In this case, we may 
  choose a presentation where $b_1 = 1$ and $b_2$ is odd (at the expense of 
  changing $b$). We take $\rho (h) = -1$, $\ell_1 = 1$, and $\ell_2 = 1$.  We 
  claim that the image of $\phi$ contains $\pi r_\pm / p$, where $\ell_z' = 
  r_\pm = (p \pm 1) / 2$.  Note that these two numbers are \emph{both} integers 
  relatively prime to $p$, and have opposite parities; thus, if $\pi r_{\pm} / 
  p$ are both in the image of $\phi$, the proof will be complete in this case.  
  To prove our claim, note that if we choose $s_2 = e^{- i \pi / a_2}$ and 
  $s_2' = e^{i \pi / a_2}$, which are both conjugate to $e^{i \pi / a_2}$, then 
  since $a_2 < p$, we have
  \[
    \Arg (e^{i \pi / 2} e^{- i \pi / a_2}) = \frac{\pi (a_2 - 2)}{2 a_2} \leq 
    \frac{\pi (p \pm 1)}{2 p} \leq \frac{\pi (a_2 + 2)}{2 a_2} = \Arg (e^{i \pi 
      / 2} e^{i \pi / a_2});
  \]
  in other words, \eqref{eq:squeezephi} is satisfied.

  We may now assume that all the $a_i$'s are odd.  Next, we consider the 
  special case that $a_1 = 3$ and $a_2 = 5$.  Again, we may choose a 
  presentation where $b_1$ and $b_2$ are both odd, and take $\rho(h) = -1$ and 
  $\ell_1 = \ell_2 = 1$. We again claim that the image of $\phi$ contains $\pi 
  r_{\pm} / p$, where $\ell_z' = r_{\pm} = (p \pm 1) / 2$. Indeed, if we choose 
  $s_2 = e^{- i \pi / 5}$ and $s_2' = e^{i \pi / 5}$, both of which are 
  conjugate to $e^{i \pi / 5}$, then since $p > 15$, we have
  \[
    \Arg (e^{i \pi/3} e^{-i \pi/5}) = \frac{2 \pi}{15} < \frac{\pi (p \pm 
      1)}{2p} < \frac{8 \pi}{15} = \Arg (e^{i \pi/3} e^{i \pi/5}),
  \]
  and \eqref{eq:squeezephi} is satisfied.

  By dispensing with the two cases above, we may assume that all the $a_i$'s 
  are odd, and further that $1 / a_1 + 1 / a_2 < 1 / 2$.  In this case, there 
  are several choices we could take for $\ell_j$ and $\rho (h)$; for 
  concreteness, we choose a presentation where $b_1$ and $b_2$ are both even, 
  and take $\ell_1 = \ell_2 = 2$ and $\rho (h) = 1$.  We choose $s_2 = e^{-i 
    \pi 2 / a_2}$ and $s_2' = e^{i \pi 2 / a_2}$, and compute the arguments to 
  be
  \[
    0 < \Arg (e^{i \pi 2 / a_1} e^{\pm i \pi 2 / a_2}) = 2 \pi 
    \paren{\frac{1}{a_1} \pm \frac{1}{a_2}} < \pi.
  \]
  As before, we now wish to find two values $r$ and $r'$ for $\ell_z'$, of 
  opposite parities, each relatively prime to $p$, such that 
  \eqref{eq:squeezephi} is satisfied, i.e.
  \[
    2 \paren{\frac{1}{a_1} - \frac{1}{a_2}} \leq \frac{r}{p}, \frac{r'}{p} \leq 
    2 \paren{\frac{1}{a_1} + \frac{1}{a_2}}.
  \]
  Let $I$ denote the interval governed by the above inequality.  Note that the 
  length of this interval is $4 / a_2$, and $4 / a_3 < 4 / a_2 < 1$.  
  Therefore, we may choose an integer $k$ with $0 < k < k + 2 < a_3$, such that 
  $[k / a_3, (k + 2) / a_3] \subset I$; we can rewrite this as
  \[
    \brac{\frac{k \frac{p}{a_3}}{p}, \frac{(k +2) \frac{p}{a_3}}{p}} \subset I.
  \]
  In fact, since $p / a_3 > a_3$, we have
  \[
    \frac{k \frac{p}{a_3} + a_3}{p}, \frac{k \frac{p}{a_3} + 2a_3}{p} \in I.
  \]
  Let $r = k p / a_3 + a_3$ and $r' = k p / a_3 + 2a_3$.  Note that $r$ and 
  $r'$ are between $0$ and $p$, and have opposite parities since $a_3$ is odd.  
  It remains to see that $r$ and $r'$ are relatively prime to $p$.  First, 
  since $0 < k < a_3$ and $a_3$ is prime, we see that $k$ is relatively prime 
  to $a_3$; thus, $r, r' \equiv k p / a_3 \not\equiv 0$ mod $a_3$.  At the same 
  time, for $i > 3$, we observe that $r \equiv a_3$ mod $a_i$ and $r' \equiv 2 
  a_3$ mod $a_i$; since the $a_i$'s are odd, prime, and greater than $a_3$, we 
  have that $a_3$ and $2 a_3$ are also relatively prime to $a_i$.  Putting this 
  together, we conclude that $r, r'$ are relatively prime to $p$, which 
  completes the proof.
\end{proof}

\begin{proof}[Proof of \fullref{prop:seifert-zariski}]
  Since the Casson invariant of any Seifert fibered homology sphere is never 
  zero, we have that the result trivially holds for $n = 3$. For $n \geq 4$, 
  the result follows from combining \fullref{lem:seifert-prime}, 
  \fullref{lem:seifert-pinch}, and \fullref{prop:seifert-composite}.
\end{proof}

\begin{proof}[Proof of \fullref{thm:number-fibers}]
  \fullref{thm:fintushel-stern}~\eqref{it:fs3} says that the Zariski tangent 
  space to the $\SUtwo$-character variety of $\Sigma(a_1, \dotsc, a_n)$ has 
  dimension less than or equal to $2 n - 6$.  By 
  \fullref{prop:seifert-zariski}, the equality is always realized at some 
  irreducible representation.  The result then follows from 
  \fullref{prop:zariski}.
\end{proof}

\begin{proof}[Proof of {\fullref{thm:seifert}}]
  \eqref{it:casson} This follows from \fullref{thm:main-i-o}, since $2 
  \abs{\casson (Y)} = \abs{\chi(\Io (Y))} = \dim \Io (Y)$ for a Seifert fibered 
  homology sphere $Y$ \cite{Saveliev}.

  \eqref{it:plumbing} The only Seifert fibered homology sphere with trivial 
  Casson invariant is $S^3$, which bounds both positive- and negative-definite 
  plumbings. Again by \cite{Saveliev}, $\Io (\Yp)$ is supported in one 
  $\zeetwo$-grading; this $\zeetwo$-grading determines the sign of $\casson 
  (\Yp)$ and hence the definiteness of the plumbing $\Yp$ bounds.  
  \fullref{thm:main-i-o} implies that $\Io (\Ym)$ is supported in the same 
  $\zeetwo$-grading.

  \eqref{it:fibers} This is \fullref{thm:number-fibers}.
\end{proof}

Note that the first two items above can also be proved using Heegaard Floer 
homology, by \cite[Theorem~1.3]{OzsSza03} and \cite[Corollary~1.4]{Plumbed}.

\begin{remark}
  \label{rmk:seifert}
  While the conclusions in \fullref{thm:seifert} seem strong, the authors do 
  not know of any ribbon $\Q$-homology cobordisms between two Seifert fibered 
  homology spheres distinct from $S^3$, or from a non--Seifert fibered space to 
  a Seifert fibered space.  For comparison, given any closed $3$-manifold 
  $\Ym$, one can always construct a ribbon $\Q$-homology cobordism from $\Ym$ 
  to a hyperbolic $3$-manifold, and one to a $3$-manifold with non-trivial JSJ 
  decomposition, as explained below.

  To construct a ribbon $\Q$-homology cobordism to a hyperbolic $3$-manifold, 
  first attach a $1$-handle to $\Ym$, so that the positive end is $\Ym \connsum 
  S^1 \cross S^2$. Let $K \subset \Ym \connsum S^1 \cross S^2$ be a hyperbolic 
  knot that is homotopic to $S^1 \cross \set{p}$; such a knot exists by a 
  result of Myers \cite[Theorem~1.1]{Myers}.  Attaching a $2$-handle along $K$ 
  with any framing will then yield a ribbon $\Z$-homology cobordism. By 
  Thurston's Hyperbolic Dehn Surgery Theorem, all but finitely many surgeries 
  along $K$ will yield a hyperbolic $3$-manifold.  In other words, for any 
  choice among all but finitely many surgery slopes, we have constructed a 
  ribbon $\Z$-homology cobordism from $\Ym$ to a hyperbolic $3$-manifold.
  
  To construct a ribbon $\Q$-homology cobordism to a $3$-manifold with 
  non-trivial JSJ decomposition, recall that the exterior of a hyperbolic knot 
  has incompressible boundary.  Again, we attach a $1$-handle to $\Ym$, and 
  choose a hyperbolic knot $K \subset \Ym \connsum S^1 \cross S^2$ that is 
  homotopic to $S^1 \cross \set{p}$.  This time, we will attach a $2$-handle 
  along a satellite knot with $K$ as the companion; for any framing, this will 
  give a ribbon $\Z$-homology cobordism as long as the pattern $P$ of the 
  satellite knot $P (K)$ has winding number $1$, since $K$ and $P (K)$ will be 
  homologous.  To carry this out, take a hyperbolic knot $P \subset S^1 \cross 
  D^2$ with winding number $1$, such as the Mazur pattern; note that Thurston's 
  Hyperbolic Dehn Surgery Theorem again implies that all but finitely many 
  surgeries along $P$ will give rise to a hyperbolic $3$-manifold. A 
  $3$-manifold obtained via surgery along the satellite $P (K) \subset \Ym 
  \connsum S^1 \cross S^2$ can be expressed as the union of $\Ym \connsum S^1 
  \cross S^2 \setminus K$ and a surgery along $P \subset S^1 \cross D^2$. In 
  other words, for any choice among all but finitely many surgery slopes, the 
  positive end of the ribbon $\Z$-homology cobordism we have constructed is 
  obtained by gluing two hyperbolic $3$-manifolds with incompressible torus 
  boundary, which is exactly a $3$-manifold with non-trivial JSJ decomposition.
\end{remark}

\fullref{thm:seifert}~\eqref{it:fibers} immediately implies a statement on 
Montesinos knots.

\begin{proof}[Proof of {\fullref{cor:montesinos}}]
  Let $C \colon \Km \to \Kp$ be a strongly homotopy-ribbon concordance from 
  $\Km$ to $\Kp$. The branched double cover of $C$ gives a ribbon $\Q$-homology 
  cobordism between Seifert fibered homology spheres $\Ypm$, where the number 
  of exceptional fibers in $\Ypm$ is precisely the number of rational tangles 
  in the Montesinos knot $\Kpm$ with denominator at most 2.
\end{proof}

\subsection{Ribbon homology cobordisms and \texorpdfstring{$L$}{L}-spaces}
\label{ssec:L-spaces}

We now utilize the $U$-action on $\HFm$ to derive two obstructions to ribbon 
homology cobordisms involving $L$-spaces.

Recall that the \emph{reduced Heegaard Floer homology} $\HFr$ is the 
$U$-torsion submodule of $\HFm$, and a $\Q$-homology sphere $Y$ is an 
\emph{$L$-space} if $\HFr (Y) = 0$.\footnote{Technically, $Y$ should be called 
  a \emph{$\zeetwo$--Heegaard $L$-space}. One could also define $L$-spaces with 
  other coefficients, or with instanton Floer homology.  However, we never 
  consider these concepts of $L$-spaces in the present article.}

\begin{corollary}
  \label{cor:lspace}
  Suppose that $\Ym$ and $\Yp$ are $\Q$-homology spheres, and that $\Yp$ is an 
  $L$-space while $\Ym$ is not. Then there does not exist a ribbon 
  $\zeetwo$-homology cobordism from $\Ym$ to $\Yp$.
\end{corollary}

Note that this applies whenever $\Ym$ is a toroidal $\Z$-homology sphere, since 
such a manifold is necessarily not an $L$-space \cite[Theorem~1.1]{Eftekhary} 
(see also \cite[Corollary~10]{HRW}).

\begin{proof}
  Suppose that $W \colon \Ym \to \Yp$ is a ribbon $\zeetwo$-homology cobordism; 
  then \fullref{thm:main-hf} implies that $\Fm_W \colon \HFm (\Ym) \to \HFm 
  (\Yp)$ is injective. Under this map, $U$-torsion elements must be mapped to 
  $U$-torsion elements; thus, we obtain an injection on $\HFr$ as well.
\end{proof}

\begin{corollary}
  \label{cor:hf-seifert}
  Suppose that $Y_1$ and $Y_2$ are $\Q$-homology spheres that are not 
  $L$-spaces.  Then there does not exist a ribbon $\Z/2$-homology cobordism 
  from $Y_1 \connsum Y_2$ to a Seifert fibered space.
\end{corollary}

\begin{proof}
  If $Y_1$ and $Y_2$ both have non-trivial $\HFr$, then in both 
  $\Z/2$-gradings, $\HFr (Y_1 \connsum Y_2)$ is not trivial. Indeed, by the 
  K\"unneth formula \cite[Theorem~1.5]{PropApp}, $\HFr (Y_1 \connsum Y_2)$ 
  contains a summand isomorphic to two copies of $\HFr (Y_1) \tensor \HFr 
  (Y_2)$, with one copy shifted in grading by $1$.  (One comes from the tensor 
  product and one from the $\Tor$ term.)
  Meanwhile, Seifert fibered spaces have $\HFr$ supported in a single 
  $\Z/2$-grading \cite[Corollary~1.4]{Plumbed}.
\end{proof}

\subsection{Ribbon homology cobordisms from connected sums}
\label{ssec:ribbon-cobs}

\fullref{cor:hf-seifert} concerns the existence of ribbon homology cobordisms 
from a connected sum. The following corollary also concerns connected sums, but 
is proved using $\repvarG$.

\begin{corollary}
  \label{cor:su2sums}
  Let $Y$ and $N$ be compact $3$-manifolds, and suppose that $N \not\homeo 
  S^3$.  Then there does not exist a ribbon $\Q$-homology cobordism from $Y 
  \connsum N$ to $Y$.
\end{corollary}

\begin{proof}[Proof of \fullref{cor:su2sums}]
  First, we fix some notation. Let $\pi$ be a group, and let $G$ be a compact, 
  connected Lie group.  Fix a presentation $\grppre{a_1, \dotsc, a_g}{w_1, 
    \dotsc, w_r}$ of $\pi$.  For each $\rho \in \repvarG(\pi)$, the words $w_i$ 
  give a smooth map $\Phi \colon G^g \to G^r$; denote by $\phi_\rho$ the 
  derivative of $\Phi$ at $(\rho (a_1), \dotsc, \rho (a_g))$, and we define 
  $\omega_G \colon \repvarG (\pi) \to \Z_{\geq 0}$ by
  \[
    \omega_G (\rho) = \dimR \ker (\phi_\rho).
  \]
  (The reader may wish to compare $\phi_\rho$ here with the map $\phi$ in 
  \eqref{eq:zariski-cochain}.)
  It is not difficult to check that $\omega_G$ is independent of the 
  presentation of $\pi$. We also define
  \[
    \omega_G (\pi) = \max_{\rho \in \repvarG (\pi)} \omega_G (\rho),
  \]
  and write $\omega_G (X)$ for $\omega_G (\pi_1 (X))$, for a path-connected 
  space $X$.

  Suppose that $\Ym$, $\Yp$, $W$, $\rhom$, $\rhop$, and $\rhoW$ are as in 
  \fullref{prop:character-embedding}.  By \fullref{prop:character-embedding-2}, 
  we have
  \[
    \omega_G (\rhom) \leq \omega_G (\rhoW) \leq \omega_G (\rhop).
  \]
  Indeed, since $\pi_1 (W)$ is obtained from $\pi_1 (\Yp)$ by adding relations, 
  the matrix for $\omega_G (\rhoW)$ contains that for $\omega_G (\rhop)$ as a 
  block, with additional rows; similarly, the matrix for $\phi_{\rhoW}$ 
  contains that for $\phi_{\rhom}$ as a block, with $m$ additional rows and 
  columns. Consequently, we see that
  \begin{equation}
    \label{eq:omega}
    \omega_G (\Ym) \leq \omega_G (W) \leq \omega_G(\Yp).
  \end{equation}
  
  Now suppose that there exists a ribbon $\Q$-homology cobordism $W \colon Y 
  \connsum N \to Y$. Homology considerations show that $N$ must be a 
  $\Z$-homology sphere (cf.\ \fullref{rmk:zhs} and \fullref{lem:gordon-h1}).  
  This means that $\pi_1 (N)$ cannot be solvable, since the Abelianization of a 
  solvable group is trivial. Thus, the residual finiteness of $3$-manifold 
  groups implies the existence of a non-trivial, finite quotient $H$ of $\pi_1 
  (N)$. Since $\pi_1 (N)$ is perfect, the quotient $H$ is also perfect. Take 
  any non-trivial, irreducible representation of $H$ in $\GL_n (\C)$; by Weyl's 
  trick, we may turn it into a non-trivial, irreducible, unitary representation 
  from $H$ to $\Un$; since $\mathrm{U} (1)$ is Abelian and $H$ is perfect, we 
  may assume $n \geq 2$. (Of course, the possible choices for $n$ depend on 
  $H$.)  Let $\eta \colon \pi_1 (N) \to \Un$ be the associated representation, 
  and choose $\rho \in \repvar_{\Un} (Y)$ that maximizes $\omega_{\Un}$.  
  Consider the free product representation $\rho \fprod \eta \colon \pi_1 (Y 
  \connsum N) \to \Un$.  It follows from the definitions that
  \[
    \omega_{\Un}(\rho \fprod \eta) = \omega_{\Un}(\rho) + \omega_{\Un}(\eta).  
  \]
  It is easy to see that since $\eta$ is irreducible, we have 
  $\omega_{\Un}(\eta) > 0$. We see that $\omega_{\Un}(Y \connsum N) > 
  \omega_{\Un}(Y)$, which contradicts \eqref{eq:omega}.
\end{proof}

Note that \fullref{cor:su2sums} can alternatively be proved if $N$ has 
non-trivial $\HFr$, by an application of the K\"unneth formula, as in the proof 
of \fullref{cor:hf-seifert}.  However, such an argument does not work for $N = 
\Sigma(2,3,5) \connsum (-\Sigma(2,3,5))$, since this is an $L$-space.  (In 
fact, in this case, $Y \connsum N$ is even $\Z$-homology cobordant to $Y$.) The 
same issue arises for framed instanton Floer homology $\Is$.  For $\Io$, it is 
difficult to study the instanton Floer homology of connected sums in general.

\fullref{cor:su2sums} can be viewed as obstructing ribbon homology cobordisms 
from a $3$-manifold with an essential sphere to one without. We now turn to 
proving \fullref{cor:composite-to-small}, which is a statement for knots with a 
similar flavor: It states that there are no strongly homotopy-ribbon 
concordances from a connected sum $K_1 \connsum K_2$ to a knot without a 
closed, non--boundary-parallel, incompressible surface in its exterior.

\begin{proof}[Proof of {\fullref{cor:composite-to-small}}]
  Write $\Km \homeo K_1 \connsum K_2$. For a knot $K \subset S^3$, denote a 
  fixed meridian by $\mu$, and a fixed longitude by $\lambda$. Also, write 
  $\repvar (K)$ for $\repvar_{\SUtwo} (S^3 \setminus K)$ and $\charvar (K)$ for 
  $\charvar_{\SUtwo} (S^3 \setminus K)$.  First, recall from the proof of 
  \cite[Proposition~12]{Klassen} that if $\rho_1 \in \repvar (K_1)$ and $\rho_2 
  \in \repvar (K_2)$ satisfy $\rho_1 (\mu) = \rho_2 (\mu)$, then we can 
  amalgamate $\rho_1$ and $\rho_2$ into a $1$-parameter family of 
  representations $\rho \in \repvar (K_1 \connsum K_2)$, by fixing $\rho$ on 
  one summand and conjugating it on the other by the $\Uone$-stabilizer of the 
  peripheral subgroup.

  Now by \cite[Theorem~4.1]{SivekZentner}, for a non-trivial knot $K \subset 
  S^3$, at least one of the following holds:
  \begin{enumerate}
    \item \label{it:sz1} There is a smooth $1$-parameter family of irreducible 
      $\rho_t \in \repvar (K)$, such that $\rho_t (\mu) = \diag (i, -i)$ and 
      $\rho_t (\lambda) = \diag (e^{it}, e^{-it})$, for $t$ in some interval 
      $(t_0, t_2)$; or
    \item \label{it:sz2} There is a smooth $1$-parameter family of irreducible 
      $\rho_s \in \repvar (K)$, such that $\rho_s (\mu) = \diag (e^{is}, 
      e^{-is})$, where $s \in [\pi/2, \pi/2 + \epsilon)$, for some $\epsilon > 
      0$.
  \end{enumerate}
  Note that in either case, there is a representation $\rho' \in \repvar (K)$ 
  with $\rho' (\mu) = \diag (i, -i)$. (See also 
  \cite[Corollary~7.17]{KM:sutured}.) If \eqref{it:sz2} holds for both $K_1$ 
  and $K_2$, then we may amalgamate representations with the same eigenvalue on 
  the meridian to get a $2$-parameter family $\rho \in \repvar (K_1 \connsum 
  K_2)$.  If \eqref{it:sz1} holds for $K_1$, then we may amalgamate this 
  $1$-parameter family with $\rho' \in \repvar (K_2)$, again to obtain a 
  $2$-parameter family of representations $\rho \in \repvar (K_1 \connsum 
  K_2)$.

  In any case, note that these representations $\rho$ can in fact be 
  conjugated. Thus, we have shown that $\repvar (\Km)$ has an open submanifold 
  of dimension at least $5$ on which the conjugation action by $\SOthree$ is 
  free.  Suppose that there is a strongly homotopy-ribbon concordance $C \colon 
  \Km \to \Kp$; \fullref{prop:character-embedding-2} implies that $\repvar 
  (\Kp)$ has an open submanifold of dimension at least $5$ on which $\SOthree$ 
  acts freely. (Although the representation variety is not a smooth manifold in 
  general, it is a real algebraic variety and hence can be stratified into the 
  union of finitely many smooth manifolds \cite{Whitney}; it is easy to see 
  that the maps $\iota_-^*$ and $\iota_+^*$ in 
  \fullref{prop:character-embedding-2} induce smooth maps on an open subset of 
  each stratum.) Thus, $\charvar (\Kp)$ must have a component of dimension at 
  least $2$.  By \cite[Proposition 15]{Klassen}, this implies that $\Kp$ is not 
  small, which is a contradiction.
\end{proof}

\begin{remark}
  \label{rmk:eliashberg}
  Eliashberg \cite{Eliashberg} shows that a Stein filling of a connected sum is 
  a boundary sum of Stein fillings.  It is interesting to compare this with the 
  two results above.
\end{remark}

\section{Surgery obstructions}
\label{sec:surgery}

In this section, we give some applications of the work above on ribbon homology 
cobordisms to reducible Dehn surgery problems.  We first explain the main idea 
in this section. Let $Y$ be a $\Q$-homology sphere, and $L$ a null-homologous 
link of $\ell$ components in $Y$. Denote by $Y_0 (L)$ the result of $0$-surgery 
along each component of $L$. Suppose that $Y_0 (L) \homeo N \connsum \ell (S^1 
\cross S^2)$; in this case, we may deduce facts about $N$ or $L$ with the 
following construction.

Consider the cobordism $W \colon Y \to N$ obtained by attaching a $0$-framed 
$2$-handle along each of the components of $L$, and then a $3$-handle along 
some $\set{p} \cross S^2$ in each of the $S^1 \cross S^2$ summands of $Y_0 
(L)$. Flipping $W$ upside down and reversing its orientation, we obtain a 
cobordism $-W \colon N \to Y$.

\begin{lemma}
  \label{lem:surgery-ribbon}
  Suppose that $Y$ is a $\Q$-homology sphere, $L$ is a null-homologous link of 
  $\ell$ components in $Y$, and $Y_0 (L) \homeo N \connsum \ell (S^1 \cross 
  S^2)$. Then the cobordism $-W \colon N \to Y$ constructed above is a ribbon 
  $\Z$-homology cobordism.
\end{lemma}

\begin{proof}
  On the one hand, since $L \subset Y$ is null-homologous, we have $H_1 (Y_0 
  (L)) \isom H_1 (Y) \dirsum (\Z^\ell / M)$, where $M$ is given by the linking 
  matrix. On the other hand, we have $H_1 (Y_0 (L)) \isom H_1 (N) \dirsum 
  \Z^\ell$.  Since $b_1 (Y) = 0$, rank considerations imply that $M$ is 
  trivial.  (In particular, the linking matrix of $L$ must be identically 
  zero.)  Thus, we have $H_1 (Y) \isom H_1 (N)$.

  Now the cobordism $-W \colon N \to Y$ consists of the same number $\ell$ of 
  $1$- and $2$-handles and so is ribbon. It is a $\Q$-homology cobordism if and 
  only if the attaching circles of the $2$-handles are linearly independent in 
  $H_1 (N \connsum \ell (S^1 \cross S^2)) / H_1 (N)$, which is obviously true 
  since the $3$-manifold resulting from the $2$-handle attachments is $Y$, 
  which has $b_1 (Y) = 0$.  By \fullref{lem:same-h1}, $-W$ is in fact a ribbon 
  $\Z$-homology cobordism.
\end{proof}

\subsection{Null-homotopic links and reducing spheres}
\label{ssec:surgery-null-homotopic}

In this subsection, we focus on proving \fullref{thm:surgery-general}, which we 
recall asserts that when $0$-surgery on an $\ell$-component null-homotopic link 
in an irreducible $\Q$-homology sphere $Y$ produces $N \connsum \ell (S^1 
\cross S^2)$, then $N$ is orientation-preserving homeomorphic to $Y$.  Recall 
that a closed $3$-manifold $Y$ is aspherical if it is irreducible and has 
infinite fundamental group.  

We begin by relating the fundamental groups of $Y$ and $N$:

\begin{proposition}
  \label{prop:surgery-fundamental}
  Suppose that $Y$ is an irreducible $\Q$-homology sphere, $L$ is a 
  null-homotopic link of $\ell$ components in $Y$, and $Y_0 (L) \homeo N 
  \connsum \ell (S^1 \cross S^2)$. Then there is an orientation-preserving 
  degree-$1$ map from $N$ to $Y$ that induces isomorphisms on $\pi_1$.  
  Moreover, the inclusions of $Y$ and $N$ into the cobordism $W \colon Y \to N$ 
  constructed above also induce isomorphisms on $\pi_1$.
\end{proposition}

\begin{proof}
  Consider the $\Z$-homology cobordism $W \colon Y \to N$ constructed above, 
  and decompose $W$ into two cobordisms $W_2 \colon Y \to N \connsum \ell (S^1 
  \cross S^2)$ and $W_3 \colon N \connsum \ell (S^1 \cross S^2) \to N$, 
  corresponding to the attachment of $2$- and $3$-handles respectively.  We 
  claim that there is a retraction $\rho \colon W \to Y$.

  Indeed, first observe that since $L$ is null-homotopic, $W_2$ is homotopy 
  equivalent to $Y \vee \ell S^2$, which retracts onto $Y$.  More precisely, 
  there is a retraction $\rho_2 \colon W_2 \to Y$ given by deformation 
  retracting the $2$-handles to their cores, homotoping the attaching curves of 
  the $2$-handles to a point, and collapsing the resulting $S^2$ summands.
  Next, we see that $\rho_2$ extends to $W_3$.  Indeed, to prove that $\rho_2$ 
  extends over the $3$-handles, it suffices to see that for each $S^1 \cross 
  S^2$ summand in $N \connsum \ell (S^1 \cross S^2)$, the image $\rho_2 
  (\set{p} \cross S^2) \subset Y$ is null-homotopic. (Recall that the 
  $3$-handles are attached along the $\ell$ copies of $\set{p} \cross S^2$.)  
  Since $Y$ is irreducible, the Sphere Theorem implies that $\pi_2 (Y) = 0$, 
  and $\rho_2$ extends to a retraction $\rho \colon W \to Y$.

  We now claim that pre-composing $\rho$ with the inclusion $\iota_N \colon N 
  \to W$ results in a map $\rho \comp \iota_N \colon N \to Y$ that induces an 
  isomorphism on $\pi_1$. First, we check that the induced map is injective.  
  Indeed, by \fullref{lem:surgery-ribbon}, $-W \colon N \to Y$ is a ribbon 
  $\Z$-homology cobordism, which implies that $(\iota_N)_* \colon \pi_1 (N) \to 
  \pi_1 (W)$ is injective by \fullref{thm:gordon}~\eqref{it:gordon-inj}.  
  Turning to $\rho_* \colon \pi_1 (W) \to \pi_1 (Y)$, note that since $\rho$ is 
  a retraction, we have $\rho \comp \iota_Y = \id_Y$.  This implies that 
  $(\iota_Y)_* \colon \pi_1 (Y) \to \pi_1 (W)$ is injective; at the same time, 
  \fullref{thm:gordon}~\eqref{it:gordon-surj} states that $(\iota_Y)_*$ is 
  surjective. Thus, $(\iota_Y)_*$, and hence $\rho_*$, is an isomorphism. This 
  shows that $(\rho \comp \iota_N)_*$ is injective.

  Next, observe that $H_3 (W) \isom H_3 (Y) \isom \Z$, which implies that 
  $\rho_* \colon H_3 (W) \to H_3 (Y)$ is an isomorphism, since $\rho$ is a 
  retraction.  In fact, since $\iota_N \colon N \to W$ also induces an 
  isomorphism on $H_3$, we see that $(\rho \comp \iota_N)_* \colon H_3 (N) \to 
  H_3 (Y)$ sends $[N]$ to $[Y]$ (and not $-[Y]$), which means that $\rho \comp 
  \iota_N \colon N \to Y$ is an orientation-preserving degree-$1$ map. 
	
  We claim that such a map must induce a surjection $(\rho \comp \iota_N)_* 
  \colon \pi_1 (N) \to \pi_1 (Y)$, for otherwise we could factor the map 
  through a non-trivial cover $\cover{Y}$ of $Y$ corresponding to $(\rho \comp 
  \iota_N)_* (\pi_1 (N))$, showing that its degree is not $1$.  We conclude 
  that $(\rho \comp \iota_N)_*$ is an isomorphism. This gives the first claim.  
  For the second claim, since $\rho_*$ is an isomorphism, we see that 
  $(\iota_N)_*$ is also an isomorphism. Since we have already proved that 
  $(\iota_Y)_*$ is an isomorphism, this concludes the proof.
\end{proof}
  
To deal with the case that $Y$ is a spherical manifold, we will need one 
additional technical lemma.  In what follows, we will fix a basepoint $p_Y \in 
Y$ and a basepoint $p_N \in N$. We will say that
$\ccover{Y}$ is an \emph{unoriented (resp.\ oriented) $\cover{Y}$-cover of $Y$} 
if $\ccover{Y}$ corresponds to a concrete subgroup of $\pi_1 (Y, p_Y)$ and 
$\ccover{Y}$ is unoriented (resp.\ orientation-preserving) homeomorphic to 
$\cover{Y}$. Here, we do \emph{not} consider subgroups up to conjugacy.
Note that, since $Y$ is spherical, $\pi_1 (Y, p_Y)$ is finite, and so all 
covers we consider are finite.

\begin{lemma}
  \label{lem:surgery-covers}
  Suppose that $\cover{Y}$ does not admit an orientation-reversing 
  homeomorphism, and that both the hypotheses and conclusions of 
  \fullref{thm:surgery-general} hold for $\cover{Y}$.  Suppose that $Y$ also 
  satisfies the hypotheses of \fullref{thm:surgery-general}. If $Y$ is 
  unoriented homeomorphic to $N$ and has an odd number $n$ of unoriented 
  $\cover{Y}$-covers, then $Y$ satisfies the conclusions of 
  \fullref{thm:surgery-general}.
\end{lemma}

\begin{proof}
  By assumption, we know that $Y$ and $N$ are homeomorphic.  We assume for 
  contradiction that they are not orientation-preserving homeomorphic, and in 
  particular that $Y \homeo -N$.  

  Suppose that $Y_0(L) \homeo N \connsum \ell (S^1 \cross S^2)$ and consider 
  the (non-ribbon) $\Z$-homology cobordism $W \colon Y \to N$ constructed in 
  the paragraph preceding \fullref{lem:surgery-ribbon}.  Since we are working 
  with covers, we will be a bit pedantic with basepoints for the cautious 
  reader.  Pick a path $\gamma$ in $W$ that starts at $p_Y \in Y$ and ends at 
  $p_N \in N$; this gives rise to a change-of-basepoint isomorphism 
  $\Phi_{\gamma} \colon \pi_1 (W, p_Y) \to \pi_1 (W, p_N)$.  By 
  \fullref{prop:surgery-fundamental}, the inclusions $(\iota_Y)_* \colon \pi_1 
  (Y, p_Y) \to \pi_1 (W, p_Y)$ and $(\iota_N)_*: \pi_1(N, p_N) \to \pi_1(W, 
  p_N)$ are isomorphisms. (In this pedantic language, the map on $\pi_1$ 
  induced by $\rho$ in \fullref{prop:surgery-fundamental} is $\rho_* \colon 
  \pi_1 (W, p_Y) \to \pi_1 (Y, p_Y)$, and the isomorphism $(\rho \comp 
  \iota_N)_*$ is really $\rho_* \comp \Phi_\gamma^{-1} \comp (\iota_N)_*$.)

  Choose an oriented $\cover{Y}$-cover $\ccover{Y}$ corresponding to a subgroup 
  $H \subgrp \pi_1 (Y, p_Y)$ of index $[\pi_1 (Y, p_Y) : H] = h$.
  Consider $(\iota_Y)_* (H) \subgrp \pi_1 (W, p_Y)$ and its associated oriented 
  cover $\ccover{W}$ of $W$.  This is a cobordism whose incoming end is 
  $\ccover{Y}$, and whose outgoing end $\ccover{N}$ is the oriented cover of 
  $N$ corresponding to $(\iota_N)_*^{-1} \comp \Phi_{\gamma} \comp 
  (\iota_{Y})_* (H) \subgrp \pi_1 (N, p_N)$.  Note that $\ccover{N}$ is path 
  connected, because $N$ is path connected and $\ccover{N}$ corresponds to a 
  concrete subgroup of $\pi_1 (N, p_N)$.
  Since $(\iota_Y)_*$ and $(\iota_N)_*$ are isomorphisms, each oriented 
  $\cover{Y}$-cover $\ccover{Y}$ of $Y$ produces a distinct oriented cover 
  $\ccover{N}$ of $N$.

  Since $W$ is built out of the same number $\ell$ of $2$- and $3$-handles, we 
  see that $\ccover{W}$ is built by attaching the same number $h \ell$ of $2$- 
  and $3$-handles to $\ccover{Y}$.  (To see this, simply pull back a Morse 
  function on $W$ using the covering map.)  Since the attaching curves for the 
  $2$-handles in $W$ are null-homotopic, the attaching curves for the 
  $2$-handles in $\ccover{W}$ form a null-homotopic link $\ccover{L} \subset 
  \ccover{Y}$. Note also that $\ccover{N}$ does not contain any non-separating 
  2-spheres, since $b_1 (\ccover{N}) = 0$. Now, since $\ccover{N}$ is 
  connected, the result of the surgery along $\ccover{L}$ in $\ccover{Y}$ must 
  be of the form $\ccover{N} \connsum h \ell (S^1 \cross S^2)$, since, from the 
  upside-down perspective, it is also the result of attaching $1$-handles along 
  $h \ell$ copies of $S^0 \cross D^3$. For homology reasons, we see that the 
  surgery coefficients for $\ccover{L} \subset \ccover{Y}$ must be identically 
  $0$.  By our assumption, \fullref{thm:surgery-general} holds for $\cover{Y}$, 
  and so we see that $\ccover{N}$ is orientation-preserving homeomorphic to 
  $\ccover{Y}$, and hence to $\cover{Y}$.

  Note that a simple orientation-reversal argument shows that 
  \fullref{thm:surgery-general} also holds for $-\cover{Y}$, and so we may also 
  repeat the above arguments for oriented ($-\cover{Y}$)-covers, where the 
  corresponding covers of $N$ are then orientation-preserving homeomorphic to 
  $-\cover{Y}$.  

  Suppose that $Y$ has $n_+$ (resp.\ $n_-$) oriented $\cover{Y}$-covers (resp.\ 
  ($-\cover{Y}$)-covers); then $n = n_+ + n_-$.  Then $N \homeo -Y$ has $n_+$ 
  (resp.\ $n_-$) oriented ($-\cover{Y}$)-covers (resp.\ $\cover{Y}$-covers).  
  However, by the above argument, each oriented $(\pm \cover{Y})$-cover of $Y$ 
  induces a distinct oriented $(\pm \cover{Y})$-cover of $N$, which implies 
  that $n_+ = n_-$.  This contradicts the fact that $n$ is odd.
\end{proof}  

With this, we are ready to prove \fullref{thm:surgery-general}.  We use the 
notation as above.

\begin{proof}[Proof of \fullref{thm:surgery-general}]
  First, suppose that $Y$ is aspherical.  Since $Y$ is 
  irreducible and $\pi_1$ detects irreducibility \cite{Stallings}, it follows 
  from \fullref{prop:surgery-fundamental} that $N$ is also irreducible.  
  \fullref{prop:surgery-fundamental} also states that there is an 
  orientation-preserving, degree-$1$ map from $N$ to $Y$ that induces an 
  isomorphism on $\pi_1$.  Asphericity implies that this map is an 
  orientation-preserving homotopy equivalence by Whitehead's Theorem.  It is 
  thus homotopic to a homeomorphism by the Borel Conjecture in dimension $3$; 
  see, for example, \cite[Theorem~0.7]{BorelConj}. Note that this homeomorphism 
  must also preserve orientation, since this property is preserved under 
  homotopy.  This concludes the proof when $Y$ is an aspherical $\Q$-homology 
  sphere.

  Therefore, we may assume that $\pi_1 (Y)$ is finite, or equivalently, that 
  $Y$ and $N$ have spherical geometry.
  We first dispense with the case that $Y$ and $N$ are lens spaces.  Recall 
  from \fullref{lem:surgery-ribbon} that $N$ and $Y$ are $\Z$-homology 
  cobordant.  Two $\Z$-homology cobordant lens spaces are 
  orientation-preserving homeomorphic; see, for example, the discussion in 
  \cite{DoigWehrli}. This concludes the proof when $Y$ is a lens space.
    
  It remains to consider the case that $Y$ is spherical but not a lens space. 
  Recall that two spherical $3$-manifolds with isomorphic fundamental groups 
  are (possibly orientation-reversing) homeomorphic unless they are lens 
  spaces; therefore, \fullref{prop:surgery-fundamental} implies that $N$ and 
  $Y$ are homeomorphic.  Recall also that a non--lens space spherical 
  $3$-manifold $Y$ has $\pi_1$ isomorphic to a central extension of a 
  polyhedral group, and in particular, $\card{\pi_1(Y)} = 2^k m$, where $k \geq 
  2$ and $m$ is odd;
  see \cite[Section~1.7]{AFW} and \cite[Section~6.2]{Orlik}. In the following, 
  we will provide a proof, first for the case that $m = 1$ by inducting on $k$, 
  and then generalizing to the case that $m \geq 3$.

  Before we proceed, we collect three additional observations here.  First, if 
  the fundamental group of a lens space has order $2^k$ with $k \geq 2$, then 
  it does not admit an orientation-reversing homeomorphism, since $-1$ is not a 
  square mod $2^k$.  Second, a non--lens space spherical manifold $Y$ also does 
  not admit an orientation-reversing homeomorphism. Indeed, by considering an 
  order-$4$ subgroup of $\pi_1 (Y)$, we see that such a manifold has an 
  unoriented $L (4, 1)$-cover. (Recall that there are no $3$-manifolds $Y$ with 
  $\pi_1 (Y) \isom \zeetwo \dirsum \zeetwo$, and $L (4, 3) \homeo - L(4, 1)$.) 
  An orientation-reversing homeomorphism of $Y$ would induce an 
  orientation-reversing homeomorphism of $L(4, 1)$, which is impossible. Third, 
  the number of index-$2$ subgroups of any finite group $\pi$ is odd.  Indeed, 
  the index-$2$ subgroups of $\pi$ correspond to non-trivial homomorphisms to 
  $\zeetwo$, and $\Hom (\pi, \zeetwo)$ is a vector space over $\zeetwo$, which 
  has even cardinality.

  We begin by showing that the theorem holds in the case that $m = 1$.  As 
  mentioned, we will induct on $k$; to help illustrate the idea of the proof, 
  we will give a more detailed description for small values of $k$. If $k = 2$, 
  then $Y$ is a lens space $L (4, q)$.  If $k = 3$, then $\pi_1 (Y)$ has an odd 
  number of index-$2$ subgroups; in other words, $Y$ has an odd number of 
  unoriented $L (4, 1)$-covers. We may thus apply \fullref{lem:surgery-covers} 
  with $\cover{Y} = L (4, 1)$, and conclude that the theorem holds for 
  spherical manifolds with order $8$.

  Next, if $k = 4$, again $\pi_1 (Y)$ has an odd number of index-$2$ subgroups.  
  This means that the total number of unoriented covers of $Y$ (of possibly 
  distinct unoriented homeomorphism types) whose $\pi_1$ has order $8$ is odd.
  Hence, there must be an unoriented homeomorphism type $\cover{Y}$ of 
  spherical manifolds with $\card{\pi_1 (\cover{Y})} = 8$, such that $Y$ has an 
  odd number of ($2$-sheeted) unoriented $\cover{Y}$-covers. Again we may apply 
  \fullref{lem:surgery-covers} with this choice of $\cover{Y}$, and the proof 
  is complete for those spherical manifolds $Y$ with $\card{\pi_1 (Y)} = 16$.  
  (Here, $\cover{Y}$ may be a lens space $L (8, q)$, or the unique type-$\dih$ 
  manifold with $\pi_1$ isomorphic to the quaternion group $Q_8$; in either 
  case, the actual homeomorphism type of $\cover{Y}$ is irrelevant, and we are 
  simply relying on the fact that there is no orientation-reversing 
  homeomorphism of $\cover{Y}$, in order to invoke 
  \fullref{lem:surgery-covers}.)  Similarly, we may now continue this induction 
  on $k$ to complete the proof in the case that $\card{\pi_1 (Y)} = 2^k$ with 
  $k \geq 2$.  

  Finally, we consider the case that $Y$ is a non--lens space spherical 
  manifold with $\card{\pi_1(Y)} = 2^k m$, where $k \geq 2$ and $m \geq 3$ is 
  odd.  In this case, the Sylow $2$-subgroups of $\pi_1 (Y)$ have order $2^k$, 
  and there are an odd number of them by the Third Sylow Theorem.  Hence, there 
  exists an unoriented homeomorphism type $\cover{Y}$ of spherical manifolds 
  with $\card{\pi_1 (\cover{Y})} = 2^k$, such that $Y$ has an odd number of 
  finite, unoriented $\cover{Y}$-covers.  We may now apply a similar argument 
  using \fullref{lem:surgery-covers}.  This completes the proof.
\end{proof}

\bibliographystyle{mwamsalphack}
\bibliography{references}

\end{document}